\documentclass[a4paper,10pt,reqno]{amsart}

\usepackage{ifpdf}
\ifpdf 
    \usepackage[pdftex]{graphicx}   
    \usepackage[pdftex,     
            plainpages=false,   
            breaklinks=true,    
            colorlinks=true,
            pdftitle=My Document
            pdfauthor=My Good Self
           ]{hyperref} 
\else 
    \usepackage{graphicx}       
    \usepackage{hyperref}       
\fi 

\usepackage{amsfonts,amsmath}	
\usepackage{amssymb}
\usepackage{verbatim}
\usepackage{amsopn}
\usepackage[english]{babel}
\usepackage{amsthm}
\usepackage{enumerate}
\usepackage{mathrsfs}	
\usepackage{mathtools}
\usepackage{color}
\usepackage{bbm}
\usepackage{subfig}

\usepackage{geometry}
\title[Regularity estimates for scalar conservation laws]{Regularity estimates for scalar conservation laws in one space dimension}
\author{Elio Marconi}
\address{E. Marconi: S.I.S.S.A., via Bonomea 265, 34136 Trieste, Italy}
\email{elio.marconi@sissa.it}

\newcounter{assu}

\theoremstyle{definition} \newtheorem{definition}{Definition}[section]
\theoremstyle{definition} \newtheorem{remark}[definition]{Remark}
\theoremstyle{plain} \newtheorem{lemma}[definition]{Lemma}
\theoremstyle{plain} \newtheorem{proposition}[definition]{Proposition}
\theoremstyle{plain} \newtheorem{theorem}[definition]{Theorem}
\theoremstyle{plain} \newtheorem{corollary}[definition]{Corollary}
\theoremstyle{definition} 
\theoremstyle{plain} 

\newtheorem{theorem01}[assu]{Theorem}
\newtheorem{theorem02}[assu]{Theorem}
\newtheorem{theorem03}[assu]{Theorem}

\DeclareMathOperator{\conv}{conv}

\DeclareMathOperator{\graph}{Graph}
\DeclareMathOperator{\dist}{dist}
\newcommand{\R}{\mathbb{R}}

\newcommand{\N}{\mathbb{N}}
\newcommand{\Z}{\mathbb{Z}}
\newcommand{\TV}{\text{\rm TV}}
\newcommand{\PTV}{\Phi\text{\rm -TV}}

\DeclareMathOperator{\BV}{BV}
\DeclareMathOperator{\SBV}{SBV}

\newcommand{\Id}{\mathbbm{I}}
\newcommand{\Int}{\mathrm{Int}}

\newcommand{\e}{\varepsilon}

\newcommand{\loc}{\mathrm{loc}}
\newcommand{\di}{\mathfrak{d}}
\newcommand{\supp}{\mathrm{supp}\,}

\newcommand{\sci}{\mathrm{sc^-}}

\newcommand{\U}{\mathtt{u}}
\newcommand{\T}{\mathtt{T}}
\newcommand{\s}{\mathtt{S}}
\newcommand{\X}{\mathtt{X}}
\newcommand{\el}{\mathbf{l}}

\numberwithin{equation}{section} 

\begin{document}

\maketitle
\begin{abstract}
In this paper we deal with the regularizing effect that, in a scalar conservation laws in one space dimension, 
the nonlinearity of the flux function $f$ has on the entropy solution.
More precisely, if the set $\{w:f''(w)\ne 0\}$ is dense, the regularity of the solution can be expressed in terms of 
$\BV^\Phi$ spaces, where $\Phi$ depends on the nonlinearity of $f$.
If moreover the set $\{w:f''(w)=0\}$ is finite, under the additional polynomial degeneracy condition at the inflection points, 
we prove that $f'\circ u(t)\in \BV_{\loc}(\R)$ for every $t>0$ and that this can be improved to $\SBV_{\loc}(\R)$ regularity except an at most
countable set of singular times.
Finally we present some examples that shows the sharpness of these results and counterexamples to related questions, namely
regularity in the kinetic formulation and a property of the fractional BV spaces.
\end{abstract}



\section{Introduction}
We consider the scalar conservation law in one space dimension:
\begin{equation}\label{E_cl}
\begin{cases}
u_t+f(u)_x=0 & \mbox{in }\R^+\times \R, \\
u(0,\cdot)=u_0(\cdot), &
\end{cases}
\end{equation}
where the flux $f:\R\to \R$ is smooth and the function $u:\R^+_t\times \R_x\rightarrow \R$ is the spatial density of the conserved quantity.
In the classical setting, the problem is well-posed only locally in time, therefore we consider solutions in the sense of distributions, 
which are however not unique. 
The well-posedness is finally obtained requiring some admissibility condition: more precisely we say that a bounded distributional solution $u$ is 
an entropy solution if for every convex entropy $\eta:\R\to \R$, it holds in the sense of distributions
\begin{equation*}
\eta(u)_t+q(u)_x\le 0,
\end{equation*}
where $q'(u)=f'(u)\eta'(u)$ is the entropy flux.
A celebrated theorem of Kruzkov \cite{Kruzhkov_contraction} establishes existence, uniqueness and continuous dependence with respect to $L^1_\loc$ topology in the
setting of bounded entropy solutions, also in several space dimensions.
Moreover, as a consequence, we have that the $\BV$ regularity is propagated in time and this allows a precise description of the structure of the
solution $u$ to \eqref{E_cl} with $u_0\in \BV(\R)$.

Since the well-posedness result by Kruzkov holds for $u_0\in L^\infty$ and in general entropy solutions have not bounded variations, it is 
natural to try to understand the structure of the entropy solution in this setting and to look for regularity estimates.
The problem is to quantify the regularizing effect that the nonlinearity of the flux $f$ has on the initial datum.
A first result in this direction is due to Oleinik \cite{Oleinik_translation}: if the flux is uniformly convex, say $f''\ge c>0$, then for every 
positive time $t$ the solution $u(t)\in \BV_{\loc}(\R)$ and it holds the one-sided Lipschitz estimate
\begin{equation}\label{E_Oleinik_i}
D_xu(t)\le \frac{\mathcal L^1}{ct}.
\end{equation}
Observe that if the flux $f$ is linear, say $f(w)=\lambda w$ for some $\lambda\in \R$, then the solution to \eqref{E_cl} is simply
\begin{equation*}
u(t,x)=u_0(x-\lambda t),
\end{equation*}
therefore no regularization occurs in this case.

Several results have been obtained between these two extremal cases.

First, we consider the case in which the flux $f$ has no flat parts: 
in order to fix the terminology, we say that $f$ is \emph{weakly genuinely nonlinear} if $\{w:f''(w)\ne 0\}$ is dense.
Under this assumption on the flux it is proved in \cite{Tartar_notes} that an equibounded family of entropy solutions to \eqref{E_cl}
is precompact in $L^1_{\loc}(\R)$. 
Always relying on some nondegeneracy condition of the flux, regularity estimates in terms of fractional Sobolev spaces can be obtained also in 
several space dimensions, by means of the kinetic formulation and averaging lemmas (see \cite{LPT_kinetic, DPLM_averaging_lemma}).
The kinetic formulation is also one of the basic tools in \cite{DLW_structure}, where the authors prove that solutions to scalar 
conservation laws in several space dimensions enjoy some properties of $\BV$ functions (see also \cite{COW_bologna}).
In \cite{BGJ_fractional,CJ_BVPhi} the regularity of the entropy solution $u$ in the case of a  strictly convex  flux $f$ is expressed in terms of 
$\BV^\Phi$ spaces: they provide a convex function $\Phi:[0,+\infty)\rightarrow [0,+\infty)$ depending on the nonlinearity of $f$ 
such that for every $t>0$ and $[a,b]\subset \R$, the solution $u(t)$ satisfies
\begin{equation}\label{E_Junca_i}
\PTV_{[a,b]}u(t):= \sup_{n\in \N,\, a<x_1<\ldots<x_n<b} \sum_{i=1}^{n-1}\Phi \left(|u(t,x_{i+1})-u(t,x_i)|\right) <+\infty.
\end{equation}

Next we require some more structure on the flux $f$:
we say that the flux $f$ has \emph{polynomial degeneracy} if $\{f''(w)=0\}$ is finite and for each $w\in \{f''(w)=0\}$ there exists $p\ge 2$ 
such that $f^{(p+1)}(w)\ne 0$. More precisely, for every $w\in \{f''(w)=0\}$ let $p_w$ be the minimal $p\ge 2$ such that $f^{(p+1)}(w)\ne 0$ and let
$\bar p=\max_wp_w$. We say that $\bar p$ is the \emph{degeneracy} of $f$.

As conjectured in \cite{LPT_kinetic}, it is proved in \cite{Jabin_scalar} that if the flux $f$ as above has degeneracy $p\in \N$, then for every 
$\e,t>0$ the entropy solution $u(t) \in W_{\loc}^{s-\e,1}(\R)$, with $s=\frac{1}{p}$. The result is proved actually in several space dimensions.
However in this setting, it seems convenient to describe the
regularity of $u$ in terms of fractional $\BV$ spaces, i.e. $BV^\Phi$ spaces with $\Phi(u)=u^\alpha$ for some $\alpha\ge 1$. 
In \cite{BGJ_fractional}, under the additional convexity assumption on the flux $f$,
the authors prove that for every $t>0$ the entropy solution $u(t)\in BV_\loc^{s}(\R)$. In particular this implies that $u(t)\in W^{s-\e,p}_{\loc}(\R)$ and
that for every $x$, the function $u(t)$ admits both left and right limits.
The strategy to prove this result is essentially to exploit the BV regularity of $f'\circ u(t)$ for $t>0$ and then to deduce the corresponding regularity
for the solution $u$ itself. 
This regularity holds even out of the convex case: if $f$ has polynomial degeneracy, then for every $t>0$,
\begin{equation}\label{E_Cheng_i}
f'\circ u(t) \in \BV_{\loc}(\R),
\end{equation}
see \cite{Cheng_speed_BV}. We also mention that the case of fluxes with a single inflection 
point is studied in \cite{BC_homogeneous} for homogeneous fluxes $f(u)=|u|^{\alpha-1}u$, by a scaling argument and in 
\cite{Dafermos_inflection} for fluxes with polynomial degeneracy at the inflection point, by an accurate description of the 
extremal backward characteristics. In both these works, the author gets the BV regularity for positive time of the following nonlinear function of the
entropy solution: 
\begin{equation*}
F\circ u(t):= f\circ u(t) - u(t)(f'\circ u(t)).
\end{equation*}
This leads to a fractional regularity of the solution of one order less accurate then the sharp one: more precisely, if $p$ is the degeneracy of 
the flux $f$, it is possible to deduce from the previous results that the entropy solution $u(t)\in BV^s(\R)$ with $s=\frac{1}{p+1}$.

A remarkable fact is that the BV regularity of $f'\circ u(t)$ can be improved to SBV regularity except an at most countable set 
$Q\subset (0,+\infty)$ of singular times. This regularity has been proved for the entropy solution $u$ in \cite{ADL_note} in the case of a uniformly
convex flux $f$, and extended to genuinely nonlinear hyperbolic systems in \cite{BC_SBV}. 
The proof in \cite{ADL_note} is based on the Lax-Oleinik formula that
gives in particular the structure of characteristics in the convex setting: once you have it, the fundamental observation is that the slope of nonintersecting
segments in a given time interval parametrized by the position of their middle points is a Lipschitz function. See also \cite{AGV_SBV}, where the 
same procedure is used to obtained the $\SBV$ regularity of $f'\circ u$ for strictly convex fluxes.
 
The estimate \eqref{E_Cheng_i} is also used in \cite{DLR_dissipation} together with the kinetic formulation to improve the velocity averaging lemma and finally to obtain
that the entropy dissipation measure is rectifiable. 

On the other hand the rectifiability of the  entropy dissipation measure holds for every entropy solution of \eqref{E_cl} with $f$ smooth
\cite{BM_structure}.
The proof is based on the notion of \emph{Lagrangian representation}: since this is the main tool  of this paper, we give some details.
Suppose for simplicity that $u_0$ is continuous. 
The underlying idea is to adapt the method of characteristics, even after the appearance of discontinuities.
We say that $\X:\R^+_t\times \R_y \to \R$ is a \emph{Lagrangian representation}
of an entropy solution $u$ of \eqref{E_cl} if $\X$ is Lipschitz with respect to $t$, increasing with respect to $y$ and 
for every $t\ge 0$ it holds
\begin{equation*}
u(t,x)=u_0(\X(t)^{-1}(x)),
\end{equation*}
for every $x\in \R\setminus N$  with $N$ at most countable.
The link with the PDE \eqref{E_cl} is encoded in the characteristic equation:
for every $y\in \R$ and for $\mathcal L^1$-a.e. $t\in \R^+$
\begin{equation*}
\partial_t \X(t,y)=\lambda (t,\X(t,y)),
\end{equation*}
where
\begin{equation*}
\lambda(t,x)=
\begin{cases}
f'(u(t,x)) & \text{if }u(t) \text{ is continuous at }x, \\
\displaystyle \frac{f(u(t,x+))-f(u(t,x-))}{u(t,x+)-u(t,x-)} & \text{if }u(t)\text{ has a jump at }x.
\end{cases}
\end{equation*}

Several versions of Lagrangian representation have been recently introduced to deal with different settings.
A preliminary version is presented in \cite{BM_scalar} for wave-front tracking approximate solutions, see also \cite{BM_continuous} to deal
with bounded and continuous initial data and \cite{BM_system} for an extension to systems. 
The existence of a Lagrangian representation in the previous settings can be proved essentially by passing
to the limit the wave-front tracking approximation scheme.

It is instead more subtle to pass to the limit for general initial data $u_0\in L^\infty(\R)$: the compactness you have by the regularity of $\X$
and the stability in $L^1_{\loc}(\R)$ granted by Kruzkov theorem seem not to be sufficient to repeat the argument of solutions with bounded variation.
The problem has been overcome in \cite{BM_structure} by the following observation: if $\X$ is a Lagrangian representation of $u$ then
there exists an \emph{existence time} function $\T:\R_y\rightarrow [0,+\infty)$ such that $u$ solves the two initial-boundary value problems
\begin{equation*}
\begin{cases} 
u_t + f(u)_x = 0 & \mbox{in }\{(t,x)\in (0,\T(y))\times \R: x<\X(t,y)\}, \\
u(0,\cdot)=u_0(\cdot)  &  \mbox{in } (-\infty,\X(0,y)), \\
u(t,\X(t,y))=u_0(y) & \mbox{in }(0,\T(y))
\end{cases}
\end{equation*}
and
\begin{equation*} 
\begin{cases}
u_t + f(u)_x = 0 & \mbox{in }\{(t,x)\in (0,\T(y))\times \R: x>\X(t,y)\}, \\
u(0,\cdot)=u_0(\cdot)  &  \mbox{in } (\X(0,y),+\infty), \\
u(t,\X(t,y))=u_0(y) & \mbox{in }(0,\T(y)).
\end{cases}
\end{equation*}
We say that the pair $(\X(\cdot, y),u_0(y))$ is an \emph{admissible boundary} of $u$ in $(0,\T(y))$. Moreover we can cover $\R^+\times \R$ with admissible
boundaries: for every $t>0$,
\begin{equation*}
\X(t,\{y:\T(y)\ge t\})=\R.
\end{equation*}
This allows to prove that the entropy solution $u$ has the following structure: there exists a partition
\begin{equation}\label{E_decomposition2}
\R^+\times \R = A\cup B\cup C,
\end{equation}
where
\begin{enumerate}
\item $A$ is covered by the graphs of at most countably many characteristics curves, in particular it is countably 1-rectifiable;
\item $B$ is the countable union of open sets $B_n$ such that $u\llcorner B_n$ is $\BV$ and $(f'\circ u)\llcorner B$ is locally Lipschitz;
\item for every $(t,x)\in C$, there exists a unique characteristic $\X(\cdot,y)$ such that $\X(t,y)=x$ and $\X(\cdot,y)$ has constant speed in $(0,t)$.
\end{enumerate}
Moreover $u$ has a representative such that for every positive time $t$ and every point $\bar x\in \R$, the limit points of $u(t,x)$ as $x\to \bar x^-$ 
belong all to the same linearly degenerate component of the flux, and similarly if $x\to \bar x^+$. Furthermore the left and the right linearly degenerate
components above are equal at every point in $\R^+\times \R\setminus A$.
In particular if $f$ is weakly genuinely nonlinear, $u$ is continuous on $(\R^+\times \R)\setminus A$ and 
for every $t>0$, the function $u(t)$ has left and right limit at every point $x\in \R$.

A natural question is to obtain the structure of $u$ described above as a consequence of $u\in X$, where $X$ is a compact subspace of $L^1$,
as for example \eqref{E_Oleinik_i} and \eqref{E_Junca_i} in the case of a convex flux.
We also notice that the proof of \eqref{E_Cheng_i} in \cite{Cheng_speed_BV} deals only with fluxes with one or two inflection points
and the author makes implicitly some simplifying
assumptions that do not hold in general. However the argument is valid and we will implement it here.

\medskip
\noindent
We now present the contributions of this paper.
First we consider the case of a weakly genuinely nonlinear flux $f$. We quantify the nonlinearity of $f$ in the following way:
for any $h>0$ let
\begin{equation*}
\di(h):=\min_{a\in [-\|u_0\|_\infty,\|u_0\|_\infty-h]} \dist (f\llcorner [a,a+h], \mathcal A(a,a+h)),
\end{equation*}
where $\mathcal A(a,a+h)$ denotes the set of affine functions defined on $[a,a+h]$ and the distance is computed
with respect to the $L^\infty$ norm.
Moreover let $\Phi$ be the convex envelope of $\di$ and set for every $\e>0$
\begin{equation*}
\Psi_\e(x)=\Phi\left(\frac{x}{2}\right)x^\e.
\end{equation*}
Then we prove the following result.

\begin{theorem01}\label{T1}
Let $f$ be weakly genuinely nonlinear and $u$ be the entropy solution of \eqref{E_cl} with $u_0\in L^\infty(\R)$ with compact support.
Let moreover $\e>0$ and $\Psi_\e$ be defined above.
Then there exists a constant $C>0$ depending on $\mathcal L^1(\conv(\supp u_0))$, $\e$, $\|u_0\|_\infty$ and $\|f'\|_\infty$ such that for every
$t>0$, it holds
\begin{equation*}
u(t)\in \BV^{\Psi_\e}(\R) \qquad \mbox{and}\qquad \Psi_\e\text{-}\TV (u(t))\le C\left(1+\frac{1}{t}\right).
\end{equation*}
\end{theorem01}

The fundamental tool to prove this result is the following ``length'' estimate: let $t>0$ and $x_1<x_2$ be such that $u(t,x_1)=u(t,x_2)=\bar w$ and
consider the characteristics $\X(\cdot,y_1)$ and $\X(\cdot, y_2)$ such that $\X(t,y_1)=x_1$ and $\X(t,y_2)=x_2$. Let
\begin{equation*}
w_m:= \inf_{(x_1,x_2)} u(t) \qquad \mbox{and} \qquad w_M:= \sup_{(x_1,x_2)} u(t).
\end{equation*}
If we denote by
\begin{equation*}
s=\max\{x_2-x_1, \X(0,y_2)-\X(0,y_1)\} \qquad \mbox{and} \qquad \di(w_m,w_M):= \dist\left(f\llcorner [w_m,w_M],\mathcal A(w_m,w_M)\right),
\end{equation*}
then it holds
\begin{equation}\label{E_length_i}
s \ge \frac{\di(w_m,w_M)t}{2\|u_0\|_\infty}.
\end{equation}
Roughly speaking it means that an oscillation between two values at time $t$ must occupy a space, at time 0 or at time $t$, of length 
bounded by below in terms of the nonlinearity between the extremal values.
In particular the number of disjoint oscillations between two given values on a given space interval is uniformly bounded and this 
implies the regularity stated in Theorem \ref{T1}.

Next, we consider fluxes of poynomial degeneracy: in particular we prove the following theorem.
\begin{theorem02}\label{T2}
Let $f$ be a flux of polynomial degeneracy and let $u$ be the entropy solution of \eqref{E_cl} with $u_0\in L^\infty(\R)$ with compact support.
Then there exists a constant $C>0$ depending on $\mathcal L^1(\conv(\supp u_0))$, $\|u_0\|_\infty$ and $f$ such that for every
$t>0$, it holds
\begin{equation}\label{T_Cheng_i}
\TV (f'\circ u(T)) \le C\left(1+\frac{1}{T}\right).
\end{equation}
\end{theorem02}
By means of \eqref{E_length_i}, we reduce the proof to the analysis of the solution in regions where the oscillation is small and the estimate follows
by a careful analysis on the characteristics. If $u$ takes values far from the inflection points, the structure of characteristics is well-known 
and it implies a one-sided Lipschitz estimate as \eqref{E_Oleinik_i} for $f'\circ u(t)$.
If instead $u$ oscillates around an inflection point, the structure of characteristics is described in details in \cite{Dafermos_inflection}.
Then we can conclude adapting the argument in \cite{Cheng_speed_BV}.

As a corollary of Theorem \ref{T_Cheng_i}, we deduce the following regularity result of the entropy solution $u$.
This improves Theorem \ref{T1} in the case of fluxes with polynomial degeneracy.
\begin{theorem03}\label{T3}
Let $f$ be a flux of degeneracy $p$ and let $u$ be the entropy solution of \eqref{E_cl} with $u_0\in L^\infty(\R)$ with compact support.
Then there exists a constant $C>0$, depending on $\mathcal L^1(\conv(\supp u_0))$, $\|u_0\|_\infty$ and $f$, such that for every
$t>0$, it holds
\begin{equation*}
u(t)\in \BV^{1/p}(\R) \qquad \mbox{and} \qquad \left(\TV^{1/p} u (t)\right)^p \le C\left(1+\frac{1}{t}\right).
\end{equation*}
\end{theorem03}

\begin{remark}\label{R_i}
In order to slightly simplify the argument, the proofs of these theorems are provided for non negative solutions with compact support.
By finite speed of propagation, this is not a restrictive assumption.
\end{remark}

Finally we prove the following theorem about the $\SBV$ regularity of $f'\circ u$.
\begin{theorem01}\label{T4}
Let $u$ be the entropy solution of \eqref{E_cl} with $f$ smooth and denote by
\begin{equation*}
\begin{split}
\mathcal B := &~ \{t\in (0,+\infty): f'\circ \bar u(t) \in \BV_{\loc}(\R)\}, \\
\mathcal S := &~ \{t\in (0,+\infty): f'\circ \bar u(t) \in \SBV_{\loc}(\R)\}.
\end{split}
\end{equation*}
Then $\mathcal B \setminus \mathcal S$ is at most countable.
\end{theorem01}
Observe that no additional regularity on the flux is needed to prove this result, but this is relevant in relation with Theorem \ref{T2},
where sufficient conditions to have $\mathcal B=\R^+$ are provided and this allows to prove
the $\SBV$ regularity of $f'\circ u$ with respect to the space-time variable $(t,x)$.
 Indeed the argument is essentially the same as in \cite{ADL_note},
relying on the structure of characteristics presented above under general assumptions on $f$, instead of relying on the Lax-Oleinik formula.
More in details, consider the partition $\R^+\times \R=A\cup B \cup C$ as in \eqref{E_decomposition2}.
Recall that $f'\circ u$ is locally Lipschitz in $B$.
This does not imply that for every compact set $K\subset \R$, $|D_x(f'\circ u(t))|(B\cap K)<+\infty$,
however, when $f'\circ u(t)\in \BV_{\loc}(\R)$, the Cantor part of its 
derivative is concentrated on the section $A_t\cup C_t$ of $A\cup C$ 
at time $t$, and therefore on $C_t$, since $A_t$ is at most countable.
Being $C$ the union of segments starting from 0, we are now in the same position as in \cite{ADL_note}, and we can similarly
prove that if $\bar t\in \mathcal B\setminus \mathcal S$, a positive measure of segments that reach time $\bar t$ cannot be prolonged for $t>\bar t$. In particular, this can happen for a set of times at most countable.

In the final part of this paper we provide some examples:
the first example shows that there exists a flux $f$ with only an inflection point and an entropy solution $u$ such that 
$f'\circ u \notin \BV_{\loc}(\R^+\times \R)$. In particular this proves that the polynomial degeneracy assumption in Theorem \ref{T2} plays a key role, as noticed for different reasons in \cite{DLR_dissipation}.

The second example is related to the possibility of repeating the analysis in \cite{DLR_dissipation}, without relying on Theorem \ref{T2}.
In order to be more precise we fix the notation in the kinetic representation:
\begin{equation*}
\partial_t\chi_{\{u>w\}}+f'(w)\partial_x\chi_{\{u>w\}}=\partial_w\mu.
\end{equation*}
In \cite{DLR_dissipation} it is proved that, under the assumptions of Theorem \ref{T2}, the distribution $\partial_{ww}\mu$ can be represented
as a finite measure. We exhibit an example with a general flux $f$ such that $\partial_w \mu$ is not a finite measure.

The third example answers to a question raised in \cite{CJJ_fractional}: we exhibit a function $u\in L^\infty((0,1))$ such that
\begin{equation*}
\PTV^+_{(0,1)} u := \sup_{n\in \N,\, 0<x_1<\ldots<x_n<1} \sum_{i=1}^{n-1}\Phi \left((u(t,x_{i+1})-u(t,x_i))^+\right) <+\infty
\end{equation*}
and $u$ does not belong to $\BV^\Phi((0,1))$.

Finally the fourth example shows that Theorem \ref{T3} is sharp. This result is already known, see e.g. \cite{CJ_oscillating} for a similar
construction. We provide it here for the sake of completeness.

It remains open the problem of the optimal regularity of $f'\circ u$ with $f$ smooth, out of the polynomial degeneracy assumption: examples in \cite{BM_structure} and at
the end of this paper suggest that the right space could be $f'\circ u \in L^1(\R^+,\BV^\Phi(\R))$ with $\Phi$ such that in a neighborhood of $0$ it
holds $\Phi(-x\log x)=x$. One difficulty is that for a fixed time $t$, there is in general no uniform estimates of $\PTV (f'\circ u(t))$.

\subsection{Plan of the paper}
The paper is organized as follows. \\
In Section \ref{S_preliminaries} we collect the preliminary results that will be useful in the following sections.
First we introduce a decomposition in ``undulations'' for piecewise monotone functions, 
then we recall the definition of the space $\BV^\Phi(\R)$ and how the $\PTV$ of a piecewise monotone function can be estimated in terms of
its undulations. After this we prove some elementary properties of the fluxes we are going to consider and we recall some result about
initial-boundary value problems for scalar conservation laws. This will be relevant in connection with the Lagrangian representation, which is
introduced at the end of this section. Then its properties and the related structure of the entropy solution are presented both in the case of
piecewise monotone entropy solutions and in the case of $L^\infty$-entropy solutions. \\ 
Section \ref{S_length} is devoted to the length estimate \eqref{E_length_i}: this result is based only on the existence of a Lagrangian representation for piecewise monotone entropy solutions. \\
In Section \ref{S_frac_reg} we prove Theorem \ref{T1}:
the main argument is contained in Lemma \ref{L_number_oscillations}, where a weak $\ell^1$ estimate is proven for the terms defining 
the $\Phi$-variation of the entropy solution $u$ at a positive time $t$. \\
Section \ref{S_Cheng} is devoted to the proof of Theorem \ref{T2}; first we recall the structure of characteristics in the case of a convex flux
(Lemma \ref{L_convex}), then we consider the case of a flux with one inflection point: Lemma \ref{L_inflection} summarizes the results obtained
in \cite{Dafermos_inflection} about the structure of extremal backward characteristics for solutions with bounded variation.
Once the structure of characteristics is established, we estimate the total variation of $f'\circ u(t)$ for piecewise monotone solutions to initial
boundary value problems with constant boundary data. Proposition \ref{P_convex} deals with the case of a convex flux and and Proposition 
\ref{P_inflection} with the case of a flux with an inflection point of polynomial degeneracy. In both proofs it is useful to recall the interpretation of
characteristics as admissible boundaries and a fundamental step in Proposition \ref{P_inflection} is the argument of \cite{Cheng_speed_BV}.
The general case can be reduced to the cases studied in Proposition \ref{P_convex} and Proposition \ref{P_inflection}, by means of the length estimate \eqref{E_length_i} (Lemma \ref{L_red_small}) and this leads to the proof of Theorem \ref{T2}. \\
In Section \ref{S_frac_Cheng}, we deduce Theorem \ref{T3} from the previous result: the argument consider separately the big and the small jumps
of the entropy solution $u(t)$. The contribution of the big jumps is controlled by the length estimate and small jumps are considered in 
Lemma \ref{L_chord}. \\
Theorem \ref{T4} is proved in Section \ref{S_SBV} combining the structure of characteristics obtained in \cite{BM_structure} with the argument of 
\cite{ADL_note} and as a consequence we get that $f'\circ u \in \SBV_{loc}(\R^+\times \R)$ (Corollary \ref{C_SBV}). \\
Finally the examples described above are presented in Section \ref{S_examples}.

\section{Preliminary results}
\label{S_preliminaries}

In this section we introduce some notation, we prove some basic lemma  and we recall for completeness results already present in the literature
that will be useful in the following sections.

\subsection{Piecewise monotone functions}\label{Ss_piec_mon}

\begin{definition}\label{D_piec_mon}
A function $u:\R\rightarrow \R$ is said to be \emph{piecewise monotone} 
if there exist $y_1<\ldots<y_k$ in $\R$ such that for every 
$i=1,\ldots,k-1$ the function $u$ is monotone in the interval $(y_i,y_{i+1})$ and in the intervals $(-\infty,y_1)$ and $(y_k,+\infty)$.
\end{definition}
We denote by $X$ the set of piecewise monotone functions $u$ such that the following assumptions are satisfied:
\begin{enumerate}
\item $u$ is bounded;
\item $u$ has compact support; 
\item $u\ge 0$;
\item for every $x\in \R$,
\begin{equation*}
u(x)=\limsup_{y\rightarrow x}u(y);
\end{equation*}
in particular $u$ is upper semicontinuous.
\end{enumerate}

We denote by $\sci u$ the lower semicontinuous envelope of $u$.
It is well-known that the left and right limits of a piecewise monotone function exist at every point and in particular it has at most
countably many discontinuity points. Under the boundedness assumption the limits are finite and  we denote them by
\begin{equation*}
u(x+):=\lim_{y\rightarrow x^+}u(y),\qquad u(x-):=\lim_{y\rightarrow x^-}u(y).
\end{equation*}

In the following proposition, we introduce a decomposition of the functions in $X$ in terms of more elementary piecewise monotone functions.

\begin{proposition}\label{P_undulations}
Let $u\in X$. Then there exist $\tilde N=\tilde N(u)\in \N$ and $\{u_i\}_{i=1}^{\tilde N}\subset X$ non identically zero such that
\begin{enumerate}
\item it holds
\begin{equation}\label{E_decomposition}
u=\sum_{i=1}^{\tilde N} u_i;
\end{equation}
\item for every $i=1,\ldots \tilde N$ there exists $\bar x_i$ such that $u_i$ is increasing in $(-\infty, \bar x_i]$ and decreasing in $[\bar x_i,+\infty)$;
\item for every $i,j = 1,\ldots, \tilde N$ with $i > j$, one of the following holds:
\begin{equation*}
\supp u_i \subset \supp u_j  \qquad \mbox{or} \qquad \Int (\supp u_i) \cap \Int (\supp u_j) = \emptyset. 
\end{equation*}
If the first condition holds, then $u_j$ is constant on the interior of the support of $u_i$ and $u(\bar x_i)\le u(\bar x_j)$.
\end{enumerate}
\end{proposition}
\begin{proof}
First we introduce an operator $\mathcal G:X\rightarrow X$.
Given $u\in X$, if $\max u=0$ we set $\mathcal G(u)=u$. 
If instead $\max u>0$ let
\begin{equation*}
\bar x=\min \{x: u(x)=\max u\}.
\end{equation*}
The existence of $\bar x$ is guaranteed by definition of $X$.
Let $v_l:(-\infty,\bar x]\rightarrow \R$ be the increasing envelope of $u\llcorner (-\infty,\bar x]$:
\begin{equation*}
v_l=\sup \{v':(-\infty,\bar x]\rightarrow \R \mbox{ such that $v'$ is increasing and }v'\le u\llcorner (-\infty,\bar x]\}.
\end{equation*}
Similarly let $v_r:(\bar x,+\infty)\rightarrow \R$ be the decreasing envelope of $u\llcorner (\bar x,+\infty)$:
\begin{equation*}
v_r=\sup \{v':(\bar x,+\infty)\rightarrow \R \mbox{ such that $v'$ is decreasing and }v'\le u\llcorner (\bar x,+\infty)\}.
\end{equation*}
Then let
\begin{equation*}
\mathcal G(u)=
\begin{cases}
v_l & \mbox{in }(-\infty,\bar x], \\
v_r & \mbox{in }(\bar x,+\infty).
\end{cases}
\end{equation*}
It is straightforward to check that $\mathcal G(u)\in X$.

Moreover $u-\mathcal G(u)\in X$ and this allows to iterate this procedure: we set $u_1=\mathcal G(u)$ and by induction for $n>1$
\begin{equation*}
u_n=\mathcal G\left(u-\sum_{i=1}^{n-1}u_i\right).
\end{equation*}
We show now that there are only finitely many $n\in \Z^+$ such that $u_n$ is not identically zero.

If $u=0$ we set $k(u)=0$ and for every $u\in X$ non identically zero, we set $k(u)=\bar k$ where $\bar k$ is the minimum value of $k\in \Z^+$ 
such that there exists $x_1<\ldots<x_k$ for which $u$ is monotone on $(-\infty,x_1), (x_k,+\infty)$ and $(x_i,x_{i+1})$ for every $i=1,\ldots,k-1$.

It is easy to check that if $\bar x \in [x_i,x_{i+1})$ for some $i\in \{2,\ldots,k-1\}$, then $u-\mathcal G(u)$ is monotone on $(x_{i-1},x_{i+1})$ and 
similarly if $\bar x \in [x_1,x_2)$, then $u$ is monotone in $(-\infty,x_2)$ and if $\bar x \in [x_k,+\infty)$, then $u$ is monotone in $(x_{k-1}, +\infty)$.
Moreover, since $\mathcal G(u)$ is constant on each connected component of $\{x: u(x)\ne \mathcal G(u)=x\}$,  the function $u-\mathcal G(u)$ 
is monotone on $(-\infty,x_1), (x_k,+\infty)$ and $(x_i,x_{i+1})$ for every $i=1,\ldots,k-1$.
Therefore
\begin{equation*}
k(u-\mathcal G(u))\le k(u)-1
\end{equation*}
and this proves that $u_n=0$ for every $n>k(u)$. 

Now we check that conditions (1), (2) and (3) in the statement are satisfied.
Let $\tilde N(u)\in \N$ be such that $u_{\tilde N(u)}\ne 0$ and $u_{\tilde N(u)+1}=0$.
Then, since $\mathcal G(u)=0 \Rightarrow u=0$, by
\begin{equation*}
0=u_{\tilde N(u)+1}= \mathcal G\left(u - \sum_{i=1}^{\tilde N(u)}u_i\right),
\end{equation*}
it follows that \eqref{E_decomposition} holds and this proves Condition (1).
Condition (2) is clearly satisfied by construction.
Consider $i,j\in \{1,\ldots, \tilde N(u)\}$ with $j<i$ such that there exists $x\in \Int (\supp u_i)\cap \Int (\supp u_j)$. 
Then if we denote by $I$ the connected component of 
\begin{equation*}
\left\{x':\sum_{l=1}^j u_l(x')<u(x')\right\}
\end{equation*}
containing $x$, it holds
\begin{equation*}
\supp u_i \subset \bar I \subset \supp u_j
\end{equation*}
and $u_j$ is constant on $I$.

Then we only have to check that $u(\bar x_i)\le u(\bar x_j)$:
since for every $l=1,\ldots, j-1$, $u_l$ is constant on $\supp u_j$ and $\bar x_i\in \supp u_j$, it holds
\begin{equation*}
\begin{split}
u(\bar x_j)= &~ \left(\sum_{l=1}^{j-1}u_l(\bar x_j) \right) + \max \left(u-\sum_{l=1}^{j-1}u_l\right) \\
= &~  \left(\sum_{l=1}^{j-1}u_l(\bar x_i) \right) + \max \left(u-\sum_{l=1}^{j-1}u_l\right) \\
\ge &~ u(\bar x_i).
\end{split}
\end{equation*}
This concludes the proof of Condition (3) and therefore the proof of the proposition.
\end{proof}

\begin{definition}\label{D_undulations}
Let $u\in \X$ and $\{u_i\}_{i=1}^{\tilde N}$ be as in Proposition \ref{P_undulations}. Then we say that $u_i$ is an \emph{undulation} of $u$ and that
$h_i:=\max u_i$ is its \emph{height}. Moreover we say that $u_i$ is a \emph{descendant} of $u_j$ if $\supp u_i\subset \supp u_j$.
\end{definition}

\begin{figure}
\centering
\def\svgwidth{0.8\columnwidth}
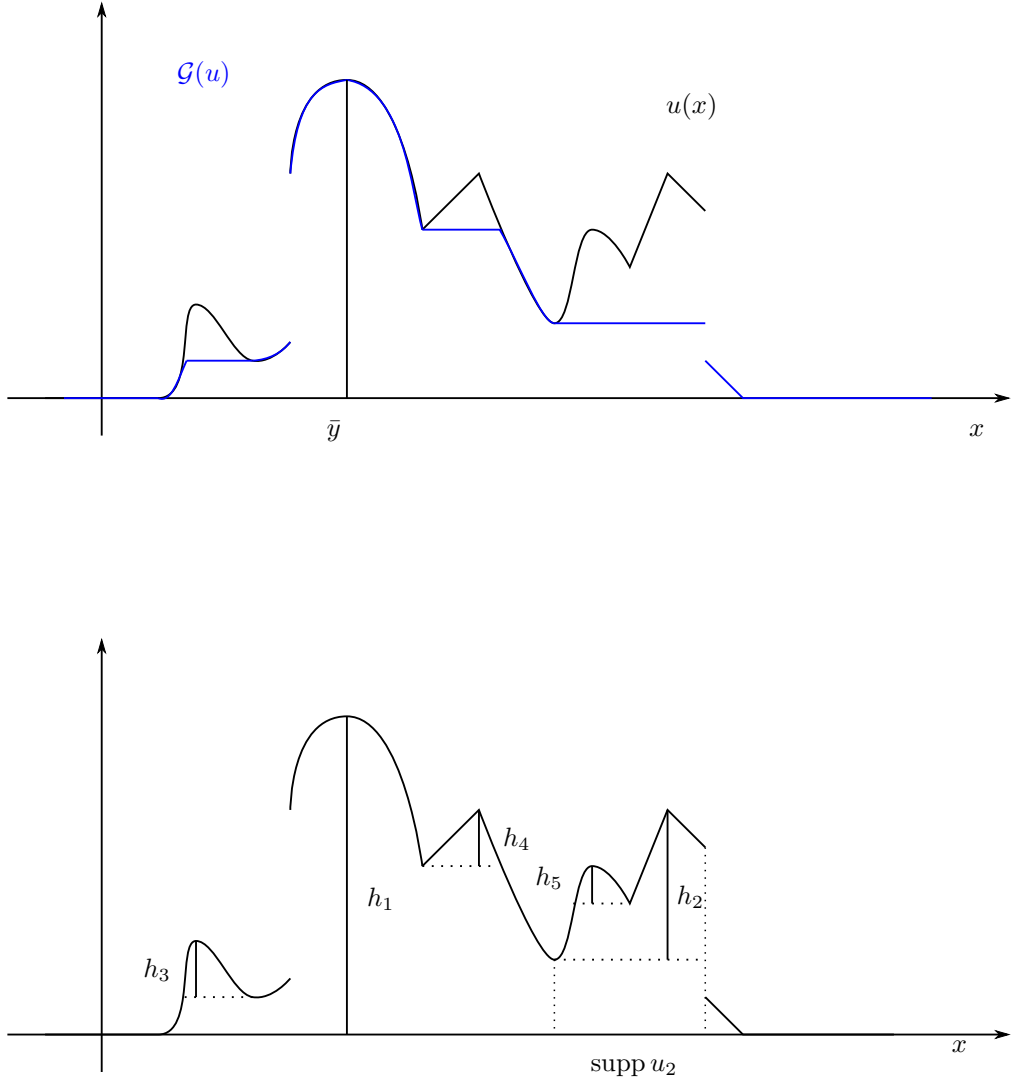
\caption{A representation of the decomposition: the figure above represents the operator $\mathcal G$ and in the figure below they are represented the height of the undulations.}
\end{figure}

\subsection{$\BV^\Phi$ spaces}
In this section we recall, for the convenience of the reader, the definition of $\BV^\Phi$ spaces on the real line (see \cite{MO_BVPhi} for more details) 
and we see how the $\Phi$-total variation of piecewise monotone functions can be estimated in terms of their undulations.
Moreover we recall some basic properties of functions of bounded variations.

\begin{definition}
Let $\Phi:[0,+\infty)\rightarrow [0,+\infty)$ be a convex function with $\Phi(0)=0$ and $\Phi>0$ in $(0,+\infty)$.
Let $I\subset \R$ be a nonempty interval and for $k\in \N$ denote by 
\begin{equation*}
\mathcal P_k(I)=\big\{(x_1,x_2,\ldots,x_k)\in I^k: x_1<x_2<\ldots<x_k\big\} \qquad \mbox{and}\qquad \mathcal P(I)=\bigcup_{k\in \N}\mathcal P_k(I).
\end{equation*}
The $\Phi$-total variation of $u$ on $I$ is
\begin{equation*}
\PTV_I (u)= \sup_{\mathcal P(I)}\sum_{i=1}^{k-1}\Phi(|u(x_{i+1})-u(x_i)|).
\end{equation*}
If the supremum is finite we say that $u\in \BV^\Phi(I)$.
\end{definition}

If $\Phi$ is the identity the $\Phi$-total variation coincides with the classical total variation.
It will be of particular interest also the case $\Phi(z)=z^p$ with $p> 1$. 
In this case if $\PTV u(I)<\infty$ we write that $u\in \BV^{1\over p}(I)$.

Let us recall an elementary lemma about convex functions due to Karamata.
\begin{proposition}\label{P_karamata}
Let $\phi:[0,+\infty)\rightarrow \R $ be increasing and convex and let $a_k,b_k\in [0,+\infty)$ for $k=1,\ldots,n$.
Assume that for every $k=1,\ldots, n-1$
\begin{equation*}
 a_{k+1}\le a_k, \qquad b_{k+1}\le b_k 
\end{equation*}
and for every $k=1,\ldots, n$
\begin{equation*}
\sum_{i=1}^ka_i\ge \sum_{i=1}^kb_i.
\end{equation*}
Then 
\begin{equation*}
\sum_{k=1}^n\phi(a_k) \ge \sum_{k=1}^n\phi(b_k).
\end{equation*}
\end{proposition}
\begin{proof}
For $i=1,\ldots k$ denote by
\begin{equation*}
\Delta \phi_i =
\begin{cases}
\displaystyle\frac{\phi(a_i)-\phi(b_i)}{a_i-b_i} & \mbox{if }a_i\ne b_i \\
\max \{\partial^- \phi(a_i+)\} & \mbox{if }a_i=b_i,
\end{cases}
\end{equation*}
where $\partial ^-\phi$ denotes the subdifferential of $\phi$.
Therefore
\begin{equation*}
\Delta \phi_i(a_i-b_i)= \phi(a_i)-\phi(b_i).
\end{equation*}
Since $\phi$ is convex and increasing, for every $i\in 1,\ldots, k-1$
\begin{equation*}
0\le \Delta \phi_{i+1}\le \Delta \phi_i.
\end{equation*}
We prove by induction that for every $k=1,\ldots, n$
\begin{equation*}
\sum_{i=1}^k\phi(a_i)-\phi(b_i)  \ge \Delta \phi_k\sum_{i=1}^k(a_i-b_i).
\end{equation*}
For $k=1$ it holds by hypothesis, and if the claim holds for $k$, then
\begin{equation*}
\sum_{i=1}^{k+1}\phi(a_i)-\phi(b_i)\ge  \Delta \phi_k\sum_{i=1}^k(a_i-b_i) + \Delta \phi_{k+1}(a_{k+1}-b_{k+1}) \ge  \Delta \phi_{k+1}\sum_{i=1}^{k+1}(a_i-b_i).
\end{equation*}
If $k=n$, we have
\begin{equation*}
\sum_{i=1}^n\phi(a_i)-\phi(b_i)\ge \Delta \phi_n\sum_{i=1}^n(a_i-b_i)\ge 0,
\end{equation*}
which is the claim.
\end{proof}

Now we prove that it is possible to control the $\Phi$-total variation of a function $u\in X$ in terms of its undulations.
To simplify the exposition we assume the following additional properties about $u$:
\begin{enumerate}
\item $u$ is continuous;
\item $\supp u=[a,b]$ for some $a,b\in \R$ and local minima and maxima of $u$ assume different values.
\end{enumerate}
The proof in general follows by a simple approximation argument.

\begin{lemma}\label{L_osc_phi}
Let $u\in X$ and let $(h_i)_{i=1}^{\tilde N(u)}$ be the heights of its undulations.
Then 
\begin{equation*}
\TV (u) =  2\sum_{i=1}^{\tilde N(u)}h_i \qquad \mbox{and} \qquad 
\PTV (u)\le 2\sum_{i=1}^{\tilde N(u)}\Phi(h_i).
\end{equation*}
\end{lemma}
\begin{proof}
Given two functions $v_1,v_2:\R\rightarrow \R$ of bounded variation and $v=v_1+v_2$ it holds
\begin{equation*}
\TV (v) \le \TV (v_1) + \TV (v_2).
\end{equation*}
If we also require that $v_1$ is constant on the support of $v_2$ then equality holds.

By Property (3) in Proposition \ref{P_undulations} and the additional continuity assumption on $u$, if $u_i$ is a descendant of $u_j$, then $u_j$ is 
constant on $\supp u_i$ and obviously the same holds if the supports of $u_i$ and $u_j$ have disjoint interiors. 
In particular for every $k=1,\ldots \tilde N(u)-1$ the function $\sum_{i=1}^k u_i$ is constant on the support of $u_{k+1}$, 
therefore we can prove by induction that
\begin{equation*}
\TV (u) = \TV \left(\sum_{i=1}^{\tilde N(u)} u_i\right) = \sum_{i=1}^{\tilde N(u)}\TV (u_i)= 2\sum_{i=1}^{\tilde N(u)} h_i.
\end{equation*}

Now we consider the case of the $\Phi$-total variation.
Let $\e>0$ and $(x_1,\ldots,x_k)\in \mathcal P$ be such that
\begin{equation*}
\PTV (u) -\e < \sum_{i=1}^{k-1}\Phi(|u(x_{i+1})-u(x_i)|).
\end{equation*}
Denote by $(w_j)_{j=1,\ldots, k-1}$ the non increasing rearrangement of $(|u(x_{j+1})-u(x_j)|)_{j=1,\ldots,k-1}$ and by 
$(\tilde z_j)_{j=1,\ldots, \tilde N(u)}$ the non increasing rearrangement of $(h_j)_{j=1,\ldots, \tilde N(u)}$.
Then let $(z_j)_{j\in \N}$ be the sequence defined by
\begin{equation*}
z_{2j-1}=z_{2j}=
\begin{cases}
\tilde z_j & \mbox{if }1\le j\le \tilde N(u), \\
0 & \mbox{if }j>\tilde N(u),
\end{cases}
\end{equation*}
and consider it restricted to $j=1,\ldots,k-1$.
The conclusion follows by Proposition \ref{P_karamata} with $a_j=z_j$ and $b_j=w_j$:
we only have to check that for every $\bar j=1,\ldots,k-1$ it holds 
\begin{equation}\label{E_hyp_kam}
\sum_{j=1}^{\bar j}z_j \ge \sum_{j=1}^{\bar j} w_j.
\end{equation}

Consider $(x_1,\ldots,x_{2k})\in \mathcal P_{2k}$ a maximum point in $\mathcal P_{2k}$ of the quantity
\begin{equation*}
\sum_{i=1}^k u(x_{2i})-u(x_{2i-1}).
\end{equation*}
Then, if we denote by $\bar x_j$ the maximum point of the undulation $u_j$ for $j=1,\ldots,\tilde N(u)$ it clearly holds that for every $i=1,\ldots,k$ there
exists $j(i)$ such that $x_{2i}=\bar x_{j(i)}$. Moreover, by the maximality of the partition, it is fairly easy to prove that if $\bar x_j=\bar x_{j(i)}$ for some
$i$ and $u_j$ is a descendant of another undulation $u_{j'}$, then there exists $i'$ such that $j'=j(i')$.
Set
\begin{equation*}
\tilde u=\sum_{i=1}^k u_{j(i)}.
\end{equation*}
Since $u_j\ge 0$, it holds $\tilde u\le u$. Moreover, if $\bar x_j$ is a maximum point of $u_j$ and $\bar j$ is such that $u_j$ is not a descendant of 
$u_{\bar j}$, then $u_{\bar j}(\bar x_j)=0$.
Therefore it holds 
\begin{equation*}
u(\bar x_{j(i)})=\sum_{l=1}^{\tilde N(u)} u_l(\bar x_{j(i)})= \sum_{i=1}^k u_i(\bar x_{j(i)})=
\tilde u(\bar x_{j(i)}) \qquad \mbox{for every }i=1,\ldots, k.
\end{equation*}
It follows that
\begin{equation*}
\sum_{i=1}^k u(x_{2i})-u(x_{2i-1}) \le \sum_{i=1}^k \tilde u(x_{2i})-\tilde u(x_{2i-1}) \le \frac{1}{2}\TV (\tilde u) = \sum_{j\in I}h_j,
\end{equation*}
which is exactly \eqref{E_hyp_kam} and this concludes the proof.
\end{proof}

\begin{remark}
Looking at the proof we have that the positive and the negative parts
\begin{equation*}
\TV_+^\Phi(u) := \sup_{\mathcal P(I)}\sum_{i=1}^{k-1}\Phi((u(x_{i+1})-u(x_i))^+) \quad \mbox{and} \quad 
\TV_-^\Phi(u) := \sup_{\mathcal P(I)}\sum_{i=1}^{k-1}\Phi((u(x_{i+1})-u(x_i))^-)
\end{equation*}
are separately bounded by $\sum_{i=1}^{N(u)} \Phi(h_i)$.
The converse is not true, even up to a constant.
The $\Phi$-total variation of a piecewise monotone function depends not only the height of its undulations, but also how they are placed.
In general the positive and the negative $\Phi$-total variations are not comparable with $\sum \Phi(h_i)$, 
and it may be that they are not  comparable with each other. 
In the last section we provide an example where the increasing $\Phi$-variation is finite and the decreasing $\Phi$-variation is not.
The question has been raised in \cite{CJJ_fractional} where it has been observed that a counterexample like this one precludes 
the possibility to obtain $\BV^\Phi$ regularity by an Oleinik type estimate in the case of convex fluxes.
\end{remark}

In the following lemma we collect some easy properties of functions of bounded variation that will be useful later.
\begin{lemma}\label{L_trivial}
Let $g:\R\rightarrow \R$  be a piecewise monotone left-continuous function with bounded variation and suppose it does not have positive jumps, i.e. for every $x\in \R$
\begin{equation*}
g(x+)\le g(x-).
\end{equation*}
Then, for every $a<b$ and $a=x_1<\ldots<x_n=b$ it holds
\begin{equation*}
\TV^+_{(a,b)} g=  \sum_{i=1}^{n-1} \TV^+_{(x_i,x_{i+1})}g
\qquad \mbox{and} \qquad
\TV^-(g) - 2\|v\|_\infty \le \TV^+(g) \le \TV^-(g) + 2\|g\|_\infty.
\end{equation*}
\end{lemma}

Finally we state an easy lemma for future reference.
\begin{lemma}\label{L_uniform_continuity}
Let $a,b\in \R$ with $a<b$ and let $g:(a,b)\rightarrow \R$ be increasing and bounded. Denote by
\begin{equation*}
h:=\max_{x\in (a,b)}g(x+)-g(x-).
\end{equation*}
Then for every $\e>h$ there exists $\delta>0$ such that
\begin{equation*}
|x_2-x_1|<\delta \qquad \Longrightarrow \qquad |g(x_2)-g(x_1)|<\e.
\end{equation*}
\end{lemma}

\subsection{Weakly genuinely nonlinear fluxes and fluxes with polynomial degeneracy}
The regularity of the entropy solution to \eqref{E_cl} depends on the nonlinearity of the flux $f$; we introduce here some terminology.
\begin{definition}\label{D_weak_nonlinear}
We say that $f:\R\rightarrow \R$ is \emph{weakly genuinely nonlinear} if the set $\{w:f''(w)\ne 0\}\subset \R$ is dense. 
\end{definition}

We will also consider the case of a flux $f\in C^\infty(\R)$ such that the set $\{w:f''(w)=0\}$ is finite; let $w_1<\ldots < w_S$ denote its elements.

\begin{definition}\label{D_degeneracy}
We say that $f$ has \emph{degeneracy $p\in \N$}, at the point $w_s$ if $p\ge 2$ and
\begin{equation*}
f^{(j)}(w_s)=0\quad  \mbox{for}\quad j =2,\ldots, p \qquad \mbox{and}\qquad f^{(p+1)}(w_s)\ne 0.
\end{equation*}
If there exists such a $p\in \N$ we say that $f$  has \emph{polynomial degeneracy at $w_s$}.
If the set $\{w:f''(w)=0\}$ is finite and $f$ has polynomial degeneracy at each of its points, we say that $f$ has \emph{polynomial degeneracy}.
Finally we say that $f$ has \emph{degeneracy $p$} if $f$ has polynomial degeneracy and $p$ is the maximum of the degeneracies of $f$ 
at the points of $\{w:f''(w)=0\}$.
\end{definition}

In Section \ref{S_Cheng} it will be important the behavior of $f$ around its inflection points.
The following lemma and its corollary will be useful to describe the small oscillations of the solution around an inflection point of the flux.
\begin{lemma}\label{L_conjugate}
Let $f:\R\rightarrow \R$ be smooth and let $\bar w$ be such that for every $w\in \R\setminus \{\bar w\}$, 
\begin{equation}\label{E_conv_conc}
f''(w)(w-\bar w)<0.
\end{equation}
Then there exists $\delta>0$ such that for every $w\in (\bar w -r, \bar w + r) \setminus \{\bar w\}$,
there exists a unique \emph{conjugate point} $w^*\in \R\setminus \{w\}$ such that
\begin{equation}\label{E_conjugate}
f'(w^*)=\frac{f(w)-f(w^*)}{w-w^*}.
\end{equation}
Assume moreover that $\bar w$ is of polynomial degeneracy, then there exist $\delta,\e>0$ such that
for every $w,w'\in(\bar w-\delta,\bar w+\delta)\setminus \{\bar w\}$ with $w\ne w'$ it holds
\begin{equation}\label{E_ratio_delta}
\frac{w^*-\bar w}{w-\bar w} \in (-1+\e,0) \qquad \mbox{and} \qquad \frac{f'(w^*)-f'(w'^*)}{f'(w)-f'(w')} \in (0,1-\e).
\end{equation}
\end{lemma}
See Figure \ref{F_conjugate}.
\begin{proof}
Suppose $w<\bar w$, being the opposite case analogous and let $g_w:\R  \to \R$ be defined by
\begin{equation*}
g_w(t)=  f(t) + f'(t)(w-t) -f(w).
\end{equation*}
Observe that \eqref{E_conjugate} is equivalent to
\begin{equation}\label{E_conjugate_2}
w^* \ne w  \qquad  \mbox{and} \qquad g_w(w^*)=0.
\end{equation}
Since $g_w'(t)=f''(t)(w-t)$, by \eqref{E_conv_conc}, it holds
\begin{equation*}
g'_w<0 \quad \mbox{in }(w,\bar w) \qquad \mbox{and} \qquad g'_w>0 \quad \mbox{in }(-\infty,w)\cup(\bar w, +\infty).
\end{equation*}
Moreover $g_w(w)=0$ and by strict monotonicity this proves that there exists at most one $w^*$ as in \eqref{E_conjugate_2}
and it exists if and only if $\lim_{t\to +\infty}g_w(t)>0$.
\begin{equation}\label{E_lim_g}
\begin{split}
\lim_{t\to +\infty}g_w(t) = &~ g_w(\bar w)  + \int_{\bar w}^{+\infty}g_w'(t)dt \\
\ge  &~ g_w(\bar w)  + \int_{\bar w}^{+\infty}|f''(t)|(t-\bar w)dt.
\end{split}
\end{equation}
Let $A:=\int_{\bar w}^{+\infty}|f''(t)|(t-\bar w)dt>0$.
Since the function $w\mapsto g_w(\bar w)$ is continuous and $g_{\bar w}(\bar w)=0$, there exists $\delta>0$ such that
$w\in (\bar w-\delta,\bar w) \Rightarrow g_w(\bar w)>-A$ and therefore, by \eqref{E_lim_g}, for which $w^*$ exists.
 
Now let us consider the case of $f$ with polynomial degeneracy $p\in \N$ at $\bar w$. Since the statement is elementary, we only sketch the
computations.
Notice that since $f''$ changes sign at $\bar w$, $p$ is even. 
It is sufficient to prove that 
\begin{equation}\label{E_limit_1}
\lim_{w\rightarrow \bar w} \frac{w^*-\bar w}{w-\bar w}=\bar \rho,
\end{equation}
with $\bar \rho\in (-1,0)$ and
\begin{equation*}
\lim_{w_1,w_2 \rightarrow \bar w} \frac{f'(w_2^*)-f'(w_1^*)}{f'(w_2)-f'(w_1)} = \bar\rho^{p}.
\end{equation*}

By assumption we have
\begin{equation*}
f(w)\simeq f(\bar w) + f'(\bar w)(w-\bar w) + \alpha (w-\bar w)^{p+1},
\end{equation*}
with $\alpha\ne 0$.

By \eqref{E_conjugate}, we have that
\begin{equation*}
\alpha(p+1)(w^*-\bar w)^p(w^*-w)\simeq \alpha [(w^*-\bar w)^{p+1}-(w-\bar w)^{p+1}],
\end{equation*}
Dividing by $(w-\bar w)^{p+1}$ and setting $\rho:= \frac{w^*-\bar w}{w-\bar w}$, we get
\begin{equation*}
(p+1)\rho^p(\rho-1) \simeq \rho^{p+1}-1.
\end{equation*}
Setting $G(\rho)=p\rho^{p+1}-(p+1)\rho^p +1$, the above formula is equivalent to $G(\rho)\simeq 0$.
It is easy to show that the polynomial $G$ has two roots in $\rho=1$ and one root $\bar \rho \in (-1,0)$.
Since $\frac{w^*-\bar w}{w-\bar w}<0$, the only possibility is that \eqref{E_limit_1} holds.
Moreover
\begin{equation*}
\frac{f'(w_2^*)-f'(w_1^*)}{f'(w_2)-f'(w_1)} \simeq 
\frac{\alpha (p+1) (\bar \rho w_2)^p -\alpha (p+1) (\bar \rho w_1)^p }{\alpha (p+1)w_2^p - \alpha (p+1)w_1^p} = \bar\rho^p.
\qedhere
\end{equation*}
\end{proof}

\begin{figure}
\centering
\def\svgwidth{0.5\columnwidth}
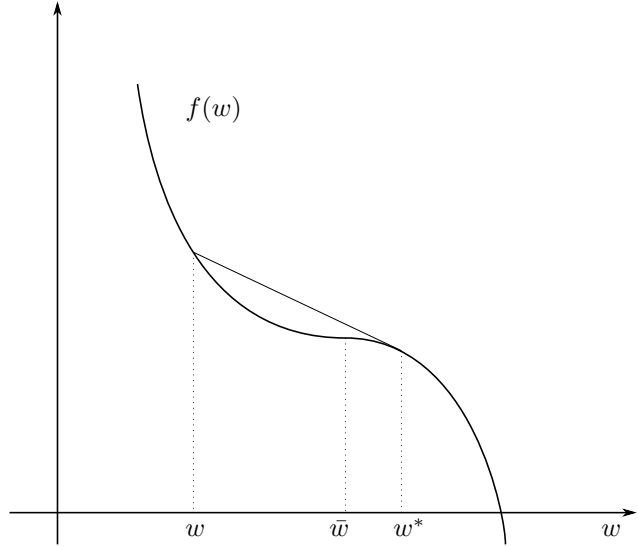
\caption{In this picture are represented the graph of a flux $f$ with an inflection point in $\bar w$ and a point $w$ with its conjugate $w^*$.}\label{F_conjugate}
\end{figure}

Applying the previous lemma around each inflection point of the flux we can easily obtain the following corollary for general fluxes with 
polynomial degeneracy.
\begin{corollary}\label{C_conjugate}
Let $f:\R\rightarrow \R$ be smooth and of polynomial degeneracy and let $w_1<\ldots<w_S$ be the points where $f''$ vanishes. Then there exists $\delta>0$
such that
\begin{enumerate}
\item it holds
\begin{equation*}
\delta < \frac{1}{2}\min_{s=1,\ldots,S-1}\big(w_s-w_{s-1}\big);
\end{equation*}
\item for every $s=1,\ldots, S$ and $w\in (w_s-\delta,w_s+\delta)$ there exists a unique $w^*\in (w_s-\delta,w_s+\delta)$ such that 
\begin{equation*}
f'(w^*)=\frac{f(w)-f(w^*)}{w-w^*};
\end{equation*}
\item there exists $\e>0$ such that for every $s=1,\ldots, S$ and every $w,w'\in(w_s-\delta, w_s+\delta)\setminus \{ w_s\}$ with $w\ne w'$ it holds
\begin{equation*}
\frac{w^*- w_s}{w- w_s} \in (-1+\e,0) \qquad \mbox{and} \qquad \frac{f'(w^*)-f'(w'^*)}{f'(w)-f'(w')} \in (0,1-\e).
\end{equation*}
\end{enumerate}
\end{corollary}

\subsection{Initial-boundary value problems}
In this section we recall some basic facts about the initial boundary value problem for \eqref{E_cl}, introduced in \cite{BLRN_boundary}. This is relevant here in connection with the notion of Lagrangian
representation, that will be introduced later.
\begin{definition}
We say that $(\eta,q)$ is an \emph{entropy-entropy flux pair} if $\eta:\R\rightarrow \R$ is convex and $q:\R\rightarrow \R$ satisfies $q'(w)=\eta'(w)f'(w)$ for $\mathcal L^1$-a.e. $w\in \R$.
In particular we will use the following notation: for every $k\in \R$ let
\begin{equation*}
\eta^+_k(u):=(u-k)^+,\qquad \eta^-_k(u):=(u-k)^-
\end{equation*}
and the relative fluxes
\begin{equation*}
 q^+_k(u):=\chi_{[k,+\infty)}(u)\big(f(u)-f(k)\big),\qquad q^-_k(u):=\chi_{(-\infty,k]}(u)\big(f(k)-f(u)\big),
\end{equation*}
where $\chi_E$ denotes the characteristic function of the set $E$:
\begin{equation*}
\chi_E(u):=
\begin{cases}
1 & \mbox{if }u\in E, \\
0 & \mbox{if }u \notin E.
\end{cases}
\end{equation*}
\end{definition}

Let $\gamma_l,\gamma_r:[0,+\infty)\rightarrow \R$ be Lipschitz with $\gamma_l\le \gamma_r$. Let $T>0$ and denote by
\begin{equation*}
\Omega:=\{(t,x)\in (0,T)\times \R : \gamma_l(t)<x<\gamma_r(t)\}.
\end{equation*}
Let $u_l,u_r:[0,T]\rightarrow \R$ and $u_0:(\gamma_l(0),\gamma_r(0))\rightarrow \R$ be functions of bounded variation and consider the 
initial-boundary value problem
\begin{equation}\label{E_cl_boundary}
\begin{cases}
u_t+f(u)_x=0 & \mbox{in }\Omega, \\
u(t,\gamma_l(t))=u_l(t) & \mbox{for }t\in (0,T), \\
u(t,\gamma_r(t))=u_r(t) & \mbox{for }t\in (0,T), \\
u(0,\cdot)=u_0(\cdot) & \mbox{in }(\gamma_l(0),\gamma_r(0)).
\end{cases}
\end{equation}

\begin{definition}
Let $u\in \BV(\Omega)$ and denote by $\tilde u_l,\tilde u_r:(0,T)\rightarrow \R$ and $\tilde u_0:(\gamma_l(0),\gamma_r(0))\rightarrow \R$ the
traces of $u$ on $\graph \gamma_l, \graph \gamma_r$ and on $\{0\}\times(\gamma_l(0),\gamma_r(0))$ respectively.
We say that $u\in \BV(\Omega)$ is an \emph{entropy solution} of the initial-boundary value problem \eqref{E_cl_boundary} 
if it is an entropy solution of \eqref{E_cl} in $\Omega$, $u_0=\tilde u_0$ and for $\mathcal L^1$-a.e. $t\in (0,T)$
\begin{equation}\label{E_right}
\begin{split}
-\dot\gamma_l(t)\eta_k^+(\tilde u_l(t))+q_k^+(\tilde u_l(t))\le 0 & \qquad \forall k \ge u_l(t), \\
-\dot\gamma_l(t)\eta_k^-(\tilde u_l(t))+q_k^-(\tilde u_l(t))\le 0 & \qquad \forall k \le u_l(t),
\end{split}
\end{equation}
and similarly
\begin{equation}\label{E_left}
\begin{split}
-\dot\gamma_r(t)\eta_k^+(\tilde u_r(t))+q_k^+(\tilde u_r(t))\ge 0 & \qquad \forall k \ge u_r(t), \\
-\dot\gamma_r(t)\eta_k^-(\tilde u_r(t))+q_k^-(\tilde u_r(t))\ge 0 & \qquad \forall k \le u_r(t).
\end{split}
\end{equation}
\end{definition}

In the following proposition the total variation of the solution of an initial-boundary value problem is estimated in terms of the total variation of the
initial and boundary data. In \cite{Amadori_wft_bdy} the estimate is proved for wave-front tracking approximate solutions in the context of
systems. Since the analysis can be repeated in this setting, we omit the proof.
\begin{proposition}\label{P_TV_est}
Let $u\in \BV(\Omega)$ be a bounded entropy solution of \eqref{E_cl_boundary}. Then 
\begin{equation*}
\TV\, \left(u(T)\right) \le \TV (u_0) + \TV(u_l) + \TV(u_r) + |u_l(0+)-u_0(\gamma_l(0)+)| + |u_r(0+)-u_0(\gamma_r(0)-)|,
\end{equation*}
where the total variations are computed on the domains of the corresponding functions.
\end{proposition}

It is useful to reverse the point of view: now let $u$ be an entropy solution of \eqref{E_cl} and consider the following definition.

\begin{definition}
Let $\gamma:\R\rightarrow \R$ be Lipschitz, $w\in\R$ and $T>0$. Moreover consider the domains
\begin{equation*}
\Omega_l=\{(t,x)\in [0,T]\times \R: x<\gamma(t)\},\qquad \Omega_r=\{(t,x)\in [0,T]\times \R: x>\gamma(t)\}.
\end{equation*}
We say that $(\gamma,w)$ is an \emph{admissible boundary} for $u$ in $[0,T]$ if $u_l:=u\llcorner \Omega_l$ and $u_r:= u \llcorner \Omega_r$ 
solve the initial boundary value problems
\begin{equation*}
\begin{cases} 
u_t + f(u)_x = 0 & \mbox{in }\Omega_l, \\
u(0,\cdot)=u_0(\cdot)  &  \mbox{in } (-\infty,\gamma(0)), \\
u(t,\gamma(t))=w & \mbox{in }(0,T),
\end{cases}
\qquad \qquad \mbox{and} \qquad \qquad 
\begin{cases}
u_t + f(u)_x = 0 & \mbox{in }\Omega_r, \\
u(0,\cdot)=u_0(\cdot)  &  \mbox{in } (\gamma(0),+\infty), \\
u(t,\gamma(t))=w & \mbox{in }(0,T),
\end{cases}
\end{equation*}
respectively.
\end{definition}
\begin{remark}
Since $u$ is an entropy solution, $(\gamma,w)$ is an admissible boundary for $u$ in $[0,T]$ if and only if for $\mathcal L^1$-a.e. $t\in [0,T]$
the corresponding versions of \eqref{E_right} and \eqref{E_left} hold.
\end{remark}

The key fact is that for every point $(t,x)\in \R^+\times \R$, there exists at least an admissible boundary $(\gamma,w)$ with $\gamma(t)=x$.
We will be more precise in the following section, relating this notion to the notion of Lagrangian representation. In the following lemma from \cite{BM_structure}, we recall
a stability result for the admissible boundaries.

\begin{proposition}[Stability]\label{P_stab}
Let $u^n$ be entropy solutions of \eqref{E_cl} and $(\gamma^n,w^n)$ admissible boundaries for $u^n$.
Suppose that 
\begin{itemize}
\item $u^n \rightarrow u$ strongly in $L^1(\R)$;
\item $w^n\rightarrow w$;
\item $\gamma^n \rightarrow \gamma$ uniformly.
\end{itemize}
Then $(\gamma,w)$ is an admissible boundary for $u$.
\end{proposition}

\subsection{Lagrangian representation and structure of entropy solutions}
In this section we recall the notion of Lagrangian representation and its properties in the case of piecewise monotone solutions and general $L^\infty$-entropy solutions.

\begin{definition}\label{D_L_infty}
Let $u$ be the entropy solution of \eqref{E_cl} with $u_0\in L^\infty(\R)$. We say that the couple $(\X,\U)$ is a \emph{Lagrangian representation} of $u$ if
\begin{enumerate}
\item $\X:[0,+\infty)\times \R\rightarrow \R$ is continuous, $t\mapsto\X(t,y)$ is Lipschitz for every $y$ and $y\mapsto \X(t,y)$ is non decreasing for every $t$;
\item $\U\in L^\infty(\R)$;
\item there exists a representative of $u$ and an at most countable set $\bar Y=\{y_n\}_{n\in \N}$ such that, setting 
\begin{equation}\label{E_def_A}
A=\bigcup_{n\in \N}\graph (\X(\cdot,y_n)),
\end{equation}
for every $(t,x)\in \R^+\times \R\setminus A$,
\begin{equation*}
\X(t)^{-1}(x)=\{y\} \qquad \mbox{and} \qquad  u(t,x) =\U(y).
\end{equation*}
\end{enumerate}
\end{definition}

In the next proposition, we state the existence of a Lagrangian representation for piecewise monotone solutions and we recall the properties 
that we will need in the following.
The properties listed here are the collection of the contributions in \cite{BM_scalar,BM_continuous,BM_structure}.

\begin{proposition}\label{P_lagr_repr}
Let $u_0\in X$ (defined in Section \ref{Ss_piec_mon}) be continuous and let $u$ be the entropy solution of \eqref{E_cl}. Then there exist a Lagrangian representation $(\X,\U)$ of $u$, an at most countable set $Q'\subset (0,+\infty)$ and a function $\T:\R\rightarrow [0,+\infty)$ 
(which we call \emph{existence time function}) such that 
\begin{enumerate}
\item $\X(0)=\Id$ and $\U=u_0$;
\item for every $t\in [0,+\infty)$, 
	\begin{equation*}
	\{(x,w):w\in [\sci u(t,x), u(t,x)]\} \subset \{(\X(t,y),u_0(y)): \T(y)\ge t\}.
	\end{equation*}
\item for every $t\in [0,+\infty)\setminus Q'$ and for every $(x,w)\in \R^2$ such that
	$u(t)$ is continuous at $x$ and $u(t,x)=w$, or $w\in (\sci u(t,x),u(t,x))$,
	there exists a unique $y(t,x,w) \in \R$ such that
	\begin{equation*}
	\T(y(t,x,w))\ge t, \quad \X(t,y(t,x,w))=x, \quad \mbox{and}\quad u_0(y(t,x,w))=w.
	\end{equation*}
	Moreover, if $u(t,x-)<u(t,x+)$ the function $w\mapsto y(t,x,w)$ is increasing in $(u(t,x-),u(t,x+))$ and 
	if $u(t,x+)<u(t,x-)$ the function $w\mapsto y(t,x,w)$ is decreasing in $(u(t,x+),u(t,x-))$;
\item for every $y\in \R$ the pair $(\X(\cdot, y),u_0(y))$ is an admissible boundary of $u$ in $[0,\T(y)]$.
\end{enumerate}
Moreover there exists a piecewise constant \emph{sign function} $\s: \R\rightarrow \{-1,1\}$ such that
\begin{enumerate}
\item[(5)] for every $t\in [0,\T(y)]\setminus Q'$,
\begin{equation}\label{E_sign}
\begin{split}
\s(y)=1\quad \Longrightarrow \quad  \Big(u_0(y)\le u(t,\X(t,y)+) \quad  \emph{or} \quad  u_0(y)\ge u(t,\X(t,y)-)\Big), \\
\s(y)=-1\quad \Longrightarrow \quad \Big(u_0(y)\le u(t,\X(t,y)-) \quad  \emph{or} \quad  u_0(y)\ge u(t,\X(t,y)+)\Big);
\end{split}
\end{equation}
\item[(6)] if $y_1,y_2\in \{y:\T(y)\ge T\}$ with $y_1<y_2$ and there exists $t\in[0,T)$ such that $\X(t,y_1)=\X(t,y_2)$, then
$u$ is strictly monotone in $(\X(T,y_1),\X(T,y_2))$ and $f'\circ u(T)$ is strictly increasing in $(\X(T,y_1),\X(T,y_2))$.
\item[(7)] the characteristic equation holds: for every $y\in \R$ and for $\mathcal L^1$-a.e. $t\in \R^+$
\begin{equation*}
\partial_t \X(t,y)=\lambda (t,\X(t,y)),
\end{equation*}
where
\begin{equation*}
\lambda(t,x)=
\begin{cases}
f'(u(t,x)) & \text{if }u(t) \text{ is continuous at }x, \\
\displaystyle \frac{f(u(t,x+))-f(u(t,x-))}{u(t,x+)-u(t,x-)} & \text{if }u(t)\text{ has a jump at }x.
\end{cases}
\end{equation*}
\end{enumerate}
\end{proposition}
The existence of a Lagrangian representation as above has been proved by explicitly constructing the representation for wave-front tracking
approximations and passing it to the limit.
We also observe that the possibility of representing the solution with a flow $\X$ continuous with respect to $y$ and such that $\X(0)=\Id$ does 
not directly follow from the natural compactness assumptions and it is related to the fact that constant regions are not created at positive times.
See \cite{BM_structure} for more details and for the proofs of the following proposition for the case of $L^\infty$-entropy solutions.
\begin{proposition}\label{P_structure}
Let $u$ be the entropy solution to \eqref{E_cl} with $u_0\in L^\infty(\R)$. Then there exists a Lagrangian representation $(\X,\U)$ of $u$.
Additionally $\bar u$ and $\bar Y$ in Definition \ref{D_L_infty} can be chosen 
so that there exists a partition of the half-plane $\R^+\times \R=A\cup B\cup C$ with the following properties:
\begin{enumerate}
\item $A$ is given by \eqref{E_def_A} and  $f'\circ \bar u$ is continuous in $\R^+\times \R \setminus A$;
\item $B$ is open and $(f'\circ \bar u)\llcorner B$ is locally Lipschitz;
\item for every $(\bar t,\bar x)\in C$, $\X(\bar t)^{-1}(\bar x)=\{\bar y\}$ for some $\bar y\in \R$ and 
\begin{equation*}
\X(t,\bar y)=\bar x - f'(\bar u(\bar t,\bar x))(\bar t-t)
\end{equation*}
for every $t\in [0,\bar t]$. In particular, $\bar u (t,\X(t,\bar y))=\U(\bar y)$ for every $t\in [0,\bar t]$.
\end{enumerate}
Moreover for every entropy-entropy flux pair $(\eta,q)$ the entropy dissipation measure
\begin{equation*}
\mu:=\eta(u)_t+q(u)_x \ll \mathcal H^1\llcorner A.
\end{equation*}
\end{proposition}

We conclude this section of preliminaries by considering the case of a weakly genuinely nonlinear flux and recalling the well-known chord
admissibility condition. The proof can be found in \cite{BM_structure}.
\begin{proposition}\label{P_chord}
Let $u$ be the entropy solution of \eqref{E_cl} with $u_0\in L^\infty(\R)$ and the flux $f$ weakly genuinely nonlinear or with a general smooth flux
and $u_0\in C^0(\R)\cap L^\infty(\R)$. Then
it is possible to choose the partition in Proposition \ref{P_structure} so that the solution $u$ is continuous in $\R^+\times \R \setminus A$
and for every $(\bar t,\bar x)\in A$ there exist both the limits
\begin{equation*}
u^-:=\lim_{x\to \bar x^-}u(\bar t, x) \qquad \mbox{and} \qquad u^+:=\lim_{x\to \bar x^+}u(\bar t, x).
\end{equation*}
Moreover there exists a set $N\subset\R^+$ such that $\mathcal L^1(N)=0$ and for every $t\in T\setminus N$ at each jump of the solution $u(t)$
the chord condition is satisfied: more precisely
\begin{equation}\label{E_chord}
\begin{split}
 u^-<u^+ \qquad \Longrightarrow \qquad & f(k)\ge f(u^-)+\frac{f(u^+)-f(u^-)}{u^+-u^-}(k-u^-) \quad \forall k \in (u^-,u^+); \\
 u^->u^+ \qquad \Longrightarrow \qquad & f(k)\le f(u^+)+\frac{f(u^+)-f(u^-)}{u^+-u^-}(k-u^+) \quad \forall k \in (u^+,u^-).
\end{split}
\end{equation} 
\end{proposition}

\begin{definition}\label{D_generic}
Let $u$ be an entropy solution of \eqref{E_cl} with $u_0\in X$. We say that $t\in (0,+\infty)$ is \emph{generic} if $t\notin Q' \cup N$, 
where $Q'$ is given in Proposition \ref{P_lagr_repr} and $N$ is given by Proposition \ref{P_chord}.
\end{definition}

We observe that Proposition \ref{P_structure} and \eqref{E_chord} out of the setting of piecewise monotone solutions are needed only in
Section \ref{S_SBV}. The analysis in Sections \ref{S_length} to \ref{S_frac_Cheng} relies only on the structure of piecewise monotone solutions
and Proposition \ref{P_lagr_repr}.

\section{Length estimate}\label{S_length}
In this section we prove a lower bound for the distances of two characteristics with the same value, depending on the nonlinearity of the flux
function $f$ between the extreme values assumed between the two characteristics at a positive time $t$. 
We only assume that the flux $f$ is smooth.
We quantify the nonlinearity of $f$ between two values $w_1\le w_2$ by considering twice the $C^0$ distance of
$f\llcorner[w_1,w_2]$ from the set of affine functions on $[w_1,w_2]$:
\begin{equation}\label{E_di}
\di(w_1,w_2):=\min_{\lambda\in \R}\max_{\{w,w'\}\in [w_1,w_2]}\big(f(w)-f(w')-\lambda(w-w')\big).
\end{equation}

In the statement and in the proof of the following theorem, we will refer to the objects introduced in Proposition \ref{P_lagr_repr}: 
$\X$ is the Lagrangian flow of the entropy solution $u$ to \eqref{E_cl} with  
$u_0\in X$ and continuous, $\T$ denotes the existence time function and $\s$ denotes the sign function.
We recall that the set $X$ has been introduced in Section \ref{Ss_piec_mon}.
\begin{theorem}\label{T_main_estimate}
Let $T>0$ and $u,\X,\T$ be as above. Let $y_l<y_r$ such that
\begin{equation*}
u_0(y_l)=u_0(y_r)=\bar w, \quad 
 \T(y_l) \ge T, \quad  \T(y_r) \ge T,
\end{equation*}
and let
\begin{equation*}
s:=\max\{y_r-y_l, \X(T,y_r)-\X(T,y_l)\}.
\end{equation*}
Then
\begin{equation*}
\di(w_m, w_M)\le \frac{2s\|u_0\|_\infty}{T},
\end{equation*}
where 
\begin{equation*}
[w_m,w_M]=\left\{w:\exists y \in [y_l,y_r]\big( u_0(y)=w, \T(y)\ge T\big)\right\}.
\end{equation*}
\begin{remark}
The set $[w_m,w_M]$ contains the closure of the convex hull of the image $u(t,(\X(T,y_l),\X(T,y_r)))$ by Proposition \ref{P_lagr_repr}, and the inclusion may be strict.
\end{remark}
\end{theorem}
\begin{proof}
Observe that by Proposition \ref{P_chord}, we can assume that $T>0$ is generic.
The general case follows considering the same $y_l$ and $y_r$ for a sequence of generic times $T_n\to T^-$.

Fix $\e>0$ and let 
\begin{equation*}
\lambda= \frac{\X(T,y_l)-\X(0,y_l)}{T}.
\end{equation*}
By Proposition \ref{P_lagr_repr} it immediately follows that for every $t>0$ the solution $u(t)$ is piecewise monotone.
In particular we can choose $w_1, w_2\in [w_m,w_M]$ different from $\bar w$ such that 
\begin{equation*}
\di (w_m,w_M)-\e \le  f(w_2)-f(w_1)-\lambda(w_2-w_1)
\end{equation*}
and such that $w_1,w_2$ are not local maximum or minimum values of $u(T)$.
We consider the case $w_1\le w_2$, being the opposite case analogous.
Since $u_0(y_l)=u_0(y_r)$ there exists $y_1\in [y_l,y_r]$ such that $u_0(y_1)=w_1$, $\s(y_1)=1$ and $\T(y_1)\ge T$ and similarly
$y_2 \in [y_l,y_r]$ such that $u_0(y_2)=w_2$, $\s(y_2)=-1$ and $\T(y_2)\ge T$.

The proof in the two cases $y_1< y_2$ and $y_2<y_1$ differs only in some sign, therefore we only consider the case $y_1< y_2$.
Let $w\in \R$ be such that $w$ is not a value of local minimum or local maximum for $u_0$.
Let $t\in [0,T]$ and compute
\begin{equation*}
m(t,w):=\mathcal L^1(\{x:u(t,x)>w\}\cap [\X(t,y_1),\X(t,y_2)]).
\end{equation*}
By Proposition \ref{P_lagr_repr}, there exist $\bar y_1<\ldots < \bar y_{2k}$ such that 
\begin{equation*}
\{x:\sci u(t,x)>w\}= \bigcup_{i=1}^k(\X(t,\bar y_{2i-1}),\X(t,\bar y_{2i})).
\end{equation*}

Let us consider 
\begin{equation*}
I^+:=[y_1,y_2]\cap \s^{-1}(1), \quad I^-:=[y_1,y_2]\cap \s^{-1}(-1).
\end{equation*}
and let
\begin{equation*}
\lambda^\pm(t,w)=\sum_{y\in I^\pm\cap  u_0^{-1}(w)}\partial_t \X(t,y). 
\end{equation*}
For every $w$ the function $m(t,w)$ is Lipschitz with respect to $t$ because the characteristics are Lipschitz and for 
$\mathcal L^1$-a.e. $t\in (0,T)$
\begin{equation*}
\partial_t m(t,w)= \lambda^-(t,w)-\lambda^+(t,w)-\partial_t\X(t,y_1)\chi_{[0, w_1)}(w) +\partial_t\X(t,y_2)\chi_{[0, w_2)}(w),
\end{equation*}
where $\chi_E$ denotes the characteristic function of the set $E$.

Fix $w \in [0,\|u_0\|_\infty]$ which is not an extremal value; integrating with respect to $t$ we get
\begin{equation*}
\begin{split}
m(T,w)-m(0,w)= &~ \int_0^T\partial_t m(t,w) dt \\
 = &~  \int_0^T( \lambda^-(t,w)-\lambda^+(t,w))dt -\Delta \X_1\chi_{[0,w_1)}(w)  + \Delta \X_2\chi_{[0,w_2)}(w),
\end{split}
\end{equation*}
where for $i=1,2$
\begin{equation*}
\Delta \X_i=\X(T,y_i)-\X(0,y_i).
\end{equation*}
Integrating with respect to $w\in [0,\|u_0\|_\infty]$, we get
\begin{equation}\label{E_tw}
\begin{split}
\int_{0}^{\|u_0\|_\infty}(m(T,w)-m(0,w))dw = &~
\int_{0}^{\|u_0\|_\infty} \int_0^T(\lambda^-(t,w)-\lambda^+(t,w))dt dw \\
&~ + \Delta \X_2w_2-\Delta \X_1w_1.
\end{split}
\end{equation}

Now consider a fixed time $t\in [0,T]$. We claim that 
\begin{equation}\label{E_altitude}
-\int_{0}^{\|u_0\|_\infty}(\lambda^-(t,w)-\lambda^+(t,w))dw \ge f(w_2)-f(w_1).
\end{equation}
This follows by the fact that $\s(y_1)=1$ and $\s(y_2)=-1$. See Figure \ref{F_flusso} and Figure \ref{F_notation_sec3} to get a graphic intuition 
of the proof.
The solution $u$ at time $t$ is piecewise monotone so denote by $x_1<\ldots<x_k$ the local minimum and maximum points of
$u(t)$ in the interval $(\X(t,y_1),\X(t,y_2))$. For every $i=1,\ldots,k$ set $a_i=u(t,x_i)$ and let $a_0=u(t,\X(t,y_1)+)$, 
$a_{k+1}=u(t,\X(t,y_2)-)$.
Since $\s(y_1)=1$, by \eqref{E_sign}, it holds $a_1\ge w_1$ and similarly $a_{k+1}\le w_2$.
Therefore
\begin{equation*}
\begin{split}
-\int_{0}^{\|u_0\|_\infty} (\lambda^-(t,&w)-\lambda^+(t,w))dw  \\
= &~ (a_0-w_1)\partial_t\X(t,y_1)+ \sum_{i=0}^k\int_{a_i}^{a_{i+1}}f'(w)dw -(a_{k+1}-w_2)\partial_t\X(t,y_2) \\
= &~ (a_0-w_1)\partial_t\X(t,y_1) -f(a_0)+f(a_{k+1}) -(a_{k+1}-w_2)\partial_t\X(t,y_2).
\end{split}
\end{equation*}
By the chord admissibility condition and the characteristic equation, 
\begin{equation*}
(a_0-w_1)\partial_t\X(t,y_1) - f(a_0) \ge -f(w_1) \quad \text{and} \quad   f(a_{k+1})-(a_{k+1}-w_2)\partial_t\X(t,y_2) \ge f(w_2),
\end{equation*}
therefore we get \eqref{E_altitude}.

Integrating this relation with respect to $t$ we get
\begin{equation}\label{E_wt}
-\int_0^T\int_{0}^{\|u_0\|_\infty}(\lambda^-(t,w)-\lambda^+(t,w))dw dt\ge T(f(w_2)-f(w_1)).
\end{equation}

Comparing \eqref{E_tw} and \eqref{E_wt}:
\begin{equation*}
\begin{split}
T(\di (w_m,w_M) -\e)\le &~ T(f(w_2)-f(w_1)) - T\lambda (w_2-w_1) \\
\le &~ -\int_0^T\int_{0}^{\|u_0\|_\infty}(\lambda^-(t,w)-\lambda^+(t,w))dw dt -T\lambda  (w_2-w_1) \\
=&~ -\int_{0}^{\|u_0\|_\infty}(m(T,w)-m(0,w))dw + (\Delta \X_2-\Delta \X_1)w_1\\
 &~+(\Delta \X_2-\lambda T)(w_2-w_1) \\
\le &~ s \|u_0\|_\infty + s w_1 + s (w_2-w_1) \\
\le &~ 2s\|u_0\|_\infty.
\end{split}
\end{equation*}
Letting $\e\rightarrow 0$ we conclude the proof.
\end{proof}

\begin{figure}
\centering
\def\svgwidth{0.8\columnwidth}
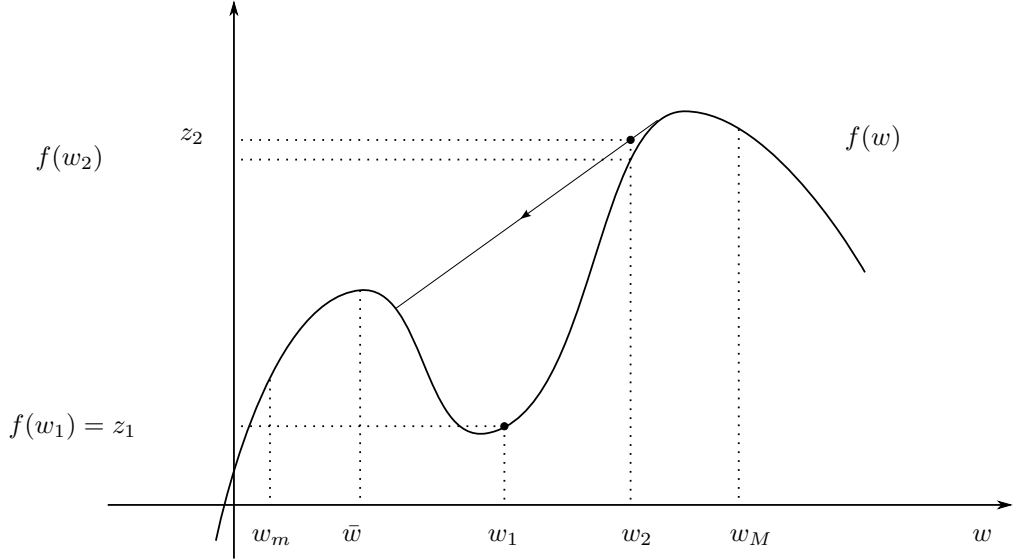
\caption{The flow $f$ and the secant denoting the shock at the point $\X(t,y_2)$.
The difference $z_2-z_1$ is equal to the l.h.s. in \eqref{E_altitude}; 
since $\s(y_2)=-1$ the secant passes above the graph of $f$, and similarly if there is a shock in $\X(t,y_1)$ it passes below.
Therefore $f(w_2)-f(w_1)\le z_2-z_1$ and this is \eqref{E_altitude}.}
\label{F_flusso}
\end{figure}

\begin{figure}
\centering
\def\svgwidth{0.9\columnwidth}
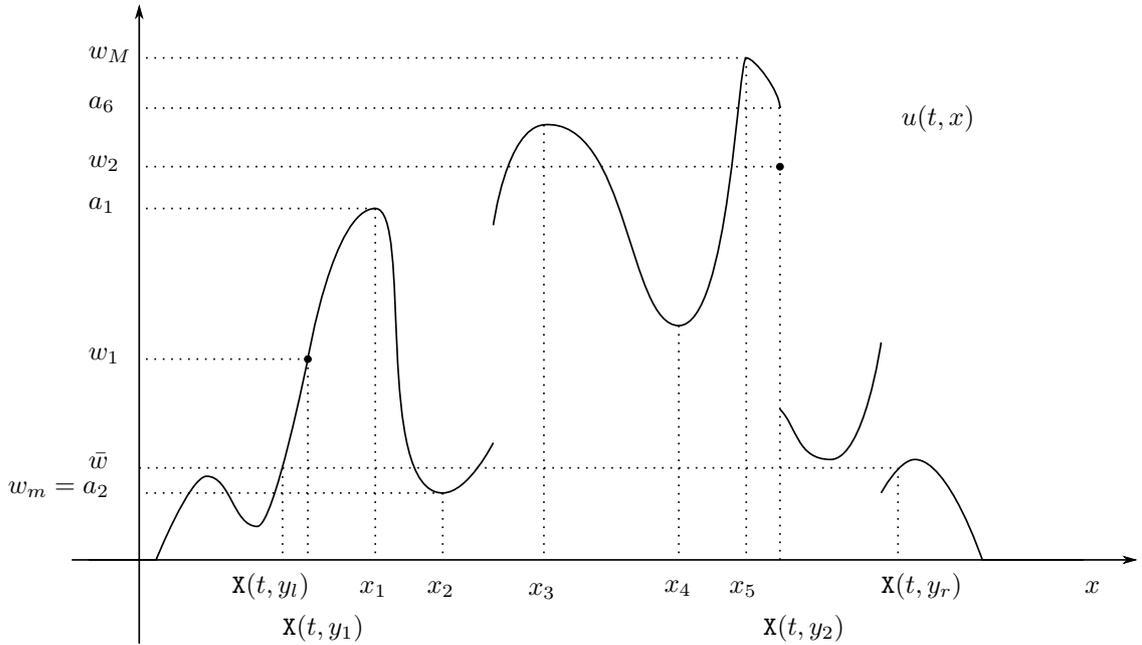
\caption{The graph of the solution $u(t)$ corresponding to the argument in Figure \ref{F_flusso} with the notation used in the proof of Theorem \ref{T_main_estimate}.}
\label{F_notation_sec3}
\end{figure}

\section{$\BV^\Phi$ regularity of the solution}
\label{S_frac_reg}

In this section we obtain the regularity of the entropy solution in terms of $\BV^\Phi$ spaces by means of 
Theorem \ref{T_main_estimate}. The definition of $\Phi$ depends on the nonlinearity of $f$. 
In particular we assume in this section that  $f$ is weakly genuinely nonlinear (see Definition \ref{D_weak_nonlinear}).

We also define $\di:\R^+\rightarrow [0,+\infty)$ by
\begin{equation}\label{E_def_di}
\di(h)=\inf_{a\in \R^+} \di(a,a+h),
\end{equation}
where $\di(a,a+h)$ is defined in \eqref{E_di}. This quantity quantifies the nonlinearity of the flux $f$. 
Since we consider bounded nonnegative solutions, the $\inf$ in \eqref{E_def_di} can be computed only on $[0,\|u\|_\infty-h]$; in that case it is a minimum and $\di(h)>0$ for every $h>0$ if and only if $f\llcorner [0,\|u_0\|_\infty]$ is weakly genuinely nonlinear.

Given a sequence $(h_n)_{n\in \N}$ with $h_n\ge 0$ for every $n\in \N$ and $h>0$, let
\begin{equation*}
N(h):= \#\{n: h_n\ge h\}.
\end{equation*}

\begin{lemma}\label{L_weak_l1}
Let $\Phi:[0,+\infty)\rightarrow [0,+\infty)$ be a convex function with $\Phi(0)=0$ and $\Phi>0$ in $(0,+\infty)$ such that for every $h>0$
\begin{equation}\label{E_weak_est}
N(h)\le \frac{1}{\Phi(h)}.
\end{equation}
Then, if we denote by $\bar h:=\max_n h_n$,  for every $\e>0$,
\begin{equation}\label{E_psi_e}
\sum_i \Psi_\e (h_i) \le \bar h^\e\frac{2^\e}{2^\e-1}, \qquad \mbox{where} \qquad \Psi_\e(x)=\Phi\left(\frac{x}{2}\right)x^\e.
\end{equation}
\end{lemma}
\begin{remark}
In the case with $\e=0$ you only get a weak $\ell^1$ estimate of the sequence $\Psi_	\e(h_i)$. Observe also that, by \eqref{E_weak_est}, $\bar h\le \Phi^{-1}(1)$.
\end{remark}
\begin{proof}
Since $\Psi_\e$ is increasing, for every $n\in \N$,
\begin{equation*}
N\left(\frac{\bar h}{2^{n+1}}\right)\Psi_\e\left(\frac{\bar h}{2^n}\right) \ge \sum_{i\in I_n} \Psi_\e(h_i),
\end{equation*}
where $I_n$ denotes the set of indexes $i$ for which $h_i\in (2^{-n-1}\bar h, 2^{-n}\bar h)$.
Finally
\begin{equation*}
\sum_i \Psi_\e(h_i)\le \sum_n N(2^{-n-1}\bar h)\Psi_\e(2^{-n}\bar h)\le \sum_n\frac{\Psi_\e(2^{-n}\bar h)}{\phi(2^{-n-1}\bar h)}= \bar h^\e\sum_n \left(\frac{1}{2^\e}\right)^n
= \bar h^\e\frac{2^\e}{2^\e-1}.
\end{equation*}
This concludes the proof of the lemma.
\end{proof}

We want to apply the lemma above to the height of the undulations of the entropy solution $u$. The existence of such a function $\Phi$
is proved in the following lemma as a corollary of Theorem \ref{T_main_estimate} in the case of weakly genuinely nonlinear fluxes.

\begin{lemma}\label{L_number_oscillations}
Let $u$ be the entropy solution of \eqref{E_cl} with $u_0\in X$ and let $t>0$. Then the number
$N(u(t),h)$ of undulations of $u(t)$ of height strictly bigger than $h>0$ is bounded by
\begin{equation*}
N(u(t),h)\le \frac{4\|u_0\|_\infty(\mathcal L^1(\conv(\supp u_0))+\|f'\|_\infty t)}{t\di(h)}.
\end{equation*}
\end{lemma}
\begin{proof}
The idea of the proof is the following: the measure of the support of an undulation of height bigger than $h$ is bounded from below by
Theorem \ref{T_main_estimate}. The inequality we want to prove states that the number of such undulations is bounded by the 
measure of the support of $u$ divided by the space occupied by each of them.
Actually the supports of the undulations are not disjoint in general and the proof consists in finding pairwise disjoint subsets of them
with the appropriate measure.

Denote for brevity by $N=N(u(t),h)$ and up to rearrangements we can assume that for $i=1,\ldots,N$ the undulations $u_i$ are the ones with height strictly bigger than $h$, 
let moreover
\begin{equation*}
\bar x_i=\min \arg\max u_i=\min \{x:u_i(x)=\max u_i\}
\end{equation*}
and let
\begin{equation*}
a_i = \sup\{x<\bar x_i : u(x)\le  u(t,\bar x_i)-h\}, \quad  b_i=\inf\{x>\bar x_i : u(x)\le  u(t,\bar x_i)-h\}.
\end{equation*}
It may happen  $a_i=\bar x_i$ or $ b_i=\bar x_i$, but it holds $ a_i< b_i$.
Moreover, since $h_i \ge h$, it holds 
\begin{equation}\label{E_supp_und}
( a_i, b_i)\subset \supp u_i.
\end{equation}

We claim that the intervals $( [a_i, b_i])_{i=1}^N$ are pairwise disjoint. 
Consider two undulations $u_i\ne u_j$ with $i,j=1,\ldots N$.
If $\supp u_i\cap \supp u_j$ has empty interior, then by \eqref{E_supp_und},
\begin{equation*}
(a_i,b_i) \cap (a_j,b_j)=\emptyset.
\end{equation*}

Suppose instead that $u_j$ is a descendant of $u_i$  and assume without loss of generality that $\bar x_j<\bar x_i$.
Then by point (3) in Proposition \ref{P_undulations}, $u(t,\bar x_i)\ge  u(t,\bar x_j)$, therefore
\begin{equation*}
\sci u(t,b_j)\le  u(t,\bar x_j)-h\le  u(t,\bar x_i)-h
\end{equation*}
and since $b_j\le \bar x_i$, by definition of $a_i$, it holds $b_j\le a_i$.
This proves that the intervals $((a_i,b_i))_{i=1}^N$ are pairwise disjoint. Finally we check that there exist no $i\ne j$ such that $a_i=b_j$.
In fact notice that by definition of $a_i$ the function $u(t)$ cannot have a decreasing jump at $a_i$ and similarly it cannot have an increasing jump
at $b_j$. In particular if $a_i=b_j$, it must be a point of continuity of $u(t)$ and therefore $u(t,a_i)= u(t,\bar x_i)-h= u(t,\bar x_j)-h$.
In particular by definition of $a_i$ and $b_j$ it holds $u(t,x)\ge u(t,\bar x_i)-h$ for every $x\in (\bar x_j,\bar x_i)$ and this is in contradiction
with the fact that both the undulations $u_i$ and $u_j$ have height strictly bigger than $h$.

By Proposition \ref{P_lagr_repr}, for every $i=1,\ldots,N$ there exists 
$y_i^-\in \R$ such that
\begin{equation*}
\X(t,y_i^-)=a_i, \quad u_0(y_i^-)=u^+(t,\bar x_i)-h\quad \mbox{and}\quad \T(y_i^-)\ge t. 
\end{equation*}
Similarly there exists $y_i^+\in \R$ such that
\begin{equation*}
\X(t,y_i^+)=b_i, \quad u_0(y_i^+)=u^+(t,\bar x_i)-h\quad \mbox{and}\quad \T(y_i^+)\ge t. 
\end{equation*}

For every $i=,\ldots, N$ apply Theorem \ref{T_main_estimate} with $y_l=y_i^-$ and $y_r=y_i^+$. Letting
\begin{equation*}
s_i:=\max\{ y^+_i- y^-_i, \X(t, y^+_i)-\X(t, y^-_i)\},
\end{equation*}
it holds
\begin{equation}\label{E_space}
s_i\ge \frac{t\di(h)}{2\|u_0\|_\infty}.
\end{equation}

Since the intervals $([a_i,b_i])_{i=1,\ldots,N}$ are pairwise disjoint, the same holds for the intervals $((y_i^-,y_i^+))_{i=1,\ldots,N}$
and for $((\X(t,y_i^-),\X(t,y_i^+)))_{i=1,\ldots, N}$ by monotonicity of the flow $\X$.

Moreover, by finite speed of propagation,
\begin{equation}\label{E_support}
\mathcal L^1(\conv(\supp u(t)))\le \mathcal L^1(\conv (\supp u_0))+ 2\|f'\|_\infty t,
\end{equation}
therefore, from \eqref{E_space}, \eqref{E_support} and the fact that we have disjoint intervals, it follows that
\begin{equation*}
N(u(t),h)\frac{t\di(h)}{2\|u_0\|_\infty}\le \sum_{i\in I_h}s_i\le2(\mathcal L^1(\conv(\supp u_0))+\|f'\|_\infty t),
\end{equation*}
and this concludes the proof.
\end{proof}

The main result of this section is the following:
\begin{theorem}\label{T_BV-Phi}
Let $u_0\in L^\infty(\R)$ be nonnegative with compact support and let $u$ be the entropy solution of \eqref{E_cl}. 
Let $\Phi$ be the convex envelope of $\di$, i.e. denote by
\begin{equation*}
\mathcal G =\{\varphi:[0,+\infty)\rightarrow [0,+\infty] \mbox{ convex }: \varphi(0)=0 \mbox{ and }\varphi(h)\le \di(h)\}
\end{equation*}
and let $\Phi=\sup_{\varphi\in \mathcal G} \varphi$. Then there exists a constant $C>0$ depending only on  
$\mathcal L^1(\conv(\supp u_0))$, $\|u_0\|_\infty$ and $\|f'\|_\infty$ such that 
for every $t>0$ and every $\e>0$ it holds
\begin{equation*}
u(t)\in \BV^{\Psi_\e} \qquad \mbox{and}\qquad \Psi_\e\mbox{-}\TV \, u(t)\le C\left(1+\frac{1}{t}\right)\frac{2^\e}{2^\e-1},
\end{equation*}
where $\Psi_\e$ is defined in \eqref{E_psi_e}.
\end{theorem}
\begin{proof}
Let $u_0\in L^\infty(\R)$ be nonnegative and with compact support. Since $X$ is dense in the space of nonnegative $L^\infty$ functions
with compact support with respect to the $L^1$-topology, there exists a sequence $(u_0^n)_{n\in \N}$ in $X$ such that 
$u_0^n\rightarrow u_0$ strongly in $L^1(\R)$.
Moreover we can also assume that for every $n\in \N$
\begin{enumerate}
\item $u_0^n$ is continuous;
\item it holds
\begin{equation*}
\conv(\supp u_0^n)\subset \conv(\supp u_0);
\end{equation*}
\item $\|u_0^n\|_\infty\le \|u_0\|_\infty$.
\end{enumerate}
By Lemma \ref{L_number_oscillations} and the choice of the approximation, there exists a constant $C'>0$ such that for every $n\in \N$ 
and for every $h>0$
\begin{equation}\label{E_merc}
N(u^n(t),h) \le C'\left(1+\frac{1}{t}\right)\frac{1}{\Phi(h)},
\end{equation}
therefore
\begin{equation*}
\begin{split}
\Psi_\e\mbox{-}\TV\, u^n(t)\le &~ 2\sum_{i=1}^{N(u(t))}\Psi^\e(h_i) \\
\le &~ C'\|u_0\|_\infty\left(1+\frac{1}{t}\right)\frac{2^\e}{2^\e-1},
\end{split}
\end{equation*}
where the first inequality holds by Lemma \ref{L_osc_phi} and the second one holds by \eqref{E_merc} and Lemma \ref{L_weak_l1}.
Finally, setting $C=C'\|u_0\|_\infty$, the result follows by lower semicontinuity of the $\Psi_\e$-total variation with respect to $L^1$ convergence.
\end{proof}

\begin{remark}
We give some comment on the previous result:
\begin{enumerate}
\item the regularity of $u$ depends crucially on the nonlinearity of $f$. Such dependence is encoded here in the condition
$\Phi(h)\le \di(h)$.
\item the upper bound for $\TV^{\Psi_\e}$ blows up as $t\rightarrow 0$, as we expect for $L^\infty$ entropy solutions;
\item in the case of $f$ of polynomial degeneracy $p\in \N$ (see Definition \ref{D_degeneracy}), it is not hard to prove that there exists $c>0$ such that for every $h>0$
\begin{equation*}
\di (h)\ge ch^{p+1}.
\end{equation*}
Therefore by Theorem \ref{T_BV-Phi} we get that for every $t>0$,
\begin{equation*}
u(t)\in \BV^{\frac{1}{p+1+\e}}(\R).
\end{equation*}
Relying on the $\BV$ regularity of $f'\circ u(t)$ (Section \ref{S_Cheng}), we will prove in Section \ref{S_frac_Cheng} that in this case the regularity of $u(t)$ can be improved to
\begin{equation*}
u(t)\in \BV^{\frac{1}{p}}(\R).
\end{equation*}
However in Section \ref{S_examples}, we prove that in general, even if $f$ is weakly nonlinearly degenerate, $f'\circ u\notin \BV_\loc(\R^+\times \R)$.
\end{enumerate}
\end{remark}

\section{$\BV$ regularity of $f'\circ u$}\label{S_Cheng}

In this section we prove that if the flux $f$ has finitely many inflection points of polynomial degeneracy (see Definition \ref{D_degeneracy}), then for every $T>0$ the velocity $f'\circ u(T)$ has bounded variation. 
In particular in this section we always assume that $f$ has degeneracy $p\in \N$.

We are going to prove a uniform estimate of $\TV (f'\circ u(T))$ for the entropy solutions $u$ of \eqref{E_cl} with $u_0\in X$ (defined in Section \ref{S_preliminaries})
and with $\|u_0\|_\infty$ and $\mathcal L^1 (\conv (\supp u_0))$  uniformly bounded. \\
The strategy is the following: we will consider separately small and big undulations of the solution $u$.
The number of big undulations is bounded a priori by Theorem \ref{T_main_estimate}. 
The contribution of small undulations is more delicate: if $u$ takes values in an interval where $f$ is convex, the
structure of characteristics is well-known and it implies a one-sided Lipschitz estimate on $f'\circ u$.
If instead $u$ oscillates around an inflection point of $f$ we adapt the argument of \cite{Cheng_speed_BV}.

We start by recalling the structure of characteristics in the convex case. We omit the proof of the following lemma that can be found in
\cite{Dafermos_book_4}: it can be proved either by means of Lax-Oleinik formula or with the method of generalized characteristics.

Let $0\le\bar t<T$ and let $\gamma_l,\gamma_r:[\bar t,T]\rightarrow \R$ be Lipschitz curves with $\gamma_l\le \gamma_r$ and consider the domain
\begin{equation}\label{E_Omega}
\Omega:=\{(t,x)\in (\bar t,T)\times \R: \gamma_l(t)<x<\gamma_r(t)\}.
\end{equation}

\begin{lemma}\label{L_convex}
Let $u$ be a piecewise monotone solution of \eqref{E_cl}.
Suppose that $u\llcorner \Omega$ takes values in $[u^-,u^+]$ and that $f\llcorner [u^-,u^+]$ is strictly convex. 
Then for every $x\in (\gamma_l(T),\gamma_r(T))$ there exists $\bar y$ and $t_0\in [\bar t,T)$ such that for every $t\in [t_0,T]$
\begin{equation*}
\X(t,\bar y)=x-(T-t)f'(u_0(\bar y))\qquad \mbox{and} \qquad \X(t_0,\bar y)\in \partial \Omega.
\end{equation*}
\end{lemma}

The characteristic structure of solutions with bounded variation when the flux has only one inflection point is studied in \cite{Dafermos_inflection}. 
We introduce some terminology and recall his result, then we translate his result in our language.

\begin{definition}
A \emph{generalized characteristic} of \eqref{E_cl} associated with the admissible $\BV$ solution $u$ is a Lipschitz trajectory
$\chi:[a,b]\rightarrow \R$, $0\le a<b<\infty$ such that for $\mathcal L^1$-a.e. $t\in [a,b]$
\begin{equation*}
\dot\chi(t)\in [f'(u(t,\chi(t)+)),f'(u(t,\chi(t)-))].
\end{equation*}
\end{definition}

\begin{definition}
A generalized characteristic $\chi:[a,b]\rightarrow \R$ of \eqref{E_cl}, associated with the admissible $\BV$ solution $u$, is called a
\emph{left contact} or a \emph{right contact} if 
\begin{equation*}
\dot\chi(t)=f'(u(t,\chi(t)-)) \qquad \mbox{or} \qquad \dot\chi(t)=f'(u(t,\chi(t)+))
\end{equation*}
for $\mathcal L^1$-a.e. $t\in [a,b]$ respectively.
\end{definition}
The existence of generalized characteristics is granted by Filippov theory. 
In general uniqueness fails, in the following two theorems, whose proof can be found in \cite{Dafermos_inflection}, it is described the structure 
of maximal and minimal backward characteristics respectively (see also \cite{Dafermos_book_4}, Section 11.12).

As the author did in \cite{Dafermos_inflection}, we assume in Theorems \ref{T_daf_max} and \ref{T_daf_min} that:
\begin{enumerate}
\item it holds $f(0)=f'(0)=f''(0)=0$;
\item it holds $uf''(u)<0$ for every $u\ne 0$;
\item the function $f''$ is nonincreasing in a neighborhood of $0$.
\end{enumerate}

\begin{theorem}\label{T_daf_max}
Let $\xi$ denote the maximal backward characteristic through any point $(\bar t,\bar x)\in (0,+\infty)\times \R$.
When $u(\bar t,\bar x-)\ne 0$ or $u(\bar t,\bar x+)\ne 0$, then there is a finite mesh $0=a_0<a_1<\ldots<a_{N+1}=\bar t$ such that
$\xi$ is a convex polygonal line with vertices at the point $(a_n,\xi(a_n))$, $n=0,\ldots,N+1$. Furthermore,
\begin{equation}\label{E_daf_max}
\begin{split}
u(t,\xi(t)-)=u(t,\xi(t)+)=u(a_{n+1},\xi(a_{n+1})+), & \quad  a_n<t<a_{n+1}, \quad n=0,\ldots, N, \\
u(a_n,\xi(a_n)-)=u(a_{n+1},\xi(a_{n+1})+), &\quad  n=1,\ldots, N, \\
u_0(\xi(0))\ge u(a_1,\xi(a_1)+),  &\quad \mbox{if }u(a_1,\xi(a_1)+)>0, \\
u_0(\xi(0))\le u(a_1,\xi(a_1)+), &\quad  \mbox{if }u(a_1,\xi(a_1)+)<0, \\
\dot\xi(t)=f'(a_{n+1},u(\xi(a_{n+1})+)), &\quad  a_n<t<a_{n+1}, \quad n=0,\ldots,N, \\
f'(u(a_n,\xi(a_n)-))=\frac{f(u(a_n,\xi(a_n)+))-f(u(a_n,\xi(a_n)-))}{u(a_n,\xi(a_n)+)-u(a_n,\xi(a_n)-)}, & \quad n=1,\ldots,N.
\end{split}
\end{equation}
When $u(\bar t,\bar x-)=u(\bar t,\bar x+)=0$, then there is $a\in [0,\bar t]$ such that $\xi(t)=\bar x$, $t\in [a,\bar t]$, and
$u(t,\xi(t)-)=u(t,\xi(t)+)=0$, $t\in (a,\bar t]$ (also at $t=a$ if $a>0$). 
Moreover, if $a>0$, there is an increasing sequence $0=a_0<a_1<\ldots$ with $a_n\rightarrow a$ and $\xi(a_n)\downarrow \bar x$ as
$n\rightarrow \infty$, such that \eqref{E_daf_max} all hold for $n=1,2,\ldots$. In particular,
\begin{equation*}
|u(t,\xi(t))|\downarrow 0, \quad f'(u(t,\xi(t)))\uparrow 0, \quad \mbox{as }t\uparrow a.
\end{equation*}
\end{theorem}

\begin{theorem}\label{T_daf_min}
Let $\zeta$ denote the minimal backward characteristic through any point $(\bar t,\bar x)\in (0,+\infty)\times \R$. Then $u(t,\zeta(t)-)$ is 
a continuous function on $(0,\bar t]$, which is nondecreasing when $u(\bar t,\bar x-)<0$, nonincreasing when $u(\bar t,\bar x-)>0$ and
constant equal to 0 when $u(\bar t,\bar x-)=0$. For $t\in (0,\bar t)$,
\begin{equation*}
\dot\zeta(t)=f'(t,\zeta(t)-)
\end{equation*}
so, in particular, $\zeta$ is a convex $C^1$ curve. Furthermore, the interval $(0,\bar t)$ is decomposed into the union of two disjoint subset $O$ and $C$ with the following properties:
$O$ is the (at most) countable union of pairwise disjoint open intervals, $O=\bigcup_n(\alpha_n,\beta_n)$, such that 
\begin{equation*}
u(t,\zeta(t)-)=u(t,\zeta(t)+)=u(\alpha_n,\zeta(\alpha_n))=u(\beta_n,\zeta(\beta_n))
\end{equation*}
for all $t\in (\alpha_n,\beta_n)$ so the restriction of $\zeta$ on $(\alpha_n,\beta_n)$ is a straight line with slope 
$f'(u(\alpha_n,\zeta(\alpha_n)-))$. For any point $t\in C$, $u(t,\zeta(t)-)\ne u(t,\zeta(t)+)$ and 
\begin{equation*}
f'(u(t,\zeta(t)-))=\frac{f(u(t,\zeta(t)+))-f(u(t,\zeta(t)-))}{u(t,\zeta(t)+)-u(t,\zeta(t)-)}.
\end{equation*}
\end{theorem}

Now we restrict our attention to the case of piecewise monotone initial data and we formulate Theorem \ref{T_daf_max} in terms of the
Lagrangian representation. 

\begin{figure}
\centering
\def\svgwidth{0.7\columnwidth}
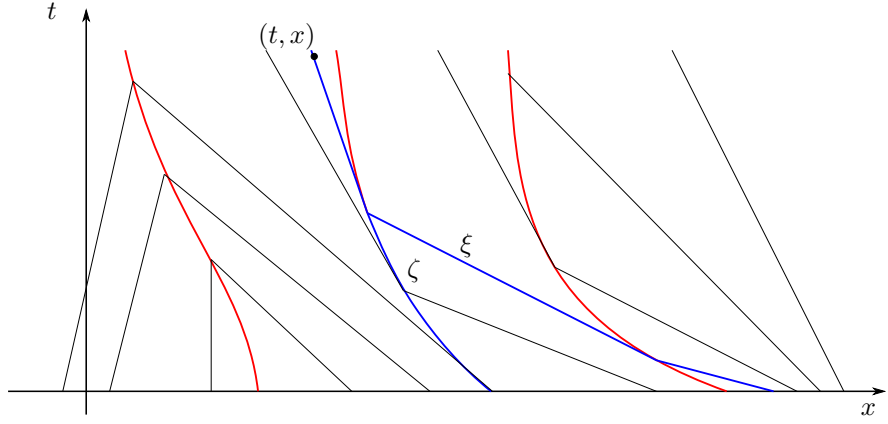
\caption{The structure of the characteristic curves: the minimal and maximal backward characteristic $\zeta$ and $\xi$ are blue and a shock and two left contact discontinuities are red.}\label{F_inflection}
\end{figure}

\begin{lemma}\label{L_inflection}
Let $u$ be the entropy solution of \eqref{E_cl} with $u_0\in X$ and let $\Omega$ be as in \eqref{E_Omega}.
Suppose that $u\llcorner \Omega$ takes values in $[u^-,u^+]$ with $\bar w\in (u^-,u^+)$,
\begin{equation*}
f''(w)(w-\bar w)<0 \quad \mbox{in}\quad [u^-,u^+]\setminus \{\bar w\}
\end{equation*}
and that $f''$ is nonincreasing in a neighborhood of $\bar w$.
Then for every $\bar x\in (\gamma_l(T),\gamma_r(T))$ the maximal backward generalized characteristic $\xi_{\bar x}$ from $(T,\bar x)$ enjoys the
following properties: there exists $N=N(\bar x)\in\N$, $\bar t\le t_0<t_1<\ldots <t_N=T$ and $y_1>\ldots>y_N$ such that
\begin{enumerate}
\item for every $n=1,\ldots,N$, for every $t\in [t_{n-1},t_n]$
\begin{equation*}
\xi_{\bar x}(t)=\X(t,y_n)=\X(t_{n-1})+(t-t_{n-1})f'(u_0(y_n)),
\end{equation*}
in particular $\xi_{\bar x}$ is piecewise affine;
\item for every $t\in (t_0,T]$, $\xi_{\bar x}(t)\in (\gamma_l(t),\gamma_r(t))$ and $(t_0,\xi_{\bar x}(t_0))\in \partial \Omega$;
\item for every $n=1,\ldots,N$
\begin{equation*}
\begin{split}
u_0(y_n)=u(t,\X(t,y_n)-) &\quad \mbox{for every }t\in (t_{n-1},t_n], \\
u_0(y_n)=u(t,\X(t,y_n)+) &\quad \mbox{for every }t\in [t_{n-1},t_n)\setminus \{t_0\};
\end{split}
\end{equation*}
\item for every $n=2,\ldots,N$
\begin{equation}\label{E_velocity_contact}
u_0(y_n)=u_0(y_{n-1})^*,
\end{equation}
where $u_0(y_{n-1})^*$ is defined by Lemma \ref{L_conjugate};
\item for every $\bar x_1,\bar x_2 \in (\gamma_l(T),\gamma_r(T))$ with $\bar x_1<\bar x_2$ it holds
\begin{equation*}
\xi_{\bar x_1}(t)< \xi_{\bar x_2}(t) \qquad \forall \, t\in (t_0(\bar x_1)\vee t_0(\bar x_2), T].  
\end{equation*}
\end{enumerate}
Moreover if $u(T,\bar x -)=u(T,\bar x+)=\bar w$, then the conditions above hold with $N=1$: in particular $\xi_{\bar x}\llcorner (t_0,T)$ 
has constant velocity.
\end{lemma}
\begin{proof}
We first observe that for every $n=1,\ldots, N$ there exists $y_n$ such that $\xi_{\bar x}\llcorner (t_{n-1},t_n) = \X(\cdot,y_n)\llcorner (t_{n-1},t_n)$.
Let $t\in (t_{n-1},t_n)$ and $y_n$ be such that $\X(t,y_n)=\xi_{\bar x}(t)$ and $\T(y_n)\ge t$. Then $u_0(y_n)=u(t,\xi_{\bar x}(t))$ and, since 
$\X(\cdot, y_n)$ satisfies the characteristic equation and $u$ is continuous, 
it holds $\xi_{\bar x}\llcorner (t_{n-1},t_n) = \X(\cdot,y_n)\llcorner (t_{n-1},t_n)$.
By monotonicity of the flow $\X$, the maximality of $\xi_{\bar x}$ and the fact that for every $n=2,\ldots, N$,
\begin{equation*}
u_0(y_n)=u(t_{n-1},\X(t_n,y_n)+) \ne u(t_{n-1},\X(t_n,y_n)-)=u(y_{n-1})
\end{equation*}
we have $y_n>y_{n-1}$.
Observe that, by \eqref{E_velocity_contact}, the value $u_0(y_{n-1})$ is uniquely determined by $u_0(y_n)$ and in particular 
\begin{equation}\label{E_change_sign}
(u_0(y_{n-1})-\bar w)(u_0(y_n)-\bar w)<0.
\end{equation}
Since the initial datum is piecewise monotone and $n\mapsto y_n$ is strictly decreasing, \eqref{E_change_sign} implies that $N$ is bounded by the 
number of monotone regions of the initial datum. In particular if $u(T,\bar x -)=u(T,\bar x+)=0$, the existence of a sequence as in 
Theorem \ref{T_daf_max} is excluded.
It remains to prove the monotonicity in (5):
Let $t \in (t_0(\bar x_1)\vee t_0(\bar x_2), T]$ the maximal time such that $\xi_{\bar x_1}(t)= \xi_{\bar x_2}(t)$. By monotonicity of the flow and
since the maximal characteristics have piecewise constant speed, the point $(t,\xi_{\bar x_1}(t))$ must belong to a left-discontinuity curve. 
Since the left-discontinuity curve has $C^1$ regularity it holds $\partial_t\xi_{\bar x_1}(t+)=\partial_t \xi_{\bar x_2}(t)$ and this implies that $t=T$.
But this is in contradiction with the hypothesis $\bar x_1<\bar x_2$ and this concludes the proof.
\end{proof}

In the following two propositions, we deduce by the structure of the characteristics an estimate of the total variation of $f'\circ u(T)$ in the two cases of a convex flux or of a flux with an inflection point with polynomial degeneracy.

\begin{proposition}\label{P_convex}
Let $u$ be the entropy solution of \eqref{E_cl} with $u_0\in X$; let $\bar t,T,\gamma_l,\gamma_r$ and $\Omega$ be defined as in \eqref{E_Omega}.
Assume that there exists $a,b\in [0,\|u_0\|_\infty]$ such that $u\llcorner \Omega$ solves the initial-boundary
value problem
\begin{equation*}
\begin{cases}
u_t+f(u)_x=0 & \mbox{in }\Omega, \\
u(t,\gamma_l(t))=a & \mbox{for }t\in (\bar t,T), \\
u(t,\gamma_r(t))=b & \mbox{for }t\in (\bar t ,T).
\end{cases}
\end{equation*}
Denote by
\begin{equation*}
I:=\conv \big( \{a,b\}\cup u(\bar t,(\gamma_l(\bar t),\gamma_r(\bar t)))\big) 
\end{equation*}
and assume moreover that $f\llcorner I$ is strictly convex.
Then
\begin{equation*}
\TV_{(\gamma_l(T),\gamma_r(T))} (f'\circ u(T)) \le 6\|f''\|_\infty\|u_0\|_\infty + 2 \frac{\gamma_r(T)-\gamma_l(T)}{T-\bar t}.
\end{equation*}
\end{proposition}
\begin{proof}
Let $\X$ be a Lagrangian representation of $u$ and consider the following decomposition: 
\begin{equation*}
(\gamma_l(T),\gamma_r(T))=A_l\cup A_m\cup A_r,
\end{equation*}
where 
\begin{equation*}
\begin{split}
A_l:=&~\left\{x\in (\gamma_l(T),\gamma_r(T)): \exists t_0\in [\bar t,T), \exists \bar y\in \R 
\big(\X(t_0,\bar y)=\gamma_l(t_0), \X(T,\bar y)=x\big)\right\}, \\
A_r:=&~ \left\{x\in (\gamma_l(T),\gamma_r(T)): \exists t_0\in [\bar t,T), \exists \bar y\in \R 
\big(\X(t_0,\bar y)=\gamma_r(t_0), \X(T,\bar y)=x\big)\right\}, \\
 A_m:=&~(\gamma_l(T),\gamma_r(T))\setminus (A_l\cup A_r).
\end{split}
\end{equation*}
By monotonicity and continuity of the flow $\X$ with respect to $y$, there exist $x_l,x_r\in [\gamma_l(T),\gamma_r(T)]$ such that
\begin{equation*}
A_l=(\gamma_l(T),x_l] \qquad \mbox{and} \qquad A_r=[x_r\gamma_r(T)).
\end{equation*}
Observe that it may be $x_r\le x_l$; in that case  $A_m=\emptyset$.

Assume $A_l$ is nonempty and let $\bar y\in\R$ and $t_0\in [\bar t, T)$ be such that
\begin{equation*}
\X(T,\bar y)= x_l \qquad \mbox{and} \qquad \X(t_0,\bar y)=\gamma_l(t_0).
\end{equation*}
Moreover let
\begin{equation}\label{E_def_y-w-}
y^-:=\max\{y:\X(T,y)=\gamma_l(T)\} \qquad \mbox{and} \qquad \bar w^- := \lim_{x\rightarrow \gamma_l(T)^+}u(T,x),
\end{equation}
and similarly
\begin{equation*}
y^+:=\min\{y:\X(T,y)=x_l\} \qquad \mbox{and} \qquad \bar w^+ := \lim_{x\rightarrow x_l^-}u(T,x).
\end{equation*}
Denote by
\begin{equation*}
\Omega_l:=\{(t,x) \in (\bar t,T)\times \R : \X(t,y^-)<x<\X(t,y^+)\}.
\end{equation*}
By definition of $y^-$, $y^+$ and the monotonicity of $\X$ with respect to $y$ there exists $t_0\in [\bar t, T)$ such that
$\X(t_0,y^-)=\X(t_0,y^+)=\X(t_0,y_j)$.
Since the limit of admissible boundaries is an admissible boundary in the sense of Proposition \ref{P_stab}, 
$(\X(\cdot,y^-),\bar w^-)$ and $(\X(\cdot,y^+),\bar w^+)$ are admissible boundaries of $u$ for $t\in [0,T)$.
Therefore the restriction $u(T)\llcorner (\gamma_l(T),x_l)$ is the entropy solution at time $T$ of the boundary value problem
\begin{equation*}
\begin{cases}
u_t+f(u)_x=0 & \mbox{in }\Omega_l, \\
u(t,\X(t,y^-)=\bar w^- & \mbox{for }t\in (t_0,T), \\
u(t,\X(t,y^+)=\bar w^+ & \mbox{for }t\in (t_0,T).
\end{cases}.
\end{equation*}
By Proposition \ref{P_TV_est}, this implies that $u(T)\llcorner (\gamma_l(T),x_l)$ is monotone, therefore
\begin{equation}\label{E_est_left}
\TV_{(\gamma_l(T),x_l)}(f'\circ u(T))\le \|f''\|_\infty\|u_0\|_\infty.
\end{equation}
Similarly we can prove that 
\begin{equation}\label{E_est_right}
\TV_{(x_r,\gamma_r(T))}(f'\circ u(T))\le \|f''\|_\infty\|u_0\|_\infty,
\end{equation}
therefore it remains to estimate the total variation on $A_m$. Assume that $A_m\ne \emptyset$ i.e. $x_l<x_r$. This case is well-known, we take 
advantage of the fact that by Lemma \ref{L_convex}, the characteristics starting from $x \in A_m$ are segments in $[\bar t,T]$ and we deduce a one-sided Lipschitz estimate.
Denote by 
\begin{equation*}
\bar y^-:=\max\{y:\X(T,y)=x_l\} \qquad \mbox{and} \qquad \bar y^+:=\max\{y:\X(T,y)=x_r\}.
\end{equation*}
By Lemma \ref{L_convex}, for every $x\in A_m$ there exists $ y(x)\in (\bar y^-,\bar y^+)$ such that for every $t\in [\bar t,T]$, it holds
\begin{equation*}
\X(t,y(x))=x-f'(u(T,x))(T-t).
\end{equation*}
By monotonicity of the flow, for every $x_l<x_1<x_2<x_r$, it holds
\begin{equation*}
x_1-f'(u(T,x_1))(T-\bar t) = \X(\bar t, y(x_1)) \le \X(\bar t, y(x_2)) =  x_2-f'(u(T,x_2))(T-\bar t)
\end{equation*}
which gives the one-sided Lipschitz estimate
\begin{equation*}
f'(u(T,x_2))\le f'(u(T,x_1)) + \frac{x_2-x_1}{T-\bar t}.
\end{equation*}
This implies that the positive total variation
\begin{equation*}
\TV^+_{(x_l,x_r)}(f'\circ u(T)) \le \frac{x_r-x_l}{T-\bar t}.
\end{equation*}
Hence, by Lemma \ref{L_trivial}, the whole total variation can be estimate by
\begin{equation}\label{E_est_middle}
\TV_{(x_l,x_r)}(f'\circ u(T)) \le 2 \TV^+_{(x_l,x_r)}(f'\circ u(T)) + 2\|f''\|_\infty\|u_0\|_\infty \le 2\frac{x_r-x_l}{T-\bar t} +  2\|f''\|_\infty\|u_0\|_\infty.
\end{equation}
Adding \eqref{E_est_left}, \eqref{E_est_right}, \eqref{E_est_middle} and taking into account the possible jumps of $f'\circ u(T)$ at the points 
$x_l$ and $x_r$ we get
\begin{equation*}
\TV_{(\gamma_l(T),\gamma_r(T))}(f'\circ u(T)) \le 6\|f''\|_\infty\|u_0\|_\infty + 2 \frac{\gamma_r(T)-\gamma_l(T)}{T-\bar t},
\end{equation*}
that is the claimed estimate.
\end{proof}

The case of a flux with an inflection point is more elaborate and it is based on the structure of maximal characteristics.
\begin{proposition}\label{P_inflection}
Let $u$ be the entropy solution of \eqref{E_cl} with $u_0\in X$; let $\bar t,T,\gamma_l,\gamma_r$ and $\Omega$ be defined as in \eqref{E_Omega}.
Assume that $\bar t,T>0$ are generic (Definition \ref{D_generic}) and that there exists $a,b\in [0,\|u_0\|_\infty]$ such that $u\llcorner \Omega$
solves the initial-boundary value problem
\begin{equation*}
\begin{cases}
u_t+f(u)_x=0 & \mbox{in }\Omega, \\
u(t,\gamma_l(t))=a & \mbox{for }t\in (0,T), \\
u(t,\gamma_r(t))=b & \mbox{for }t\in (0,T).
\end{cases}
\end{equation*}
Denote by
\begin{equation*}
I:=\conv \big( \{a,b\}\cup u(\bar t,(\gamma_l(\bar t),\gamma_r(\bar t)))\big).
\end{equation*}
Assume moreover that there exists a unique inflection point $\bar w\in I$ of $f$ and that $\bar w$ has degeneracy $p\in \N$.
Let $\delta',\e'>0$ be given by Lemma \ref{L_conjugate} and assume finally that $I\subset (\bar w-\delta',\bar w+\delta')$.

Then there exists a constant $C>0$ depending on $\e', \|f'\|_\infty, \|f''\|_\infty, \mathcal L^1(\conv (\supp u_0)),\|u_0\|_\infty$ such that
\begin{equation}\label{E_est_inflection}
\TV_{(\gamma_l(T),\gamma_r(T))} (f'\circ u(T)) \le C\left(1+\frac{1}{T-\bar t}\right).
\end{equation}
\end{proposition}
\begin{proof}
The structure of the proof of this proposition is similar to the one of Proposition \ref{P_convex}. 
Here we reach the final estimate studying the behavior of maximal backward characteristics.
Let $(\X,\T)$ be a Lagrangian representation of $u$ and for every $x\in (\gamma_l(T),\gamma_r(T))$ let $\xi_x$ be the maximal backward
characteristic from $(T,x)$. Consider moreover the corresponding $y_1(x),\ldots y_{N(x)}$ and $t_0(x),\ldots, t_{N(x)}$ given by Lemma
\ref{L_inflection}.
Consider the decomposition 
\begin{equation*}
(\gamma_l(T),\gamma_r(T))=A_l\cup A_m \cup A_r,
\end{equation*}
where
\begin{equation*}
\begin{split}
A_l:=&~\left\{\bar x\in (\gamma_l(T),\gamma_r(T)): \xi_{\bar x}(t_0(\bar x))=\gamma_l(t_0(\bar x)) \right\}, \\
A_r:=&~ \left\{\bar x\in (\gamma_l(T),\gamma_r(T)): \xi_{\bar x}(t_0(\bar x))=\gamma_r(t_0(\bar x)) \right\}, \\
A_m:=&~(\gamma_l(T),\gamma_r(T))\setminus (A_l\cup A_r).
\end{split}
\end{equation*}

Let $\bar x\in A_l$ and set 
\begin{equation*}
\Omega_l:=\left\{(t,x)\in (t_0(\bar x),T]\times \R: \gamma_l(t)<x<\xi_{\bar x}(t)\right\}.
\end{equation*}
Then $u(T)\llcorner (\gamma_l(T),\bar x)$ is the entropy solution at time $T$ of the boundary value problem
\begin{equation*}
\begin{cases}
u_t+f(u)_x=0 & \mbox{in }\Omega_l, \\
u(t,\gamma_l(t))=a & \mbox{for }t\in (t_0,T), \\
u(t,\xi_{\bar x}(t))=u^+(t) & \mbox{for }t\in (t_0,T),
\end{cases}
\end{equation*}
where $u^+(t)=u(t,\xi_{\bar x}(t))$, and this definition makes sense since by Lemma \ref{L_inflection}, for every $t\notin \{t_i\}_{i=1}^n$  
the solution $u$ is continuous at $(t,\xi_{\bar x}(t))$.
By Proposition \ref{P_TV_est}, it holds
\begin{equation}\label{E_tvleft}
\TV_{(\gamma_l(T),\bar x)}(u(T))\le \TV_{(t_0,T)} (u^+) + |u^+(t_0+)-a|,
\end{equation}
and by Lemma \ref{L_inflection}, we have that
\begin{equation*}
\TV_{(t_0,T)} (u^+) = \sum_{n=2}^{N(\bar x)}|u_0(y_n)-u_0(y_{n-1})| \le 2\sum_{n=1}^{N(\bar x)}|u_0(y_n)-\bar w|.
\end{equation*}
Since for every $n=2,\ldots, N(\bar x)$, it holds $u_0(y_n(\bar x))=u_0(y_{n-1}(\bar x))^*$, by Lemma \ref{L_conjugate} and \eqref{E_velocity_contact}, it holds
\begin{equation*}
|u_0(y_n(\bar x))-\bar w|\le |u_0(y_1(\bar x))-\bar w|(1-\e')^{h-1}.
\end{equation*}
So we finally have that
\begin{equation}\label{E_tvu+}
\TV_{(t_0,T)} (u^+) \le 2|u_0(y_1(\bar x))-\bar w|\sum_{n=1}^{N(\bar x)}(1-\e')^{n-1} \le \frac{2\|u_0\|_\infty}{\e'}.
\end{equation}
By \eqref{E_tvleft} and \eqref{E_tvu+}, we get
\begin{equation*}
\TV_{(\gamma_l(T),\bar x)}(u(T))\le \|u_0\|_\infty \left(1+\frac{2}{\e'}\right),
\end{equation*}
and, since the estimate is independent of $\bar x\in \Int A_l$, it holds
\begin{equation*}
\TV_{\Int A_l}(u(T))\le  \|u_0\|_\infty \left(1+\frac{2}{\e'}\right).
\end{equation*}
Therefore it immediately follows that
\begin{equation}\label{E_est_A_l}
\TV_{\Int A_l}(f'\circ u(T))\le \|f''\|_\infty \|u_0\|_\infty \left(1+\frac{2}{\e'}\right),
\end{equation}
and the same argument proves that
\begin{equation}\label{E_est_A_r}
\TV_{\Int A_r}(f'\circ u(T))\le \|f''\|_\infty \|u_0\|_\infty \left(1+\frac{2}{\e'}\right).
\end{equation}

It remains to prove the estimate in $A_m$.
Again we take advantage of the fact that the generalized characteristics $\xi_{\bar x}$ for $\bar x\in \Int A_m$ are defined on the whole time interval
$[\bar t,T]$. 
The estimate is obtained by partitioning $(x_l,x_r)=\Int A_m$ in regions where the maximal characteristics of each region cross the same set of minimal 
backward characteristic, bounding $\TV^+ (f'\circ u)$ on each of these regions and adding them.

Let
\begin{equation*}
y^-:=\max\{y: \X(T,y)=\inf A_m\} \qquad \mbox{and} \qquad y^+:=\min\{y: \X(T,y)=\sup A_m\}.
\end{equation*}
Since $u_0$ is piecewise monotone, there exist $L \in \N$ and $y^-=\bar y_0<\ldots< \bar y_L=y^+$ such that
\begin{enumerate}
\item for every $l=1,\ldots, L$,
\begin{equation*}
\bar y_l\in \{y:\T(y)\ge \bar t\} \qquad \mbox{and} \qquad u_0(\bar y_l)=\bar w.
\end{equation*}
\item the function $u_0$ alternates the sign on $((y_l,y_{l+1}))_{l=1}^{L-1}$, i.e. the function 
\begin{equation*}
\sum_{l=1}^{L-1}(-1)^lu_0\llcorner \{y\in(\bar y_l,\bar y_{l+1}): \T(y)\ge \bar t\}
\end{equation*}
has constant sign; without loss of generality we assume that it is nonnegative.
\item  for every $l=1,\ldots,L-2$, there exist $y'\in \{y\in(\bar y_l,\bar y_{l+1}): \T(y')\ge \bar t\}$ and $y''\in \{y\in(\bar y_{l+1},\bar y_{l+2}): \T(y)\ge \bar t\}$
such that $u_0(y')u_0(y'')<0$.
\end{enumerate}
For every $x\in(x_l,x_r)$ and for every $n=1,\ldots, N(x)$, let $l(x,n)$ be the unique value in $\{1,\ldots,L\}$ such that 
\begin{equation*}
y_n(x)\in[\bar y_{l(x.n)},\bar y_{l(x_n)+1}) \qquad \mbox{and let} \qquad \el(x):=\{l(x,n):n=1,\ldots,N(x)\}.
\end{equation*}
For every $\el \in  \mathcal P (\{1,\ldots,L-1\})$, let
\begin{equation*}
A(\el):=\{x\in \Int A_m : \el(x)=\el\}.
\end{equation*}
Clearly it holds
\begin{equation*}
\bigcup_{\el \in \mathcal P({1,\ldots,L})} A(\el)=(x_l,x_r),
\end{equation*}
now we check that for every $\el \in \mathcal P (\{1,\ldots,L\})$ the set $A(\el)$ is an interval.

In order to do this let us introduce a partial ordering on $\mathcal P (\{1,\ldots,L-1\})$:
we say that $\el_1 \preceq \el_2$ if
\begin{enumerate}
\item $\min \el_1\le \min \el_2$;
\item $\max \el_1\le \max \el_2$;
\item for every $j\in [\min \el_2, \max \el_1]$,
\begin{equation*}
j\in \el_2 \quad \Rightarrow \quad j\in \el_1.
\end{equation*}
\end{enumerate}
It is standard to check that $\preceq$ is a partial ordering, so in order to prove that $A(\el)$ are intervals, it suffices to prove that for every 
$x_1,x_2\in (x_l,x_r)$ it holds
\begin{equation*}
x_1<x_2 \qquad \Longrightarrow \qquad  \el(x_1) \preceq \el (x_2).
\end{equation*}
The conditions (1) and (2) of the definition of $\preceq$ immediately follow from the monotonicity of $x\mapsto \xi_x$ (Point (5) of Lemma \ref{L_inflection}).
Finally by Proposition \ref{P_lagr_repr}, it follows that if $\X(t',\bar y_{l_1})=\X(t', \bar y_{l_2})$ for some $t'\in (\bar t, T)$, then for every $t\in [t',T]$
it holds $\X(t,\bar y_{l_1})=\X(t, \bar y_{l_2})$ and, by Point (5) in Lemma \ref{L_inflection}, this implies that if $j\in [\min \el_2, \max \el_1]$ is such that
$j\notin \el_1$. Then $j\notin \el_2$ and this proves condition (3) in the definition of $\preceq$.

\textbf{Claim 1.} There exist $V\in \N$ and $x_l=\bar x_1<\ldots<\bar x_V=x_r$ such that for every $v=1,\ldots, V-1$ there exists $\el(v) \in \mathcal P (\{1,\ldots,L-1\})$ such that
\begin{equation}\label{E_Aelv}
(\bar x_v,\bar x_{v+1})\subset A(\el(v)).
\end{equation}
Moreover if $v= 1,\ldots, V-1$ is such that $\#\el(v)\ge 2$, then
\begin{enumerate}
\item for every $x_1,x_2\in (\bar x_v,\bar x_{v+1})$ it holds
\begin{equation}\label{E_small_gap_t}
\sum_{n=1}^{\#\el(v)}|t_n(x_2)-t_n(x_1)| < \frac{T-\bar t}{2}.
\end{equation}
\item the velocity $f'\circ u(T)$ is strictly increasing in $(\bar x_v,\bar x_{v+1})$;
\end{enumerate}

\emph{Proof of Claim 1.}
Let
\begin{equation*}
\{\tilde x_1, \ldots, \tilde x_Q\}:= \{\inf A(\el): \el \in \mathcal P(\{1,\ldots, L-1\}) \mbox{ and }A(\el)\ne \emptyset\}\cup\{x_l,x_r\}
\end{equation*}
with $\tilde x_1<\ldots, \tilde x_Q$. Since $(A(\el))_{\el\in\mathcal P(\{1,\ldots, L-1)}$ is a family of pairwise disjoint intervals, the sequence 
$(\tilde x_q)_{q=1}^Q$ satisfies \eqref{E_Aelv}. \newline 
Therefore, in order to prove (1), it is sufficient to prove that for every $q=1,\ldots, Q-1$ there exists $V'\in \N$ and $\tilde x_q=x_1'<\ldots<x_{V'}'=\tilde x_{q+1}$ 
such that condition (2) holds for every $v'=1\ldots V'-1$.
Let $q\in \{1,\ldots, Q-1\}$ and let $\el \in \mathcal P(\{1,\ldots, L-1\})$ be such that
\begin{equation*}
(\tilde x_q, \tilde x_{q+1})\subset A(\el).
\end{equation*}
Since for every $x\in \Int A(\el)$ and for every $n=1,\ldots, \#\el(v) $ the function $x\mapsto t_n(x)$ is increasing, the set
\begin{equation*}
F:=\left\{x: \exists n=1,\ldots, \#\el(v): \left( t_l(x+)-t_l(x-) \ge \frac{T-\bar t}{2L} \right) \right\}
\end{equation*}
is finite.
For every $n=1,\ldots, \#\el(v)$ let $\tilde \delta (n)$ be given by Lemma \ref{L_uniform_continuity} with 
\begin{equation*}
g=t_n:\Int A(\el)\rightarrow \R \qquad \mbox{and} \qquad \tilde \e = \frac{T-\bar t}{2L},
\end{equation*}
and fix
\begin{equation*}
\tilde \delta = \min_{n=1,\ldots, \#\el(v)}\delta(n).
\end{equation*}
Then, by Lemma \ref{L_uniform_continuity}, any sequence $(x'_{v'})_{v'=1}^{V'}$ with
$x_1'=\tilde x_q$, $x'=\tilde x_{q+1}$ and such that for every $v'=1,\ldots,V'-1$
\begin{equation*}
0<x_{v'+1}'-x_{v'}'<\tilde \delta
\end{equation*}
satisfies condition (1) of Claim 1.
The Point (2) follows by the Point (5) in Proposition \ref{P_lagr_repr} and this concludes the proof of Claim 1.

\textbf{Claim 2.}
Let $(\bar x_v,\bar x_{v+1})$ as in Claim 1.
There exist $\bar \delta>0$ and a constant $C$ as in the statement of Proposition \ref{P_inflection} such that the positive total variation
\begin{equation*}
\TV_{(\bar x_v,\bar x_{v+1})}^+(f'\circ u(T)) \le \frac{C}{(T-\bar t)^2}A_{v},
\end{equation*}
where $A_{v}$ denotes the area of the region
\begin{equation*}
\Omega_{v}:=\{(t,x) \in (\bar t,T): \xi_{\bar x_v}(t)<x<\xi_{\bar x_{v+1}}(t)\}.
\end{equation*}
\emph{Proof of Claim 2.}
If $\#\el(v)=1$ we are in the same position as in Proposition \ref{P_convex}: in particular $f'\circ u(T)$ is one-sided Lipschitz and
\begin{equation}
\TV^+_{(\bar x_v,\bar x_{v+1})}(f'\circ u(T)) \le \frac{\bar x_{v+1}-\bar x_v}{T-\bar t}\le \frac{A_v}{2(T-\bar t)^2}.
\end{equation}
Now we consider the case $\#\el(v)\ge 2$ so that by Claim 1 for every $x_1<x_2$ in $(\bar x_v, \bar x_{v+1})$, it holds
\begin{equation*}
\TV_{(x_1,x_2)}(f'\circ u(T)) = \TV^+_{(x_1,x_2)}(f'\circ u(T))= f'(u(T,x_2))-f'(u(T,x_1)).
\end{equation*}

For every $n=2,\ldots, \#\el(v)$, consider the time $t'_n \in \R$ for which the straight-line extensions of the segments 
$\X(\cdot, y_n(x_1))\llcorner[t_{n-1}(x_1),t_n(x_1)]$ and $\X(\cdot, y_n(x_2))\llcorner[t_{n-1}(x_2),t_n(x_2)]$ intersect.
Since they are tangent to the same convex curve at the time $t_{n-1}(x_1)$ and $t_{n-1}(x_2)$ respectively it holds 
\begin{equation}\label{E_bound_t'}
t'_n\in (t_{n-1}(x_1),t_{n-1}(x_2)).
\end{equation}
See Figure \ref{F_Cheng}. Moreover for every $n=2,\ldots,\#\el(v)$, set
\begin{equation*}
\tau_n:= (t_n(x_1)-t'_n)^+ \qquad \mbox{and let}\qquad \tau_1:=t_1(x_1)-\bar t.
\end{equation*}
Let $\Delta_{\#\el(v)}$ be the area of the triangle bounded by the following three lines:
\begin{equation*}
\begin{split}
& \{(t,x): x= x_1 - f'(u_0(y_{\#\el(v)}(x_1)))(T- t)\}, \\
& \{(t,x): x= x_2 - f'(u_0(y_{\#\el(v)}(x_2)))(T- t)\}, \\
& \{(t,x): t=T\}.
\end{split}
\end{equation*}
If $n=2,\ldots, \#\el(v)-1$ is such that $\tau_n>0$ let $\Delta_n$ be the area of the triangle bounded by the following three lines:
\begin{equation*}
\begin{split}
& \{(t,x): x= \X(t_n(x_1),y_n(x_1)) - f'(u_0(y_n(x_1)))(t_n(x_1)- t)\}, \\
& \{(t,x): x= \X(t_n(x_2),y_n(x_2)) - f'(u_0(y_n(x_2)))(t_n(x_2)- t)\}, \\
& \{(t,x): t=t_n(x_1)\}.
\end{split}
\end{equation*}
If $n=2,\ldots, \#\el(v)-1$ is such that $\tau_n=0$, let $\Delta_n=0$, and finally let $\Delta_1$ be the area of the trapezoid delimited by the lines 
\begin{equation*}
\begin{split}
& \{(t,x): x= \X(t_1(x_1),y_1(x_1)) - f'(u_0(y_1(x_1)))(t_1(x_1)- t)\}, \\
& \{(t,x): x= \X(t_1(x_2),y_1(x_2)) - f'(u_0(y_1(x_2)))(t_1(x_2)- t)\}, \\
& \{(t,x): t=t_1(x_1)\}, \\
& \{(t,x): t=\bar t\}.
\end{split}
\end{equation*}
For every $n=2,\ldots,\#\el(v)$ the area of the triangle is given by
\begin{equation}\label{E_Delta_k}
\Delta_n=\frac{\tau_n^2}{2}\big(f'(u_0(y_n(x_2)))-f'(u_0(y_n(x_1)))\big),
\end{equation}
and for $n=1$
\begin{equation*}
\begin{split}
\Delta_1=&~ \frac{\tau_1^2}{2}\big(f'(u_0(y_1(x_2)))-f'(u_0(y_1(x_1)))\big)+ \tau_1\big(\X(\bar t,y_1(x_2))- \X(\bar t,y_1(x_1))\big) \\ 
\ge &~  \frac{\tau_1^2}{2}\big(f'(u_0(y_1(x_2)))-f'(u_0(y_1(x_1)))\big).
\end{split}
\end{equation*}

\begin{figure}
\centering
\def\svgwidth{0.9\columnwidth}
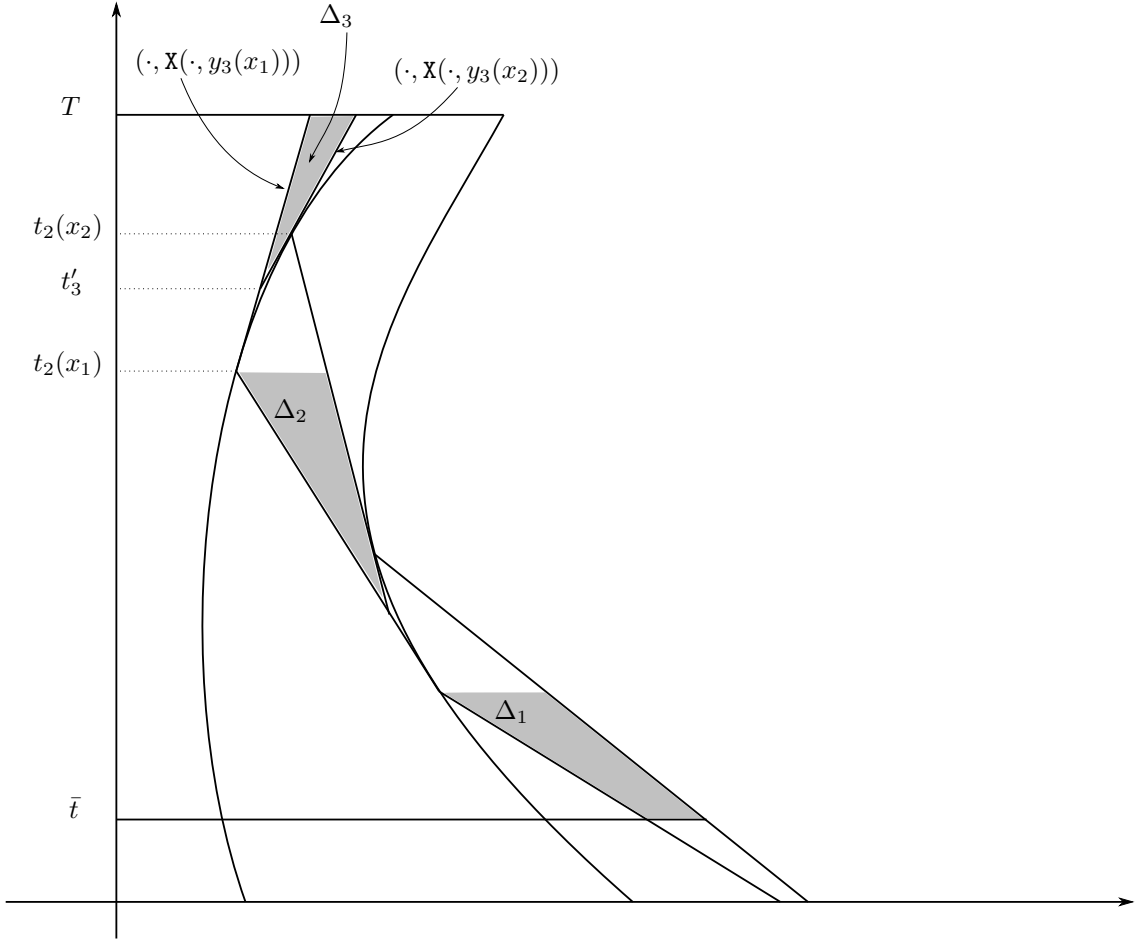
\caption{The notation of the construction to estimate $\TV^+( f'\circ u(T))$ in $A_m$ for fluxes with an inflection point of polynomial degeneracy.}\label{F_Cheng}
\end{figure}

We now prove that
\begin{equation}\label{E_lower_tau}
\sum_{n=1}^{\#\el(v)}\tau_n \ge \frac{T-\bar t}{2}.
\end{equation}

Recalling that $t_{\#\el(v)}(x_1)=T$, $t_0(x_2)=\bar t$ and \eqref{E_bound_t'}, we have that
\begin{equation*}
\begin{split}
\sum_{n=1}^{\#\el(v)}\tau_n \ge &~ \sum_{n=1}^{\#\el(v)}\big(t_n(x_1)-t'_n\big) \\
\ge &~ \sum_{n=1}^{\#\el(v)}\big( t_n(x_1)-t_{n-1}(x_2)\big) \\
= &~ T-\bar t + \sum_{n=1}^{\#\el(v)-1}\big(t_n(x_1)-t_n(x_2)\big) \\
\ge &~ \frac{T-\bar t}{2},
\end{split}
\end{equation*}
where the last inequality follows by \eqref{E_small_gap_t}.

Since for every $n=2,\ldots,\#\el(v)$ and $s=1,2$, $u_0(y_n(x_s))=u_0(y_{n-1}(x_s))^*$, by  \eqref{E_Delta_k} and iterating \eqref{E_ratio_delta}, we have that for every $n=1,\ldots, \#\el(v)$ it holds
\begin{equation}\label{E_estimate_Delta_n}
\Delta_n\ge \frac{\tau_n^2}{2}\big(f'(u_0(y_{\#\el(v)}(x_2)))-f'(u_0(y_{\#\el(v)}(x_1)))\big)\bigg(\frac{1}{1-\e'}\bigg)^{\#\el(v)-n}.
\end{equation}
Let us for brevity denote by
\begin{equation*}
\lambda:=\frac{1}{1-\e'}>1.
\end{equation*}
Since the $(\Delta_n)_{n=1}^{\#\el(v)}$ are the area of pairwise disjoint regions contained in $\Omega_{x_1,x_2}$, it holds
\begin{equation*}
\sum_{n=1}^{\#\el(v)} \Delta_n \le A_{x_1,x_2}.
\end{equation*}
Therefore adding for $n=1,\ldots, \#\el(v)$ the inequality \eqref{E_estimate_Delta_n} we obtain
\begin{equation*}
f'(u_0(y_{\#\el(v)}(x_2)))-f'(u_0(y_{\#\el(v)}(x_1)) \le A_{x_1,x_2}\bigg(\sum_{n=1}^{\#\el(v)} \frac{\tau_n^2}{2}\lambda^{\#\el(v)-n} \bigg)^{-1}.
\end{equation*}
Hence the proof of Claim 2 reduces to proving that there exists a constant $C>0$ as in the statement of Proposition \ref{P_inflection} such that
\begin{equation*}
\left(\sum_{n=1}^{\#\el(v)} \frac{\tau_n^2}{2}\lambda^{\#\el(v)-n} \right)^{-1}\le C,
\end{equation*}
or equivalently that there exists $c>0$ such that  
\begin{equation*}
\sum_{n=1}^{\#\el(v)} \frac{\tau_n^2}{2}\lambda^{\#\el(v)-n} \ge c.
\end{equation*}
This follows by \eqref{E_lower_tau} and $\lambda>1$. In fact let $a,b\in \R^{\#\el(v)}$ be the vectors of components
\begin{equation*}
a_n=\tau_n\lambda^{\frac{\#\el(v)-n}{2}}, \qquad \mbox{and}\qquad b_n= \lambda^{-\frac{\#\el(v)-n}{2}}.
\end{equation*}
Then, by the Cauchy-Schwarz inequality,
\begin{equation*}
\left(\sum_{n=1}^{\#\el(v)}\tau_n\right)^2\le \left(\sum_{n=1}^{\#\el(v)} \tau_n^2\lambda^{\#\el(v)-n}\right)\sum_{n=1}^{\#\el(v)}\lambda^{\#\el(v)-n},
\end{equation*}
so that by \eqref{E_lower_tau},
\begin{equation*}
\sum_{n=1}^{\#\el(v)} \tau_n^2\lambda^{\#\el(v)-n}\ge \left(\frac{T-\bar t}{2}\right)^2\left(\sum_{n=1}^\infty\lambda^{n}\right)^{-1},
\end{equation*}
and this concludes the proof of Claim 2.
Since $T>0$ is generic, the function $f'\circ u(T)$ does not have jumps of positive sign, therefore applying
Point(1) of Lemma \ref{L_trivial} with $n=V$ and $x_i=\bar x_i$ for $i=1,\ldots, V$, we get
\begin{equation}\label{E_est_A_m}
\TV^+_{A_m}(f'\circ u(T))\le \frac{C}{(T-\bar t)^2}A_v.
\end{equation}
Finally by \eqref{E_est_A_l}, \eqref{E_est_A_r} and \eqref{E_est_A_m}, it follows \eqref{E_est_inflection} and this
concludes the proof of Proposition \ref{P_inflection}.
\end{proof}

The next lemma will be used to reduce the estimate of the total variation of $f'\circ u(T)$ to the estimate on the
regions where the oscillation of the solution is small. The smallness parameter $\delta>0$ will be chosen later. 

\begin{lemma}\label{L_red_small}
Let $u$ be the entropy solution of \eqref{E_cl} with $u_0\in X$ and let $\bar t,\delta>0$ with $\bar t$ generic.
Then there exists $M\in\N$ depending only on $\delta, f, \|u_0\|_\infty, \mathcal L^1(\conv (\supp u_0)), \bar t$ and there exist 
$y_1,\ldots, y_{M(u_0)}$ with $M(u_0)\le M$ such that for every $m=1,\ldots M-1$, there exists $k=k(m)\in \N$ for which for every $t>\bar t$, it holds
\begin{equation}\label{E_small_osc}
u(t,(\X(t,y_m),\X(t,y_{m+1})))\subset [(k-2)\delta, (k+2)\delta]
\end{equation}
and
\begin{equation}\label{E_small_osc_lat}
u(t,(-\infty,\X(t,y_1)))\subset [0,2\delta], \qquad u(t,(\X(t,y_m),+\infty))\subset [0,2\delta].
\end{equation}
\end{lemma}
\begin{proof}
Let $\X$ be a Lagrangian representation of $u$ and let $\T$ be a time existence function as in Proposition \ref{P_lagr_repr}.
Consider the map $y=y(t,x,w)$ defined in Proposition \ref{P_lagr_repr} and for every $w\in\R$ let
\begin{equation*}
A_w:=\{y\in \R: \T(y)\ge \bar t \mbox{ and } u_0(y)=w\}.
\end{equation*}
Let $y_1:=\min A_\delta$ and for $m\in \N$ with $l\ge 2$ we define recursively
\begin{equation*}
y_m := \min \big((A_{u_0(y_{m-1})+\delta}\cup A_{u_0(y_{m-1})-\delta})\cap [y_m,+\infty)\big)
\end{equation*}
if the set on the right hand side is nonempty, otherwise we set $y_m=+\infty$ (see Figure \ref{F_y-l}).

\begin{figure}
\centering
\def\svgwidth{0.7\columnwidth}
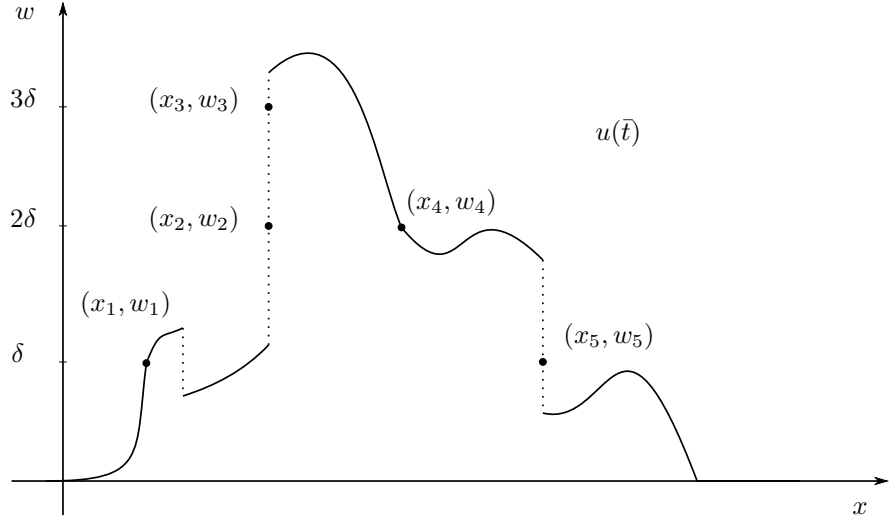
\caption{A representation of the construction of $(y_m)_m$: for $m=1,\ldots,5$, $y_m=y(x_m,w_m)$.}\label{F_y-l}
\end{figure}

By definition it is obvious that the sequence $(y_m)_{m\in \N}$ is increasing.
For every $u_0\in X$ denote by $M(u_0)$ the number of indexes $m$ such that $y_m$ is finite; by construction we
have the estimate
\begin{equation*}
M(u_0)\delta \le \TV (u(\bar t)) \le \TV (u_0).
\end{equation*}

Since $\|u(\bar t)\|_\infty\le \|u_0\|_\infty$ the number $N(u(\bar t),\delta)$ of undulations of $u(\bar t)$ of height bigger than $\delta$ is bounded
by below by
\begin{equation}\label{E_est_below}
N(u(\bar t),\delta)\ge \frac{M(u_0)\delta}{2\|u_0\|_\infty}.
\end{equation}
Moreover, by Lemma \ref{L_number_oscillations},
\begin{equation}\label{E_est_above}
N(u(\bar t),\delta) \le \frac{4\|u_0\|_\infty(\mathcal L^1(\conv(\supp u_0))+\|f'\|_\infty \bar t)}{\bar t\di(\delta)}.
\end{equation}
By \eqref{E_est_below} and \eqref{E_est_above}, we have that $M(u_0)$ is uniformly bounded by a constant $M$ as in the statement. 
Now it remains to prove \eqref{E_small_osc} and \eqref{E_small_osc_lat} and they follow by the definition of Lagrangian representation and the construction of 
$(y_m)_{m\in \N}$.
\end{proof}

We now have all the ingredients to prove the main result of this section.
\begin{theorem}\label{T_Cheng}
Let $f$ be a flux with polynomial degeneracy and let $u$ be the entropy solution of \eqref{E_cl} with $u_0\in L^\infty(\R)$ nonnegative and
with compact support.
Then there exists a constant $C>0$ depending on $f$, $\|u_0\|_\infty$ and $\mathcal L^1(\conv(\supp u_0))$ such that for every $T>0$
\begin{equation}\label{E_Cheng}
\TV (f'\circ u(T)) \le C\left(1+\frac{1}{T}\right).
\end{equation}
\end{theorem}
\begin{proof}
Observe that it is enough to prove the statement for $u_0 \in X$: indeed for every $u_0$ as in the statement there exists a sequence 
$(u_0^n)_{n\in \N}$ contained in $X$ such that $u_0^n\rightarrow u_0$ in $L^1(\R)$ and $\|u^n_0\|_\infty, \mathcal L^1(\conv(\supp u^n_0))$ are
uniformly bounded by $\|u_0\|_\infty$ and $\mathcal L^1(\conv(\supp u_0))$ respectively.

Then notice also that, since $t\mapsto u(t)$ is continuous with respect to the $L^1$ topology and the total variation is lower semicontinuous
with respect to the same topology, it is sufficient to prove \eqref{E_Cheng} for a dense set of $T>0$. In particular we assume that $T$ is generic 
(see Definition \ref{D_generic}).

Since every inflection point $w_s$ of $f$ has polynomial degeneracy, if $f''$ changes sign at $w_s$, then there exists $\delta_s>0$ such that $f''$ is
monotone in $(w_s-\delta_s,w_s+\delta_s)$. Consider $\delta',\e'>0$ given by Corollary \ref{C_conjugate} and apply Lemma \ref{L_red_small} with
\begin{equation*}
\delta < \left(\frac{\delta'}{2} \wedge \min_{s=1,\ldots,S-1} \delta_s\right).
\end{equation*}
Taking into account the possible jumps of $f'\circ u(T)$ at the points $\X(T,y_m)$ for $m=1,\ldots, M(u_0)$, we have that
\begin{equation}\label{E_total_variation}
\TV (f'\circ u(T)) \le \sum_{m=1}^{M(u_0)-1}\TV_{(\X(T,y_m)\X(T,y_{m+1}))}(f'\circ u(T)) + M\|f''\|_\infty\|u_0\|_\infty,
\end{equation}
where $M$ and $y_1,\ldots, y_{M(u_0)}$ are given by Lemma \ref{L_red_small}.
By the choice of $\delta$, it holds in particular that for every $m=1,\ldots M(u_0)-1$, there exists at most one inflection point $w_s$ of $f$ such that
\begin{equation}\label{E_inflection}
w_s\in u(\bar t,(\X(\bar t, y_m),\X(\bar t, y_{m+1}))).
\end{equation}
We can therefore distinguish the following cases:
\begin{enumerate}
\item there exists no $s$ such that \eqref{E_inflection} holds;
\item there exists $s=1,\ldots, S$ such that \eqref{E_inflection} holds and $f'$ does not change sign at $w_s$;
\item there exists $s=1,\ldots, S$ such that \eqref{E_inflection} holds and $f'$ changes sign at $w_s$.
\end{enumerate}
Now we check that in Case (1) and in Case (2) we can apply Proposition \ref{P_convex} (or its obvious version in the concave case).
Let
\begin{equation*}
y^-:=\max\{y: \X(T,y)=\X(T,y_m)\} \qquad \mbox{and} \qquad w^-:=\lim_{x\rightarrow \X(T,y_m)^+}u(T,x).
\end{equation*}
By the stability of the notion of admissible boundary (Proposition \ref{P_stab}), we have that $(\X(\cdot,y^-),w^-)$ is an admissible
boundary of $u$.
Similarly, if we let  
\begin{equation*}
y^+:=\min\{y: \X(T,y)=\X(T,y_{m+1})\} \qquad \mbox{and} \qquad w^+:=\lim_{x\rightarrow \X(T,y_{m+1})^-}u(T,x),
\end{equation*}
we have that $(\X(\cdot,y^+),w^+)$ is an admissible boundary of $u$.
Therefore we can apply Proposition \ref{P_convex} and we get that
\begin{equation}\label{E_est_convex}
\TV_{(\X(T, y_m),\X(T, y_{m+1}))}(f'\circ u(T)) \le 5\|f''\|_\infty\|u_0\|_\infty + 2 \frac{\X(T,y_{m+1})-\X(T,y_m)}{T-\bar t}.
\end{equation}

The same argument shows that in Case (3) we can apply Proposition \ref{P_inflection} and it implies that there exists a constant $C>0$ depending
on $f,\|u_0\|_\infty,\mathcal L^1(\conv(\supp u_0))$ such that 
\begin{equation}\label{E_est_inflection2}
\TV_{(\X(T,y_m),\X(T,y_{m+1}))} (f'\circ u(T)) \le C\left(1+\frac{1}{T-\bar t}\right).
\end{equation}

By finite speed of propagation
\begin{equation*}
\begin{split}
\sum_{m=1}^{M(u_0)-1}(\X(T,y_{m+1})-\X(T,y_m))=&~ \X(T,y_{M(u_0)})-\X(T,y_1) \\
\le &~  \mathcal L^1(\conv(\supp u(T)))\\ 
\le &~  \mathcal L^1(\conv(\supp u_0)) + 2\|f'\|_\infty T.
\end{split}
\end{equation*}

Therefore choosing $\bar t=T/2$ and combining \eqref{E_est_convex}, \eqref{E_est_inflection2} and \eqref{E_total_variation}, we get that
there exists a constant $C$ depending on $f,\|u_0\|_\infty,\mathcal L^1(\conv(\supp u_0))$ such that 
\begin{equation*}
\TV (f'\circ u(T)) \le C\left(1+\frac{1}{T}\right),
\end{equation*}
and this concludes the proof.
\end{proof}

We conclude this section with the following remark about the hypothesis of Theorem \ref{T_Cheng}.

\begin{remark}\label{R_sharp_assumptions}
Proposition \ref{P_inflection} requires the polynomial degeneracy assumption of $f$ at the inflection point; on the contrary Proposition \ref{P_convex}
does not. Moreover the structure of characteristics described in Lemma \ref{L_convex} holds for every strictly convex flux $f$.
In particular Theorem \ref{T_Cheng} holds under the following assumptions on the flux: there exists $w_1<\ldots <w_S$ such that
$f\llcorner (w_s,w_{s+1})$ is strictly convex or strictly concave for every $s=1,\ldots, S-1$ and that $f$ has polynomial degeneracy at $w_s$ for 
every $s=1,\ldots, S$.
\end{remark}

\section{Fractional $\BV$ regularity of the solution}\label{S_frac_Cheng}
In this section we want to deduce a $\BV^{1/p}$ regularity result of the solution $u$ from the $\BV$ regularity of $f'\circ u$ 
obtained in Section \ref{S_Cheng}, where $p$ is the degeneracy of $f$.

We briefly describe the argument. If the flux is strictly convex, then the polynomial degeneracy  of $f$ implies an H\"{o}lder type estimate for $(f')^{-1}$: 
\begin{equation}\label{E_Holder}
(b-a)^{p-1}\le C|f'(b)-f'(a)|
\end{equation}
for some $C>0$ and this is sufficient to conclude.
Of course \eqref{E_Holder} does not hold for general fluxes $f$ of polynomial degeneracy, 
but it holds for every $a<b$ for which $f$ has no inflection points in $(a,b)$, (Lemma \ref{L_convex_flux}).
This is sufficient to conclude the proof for continuous solutions.

It remains to consider jumps. As before we distinguish among big and small jumps: big jumps are treated as in Section \ref{S_Cheng}
by means of Theorem \ref{T_main_estimate} and small jumps around the inflection point between two different values $w,w'$ with $f'(w)\simeq f'(w')$ are excluded by
the entropy admissibility condition (Lemma \ref{L_chord} and Figure \ref{F_secant}).

\begin{figure}
\centering
\def\svgwidth{0.5\columnwidth}
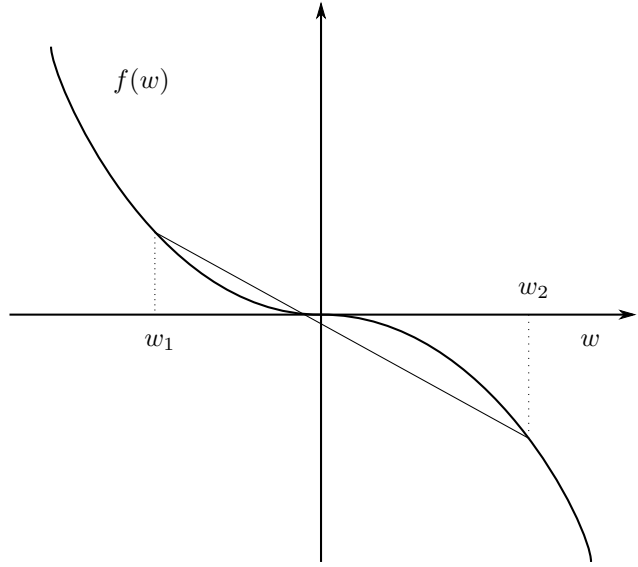
\caption{If $f'(w_1)\approx f'(w_2)$ with $w_1<\bar w<w_2$ the shocks between $w_1$ and $w_2$ are not admissible.}
\label{F_secant}
\end{figure}

\begin{lemma}\label{L_convex_flux}
Let $g:[0,M]\rightarrow \R$ be a smooth function and $p \ge 2$ be an integer such that 
\begin{equation*}
g'\ne 0 \mbox{ in }(0,M], \qquad g^{(j)}(0)=0 \mbox{ for }j=1,\ldots, p-1 \qquad \mbox{and}\qquad  g^{(p)}(0)\ne 0.
\end{equation*}
Then for every $l\ge p$ there exists a constant $C>0$ such that for every $0\le a\le b\le M$
\begin{equation*}
(b-a)^l\le C|g(b)-g(a)|.
\end{equation*}
\end{lemma}
\begin{proof}
The result easily follows for $l=p$ and hence for every $l\ge p$ by Taylor expansion in a right neighborhood $[0,\delta)$ of zero and 
by the fact that 
\begin{equation*}
\min_{[\delta,M]}|g'|> 0. \qedhere
\end{equation*}
\end{proof}

\begin{lemma}\label{L_chord}
Let $u$ be an entropy solution of \eqref{E_cl} with $u_0\in L^\infty(\R)$ and $f$ of degeneracy $p\in \N$ and let $\bar t>0$ be generic.
Then there exist two constants $c,\delta'>0$, depending on $f$ and $\|u_0\|_\infty$, such that
for every $x_1,x_2$ with
\begin{equation*}
w_s-\delta' < u(\bar t,x_1)<w_s< u(\bar t,x_2)<w_s+\delta'
\end{equation*}
for some $s=1,\ldots,S$,
it holds
\begin{equation*}
\TV_{\mathcal I(x_1,x_2)}(f'\circ u(\bar t)) \ge c |u(\bar t, x_2)-u(\bar t,x_1)|^p,
\end{equation*}
where $\mathcal I(x_1,x_2)=(x_1,x_2)\cup (x_2,x_1)$ denotes the open interval with endpoints $x_1$, $x_2$. 
\end{lemma}
\begin{proof}
We assume for simplicity that $w_s=0$ and $f'(0)=0$. Moreover it is not restrictive to assume that $x_1<x_2$.
Let $\delta,\e$ be given by Corollary \ref{C_conjugate}. Then
\begin{equation*}
|w|<\delta \qquad \Rightarrow \qquad \frac{|w^*|}{|w|}\in (0,1-\e).
\end{equation*}
Therefore there exists $\tilde c>0$ such that if $|w|<\delta$, then
\begin{equation*}
|f'(w)-f'(w^*)|\ge \tilde c |w|^p \qquad \mbox{and in particular} \qquad |f'(w)|\ge \tilde c |w|^p.
\end{equation*}
We distinguish three cases:
\begin{enumerate}
\item there exists $\bar x \in (x_1,x_2)$ such that $u(\bar t,\bar x)=0$;
\item there exists $\bar x \in (x_1,x_2)$ such that $u(\bar t,\bar x) \notin (-2\delta,2\delta)$.
\item there exists $\bar x \in (x_1,x_2)$ such that 
\begin{equation*}
-2\delta<u(\bar t, \bar x+)<0<u(\bar t, \bar x-)<2\delta;
\end{equation*}
\end{enumerate}
Case (1): it holds
\begin{equation*}
\begin{split}
\TV_{[x_1,x_2]}(f'\circ u(\bar t)) \ge&~ |f'(u(\bar t, x_1))-f'(u(\bar t, \bar x))|+|f'(u(\bar t, \bar x))-f'(u(\bar t, x_2))| \\
= &~|f'(u(\bar t, x_1))|+|f'(u(\bar t, x_2)))| \\
\ge &~\tilde c (|u(\bar t, x_1)|^p+|u(\bar t, x_2)|^p) \\
\ge &~\tilde c2^{-p}|u(\bar t, x_1)-u(\bar t, x_2)|^p.
\end{split}
\end{equation*}
Case (2): similarly to the case above above, it holds 
\begin{equation*}
\begin{split}
\TV_{[x_1,x_2]}(f'\circ u(\bar t)) \ge&~ |f'(u(\bar t, x_1))-f'(u(\bar t, \bar x))| \\
\ge &~\tilde c |u(\bar t, x_1)-u(\bar t, \bar x)|^p \\
\ge &~\tilde c \delta^p \\
\ge &~\tilde c2^{-p}|u(\bar t, x_1)-u(\bar t, x_2)|^p.
\end{split}
\end{equation*}
Case (3):
suppose additionally that $\max\{|u(\bar t,x_1)|, |u(\bar t,x_2)|\}\le 2 |u(\bar t, \bar x+)|$. Then
\begin{equation*}
\begin{split}
\TV_{[x_1,x_2]}(f'\circ u(\bar t)) \ge&~ |f'(u(\bar t, \bar x+))-f'(u(\bar t, \bar x-))| \\
\ge &~ |f'(u(\bar t, \bar x+))-f'(u(\bar t, \bar x+)^*)| \\
\ge &~\tilde c |u(\bar t, \bar x+)|^p \\
\ge &~\tilde c 2^{-p}\max\{|u(\bar t,x_1)|, |u(\bar t,x_2)|\}^p \\
\ge &~\tilde c4^{-p}|u(\bar t, x_1)-u(\bar t, x_2)|^p.
\end{split}
\end{equation*}
If instead for definitness $|u(\bar t,x_1)|=\max\{|u(\bar t,x_1)|, |u(\bar t,x_2)|\}\ge 2 |u(\bar t, \bar x+)|$, then
\begin{equation*}
\begin{split}
\TV_{[x_1,x_2]}(f'\circ u(\bar t)) \ge&~ |f'(u(\bar t, \bar x+))-f'(u(\bar t, x_1))| \\
\ge &~ \tilde c |u(\bar t, \bar x+)-u(\bar t, x_1)|^p \\
\ge &~ \tilde c 2^{-p}|u(\bar t,x_1)|^p \\
\ge &~\tilde c 4^{-p}|u(\bar t, x_1)-u(\bar t, x_2)|^p.
\end{split}
\end{equation*}
Setting $c=\tilde c 4^{-p}$, the lemma is proved.
\end{proof}

The main result of this section is stated in the following theorem.

\begin{theorem}\label{T_frac_regularity}
Let $u$ be the entropy solution of \eqref{E_cl} with $u_0\in L^\infty(\R)$ nonnegative and with compact support, and let $p$ be the degeneracy of the flux $f$. 
Then for every $t>0$ the solution
\begin{equation*}
u(t)\in \BV^{1/p}(\R),
\end{equation*}
and there exists a constant $C>0$ depending on $f,\|u_0\|_\infty$ and $\mathcal L^1(\conv (\supp u_0))$ such that for every $t>0$
\begin{equation*}
\left(\TV^{1/p}u(t)\right)^p\le C\left(1+\frac{1}{t}\right).
\end{equation*}
\end{theorem}
\begin{proof}
By lower semicontinuity of the $\TV^{1/p}$ and the continuity in time of the entropy solution with respect to the $L^1$ topology, it suffices to prove the estimate for a dense set of $t>0$; in particular we can assume that $t$ is generic.
Let $\delta>0$ be given so that the conclusion of Lemma \ref{L_chord} holds.
By Theorem \ref{T_main_estimate} there exists a constant $\bar N=\bar N(\mathcal L^1(\conv(\supp u_0)),\delta,\|u_0\|_\infty,f)$ 
such that for every $x_1<\ldots<x_m$,
\begin{equation*}
\#\{i: |u(t,x_{i+1})-u(t,x_i)|\ge \delta\} \le  \bar N \left(1+\frac{1}{t}\right).
\end{equation*}
So
\begin{equation*}
\sum_{i=1}^{m-1} |u(t,x_{i+1})-u(t,x_i)|^p \le \bar N (2\|u_0\|_\infty)^p + \sum_{I_\delta}|u(t,x_{i+1})-u(t,x_i)|^p,
\end{equation*}
where $I_\delta=\{i: |u(t,x_{i+1})-u(t,x_i)|<\delta\}$.
If $f''\ne 0$ in $\mathcal I(u(t,x_{i+1}),u(t,x_i))$, by Lemma \ref{L_convex_flux} with $g=f'$, we get 
\begin{equation*}
|u(t,x_{i+1})-u(t,x_i)|^p\le C|f'(u(t,x_{i+1}))-f'(u(t,x_i))|.
\end{equation*}
Similarly if $w_s\in \mathcal I(u(t,x_{i+1}),u(t,x_i))$ for some $s=1,\ldots,S$, by Lemma \ref{L_chord}, 
\begin{equation*}
\TV_{[x_1,x_2]}(f'\circ u(t)) \ge c |u(t,x_2)-u(t,x_1)|^p.
\end{equation*}
Finally
\begin{equation*}
\sum_{i=1}^{m-1} |u(t,x_{i+1})-u(t,x_i)|^p \le  \bar N\left(1+\frac{1}{t}\right) (2\|u_0\|_\infty)^p + \tilde C \TV (f'\circ u(t)),
\end{equation*}
where $\tilde C=\max\{C,1/c\}$. The conclusion follows by Theorem \ref{T_Cheng}.
\end{proof}

\section{$\SBV$ regularity of $f'\circ u$}
\label{S_SBV}
In this section we prove that the $\BV$ regularity of the velocity $f'\circ u$ can be improved to $\SBV$ regularity.
As in the convex case \cite{ADL_note}, the proof is based on the structure of characteristics.
Recall the partition $\R^+\times \R=A\cup B\cup C$ introduced in Proposition \ref{P_structure} and let us set 
for every $t>0$, the time sections
\begin{equation*}
A_t:=\{x:(t,x)\in A\}, \qquad B_t:=\{x:(t,x)\in B\}, \qquad C_t:=\{x:(t,x)\in C\}.
\end{equation*}

\begin{proposition}\label{P_SBV}
Let the flux $f$ be smooth and $\bar u$ be a representative of the entropy solution $u$ to \eqref{E_cl} as in Proposition \ref{P_structure}. Denote by
\begin{equation*}
\begin{split}
\mathcal B := &~ \{t\in (0,+\infty): f'\circ \bar u(t) \in \BV_{\loc}(\R)\}, \\
\mathcal S := &~ \{t\in (0,+\infty): f'\circ \bar u(t) \in \SBV_{\loc}(\R)\}.
\end{split}
\end{equation*}
Then $\mathcal B \setminus \mathcal S$ is at most countable.
\end{proposition}
\begin{proof}
By finite speed of propagation, it is not restrictive to assume that $\supp u_0 \subset [a,b]$ for some $a,b\in \R$.
For every $t >0$, we set
\begin{equation*}
F(t):=\mathcal L^1\left(\{\X(0,y)\in[a,b]: \X(t,y)\in C_t\} \right),
\end{equation*}
where $\X$ is the Lagrangian flow given by Proposition \ref{P_structure}.
Observe that $F$ is decreasing.
We are going to prove that if $t\in \mathcal B\setminus \mathcal S$, then $F(t+)<F(t-)$ and this easily implies the claim.
So let $t\in  \mathcal B\setminus \mathcal S$, and denote by $v:=f'\circ \bar u(t)$. Finally let $\mu$ be the Cantor part of $Dv$.
Since $v$ is locally Lipschitz in $B_t$ by Proposition \ref{P_structure}, and the fact that $A_t$ is at most countable, we have that the measure 
$\mu$ is concentrated on $C_t$.
Moreover, as already observed in the proof of Proposition \ref{P_convex}, for every $x_1<x_2$ in $C_t$ it holds the one-sided Lipschitz estimate
\begin{equation*}
v(x_2)-v(x_1)\le \frac{x_2-x_1}{t}.
\end{equation*}
Therefore $\mu$ is a negative measure.
Fix $\e\in \left(0,\frac{1}{3}\right)$; since $\mu$ is negative and it is singular to $\mathcal L^1 + |Dv-\mu|$, by Besicovitch differentation theorem, 
there exists $E\subset \R$ such that
\begin{enumerate}
\item $\mu$ is concentrated on $E$; 
\item $\mathcal L^1(E)=0$;
\item for every $x\in E$ there exists two sequences $z_i^1(x)\to x$ and $z_i^2(x)\to x$ with $z_i^1(x)<x$, $z_i^2(x)>x$ 
and  such that for every $i\ge 1$,
\begin{equation}\label{E_all_mu}
v(z^1_i(x))-v(x) \ge (1-\e)|Dv|([z^1_i(x),x]) \qquad \mbox{and} \qquad v(x)- v(z^2_i(x))\ge (1-\e)|Dv|([x,z^2_i(x)]).
\end{equation}
\end{enumerate}
Observe that since $\mu$ has no atoms, up to removing a countable set from $E$, we can assume that the sequences $z^1_i$ and $z^2_i$
are contained in $C_t$.

The next step is to give a lower bound on $\mathcal L^1(\{\X(0,y): \X(t,y) \in E\}$, see Figure \ref{F_backward}.
Denote by
\begin{equation*}
Y:=\{y:\X(t,y)\in E\} \qquad \mbox{and} \qquad \nu:= \X(t,\cdot)_\sharp \left(\mathcal L^1 \llcorner \X(0,\R\setminus Y)\right);
\end{equation*}
since $\mu$ is concentrated on $E$, it holds $\mu \bot \nu$. Therefore, by Besicovitch covering theorem, 
there exist $x_1,\ldots, x_N \in E$ and $a_n:=z^1_i(x_n)$, $b_n:=z^2_j(x_n)$ for some $i,j\ge 1$ such that $([a_n,b_n])_{n=1}^N$ is a pairwise
disjoint family of intervals and
 \begin{equation}\label{E_choice_cover}
|\mu|\left(\bigcup_{n=1}^N[a_n,b_n]\right)\ge (1-\e)\|\mu\|, \qquad 
\nu \left(\bigcup_{n=1}^N[a_n,b_n]\right) \le t\e\|\mu\|.
\end{equation}
Since $a_n,b_n\in C_t$ for every $n=1,\ldots,N$, it holds
\begin{equation*}
y(t,b_n)-y(t,a_n)=b_n-a_n+t(v(a_n)-v(b_n)) > t(v(a_n)-v(b_n)).
\end{equation*}
Moreover, by \eqref{E_all_mu}, we have $t(v(a_n)-v(b_n))>(1-\e)|\mu|([a_n,b_n])$.
Set $U:= \X(t)^{-1}\left(\bigcup_{n=1}^N[a_n,b_n]\right)$. By \eqref{E_choice_cover}, summing on $n=1,\ldots, N$, we get
\begin{equation*}
\mathcal L^1(U)\ge t(1-\e) \sum |\mu_t|([a_n,b_n]) \ge t(1-\e)^2\|\mu\| \qquad \mbox{and} \qquad \mathcal L^1(U\setminus \X(0, Y))<t\e \|\mu\|.
\end{equation*}
Therefore we have
\begin{equation}\label{E_dying}
\mathcal L^1(\X(0,Y)) \ge t(1-3\e)\|\mu\|.
\end{equation}
Then we conclude by the following geometrical observation:
let $\tilde Y\subset \R$ be such that 
\begin{equation*}
\mathcal L^1(\X(0,\tilde Y))>0, \qquad \mbox{and} \qquad\mathcal L^1(\X(t,\tilde Y))=0.
\end{equation*}
Let $\tau>t$ and consider the set $\tilde Y(\tau)$ of points $y\in \tilde Y$ such that $\X(\cdot,y)$ has constant speed in $[0,\tau]$;
then $\mathcal L^1(\X(0,\tilde Y(\tau)))=0$. 

This follows from the monotonicity of the map $\X$, see Figure \ref{F_opening}. Indeed for any $y_1<y_2$ in $\tilde Y(\tau)$, since $\X(0,y_1)<\X(0,y_2)$, we have
\begin{equation*}
\begin{split}
(\X(0,y_2) - \X(0,y_1))(\tau-t) =&~ (\X(\tau,y_2)-\X(\tau,y_1))(\tau -t) - (\partial_t\X(t,y_2)-\partial_t\X(t,y_1))(\tau-t)\tau \\
\le &~ (\X(\tau,y_2) -\X(\tau,y_1))\tau + (\partial_t\X(t,y_2)-\partial_t\X(t,y_1))(\tau-t)\tau \\
= &~ (\X(t,y_2)-\X(t,y_1))\tau,
\end{split}
\end{equation*}
i.e. the map 
\begin{equation*}
\X(t,y)\mapsto \X(0, y)
\end{equation*}
is $\tau/(\tau-t)$ Lipschitz on $\tilde Y(\tau)$. In particular, since $\mathcal L^1(\X(t,\tilde Y(\tau)))=0$, then $\mathcal L^1(\X(0,\tilde Y(\tau)))=0$.

Applying this observation to our case with $\tilde Y=Y$ and an arbitrary $\tau >t$, we get that
\begin{equation*}
\mathcal L^1(\{y\in Y: \X(\tau,y)\in C_\tau\})=0.
\end{equation*}
Since $\tau>t$ is arbitrary, by \eqref{E_dying}, we have that
\begin{equation*}
F(t)-F(t+)\ge t(1-3\e)\|\mu\| >0,
\end{equation*}
and this concludes the proof.
\end{proof}
\begin{corollary}\label{C_SBV}
Let $u$ be the entropy solution of \eqref{E_cl} with $u_0\in L^\infty$ and $f$ of polynomial degeneracy (or more in general as in Remark \ref{R_sharp_assumptions}).
Then
\begin{equation*}
f'\circ u\in \SBV_{\loc}(\R^+\times \R).
\end{equation*}
\end{corollary}
\begin{proof}
By Theorem \ref{T_Cheng} and Proposition \ref{P_SBV}, it immediately follows  that  
there exist a representative $\bar u$ of $u$ and  an at most countable set $Q\subset \R^+$ such that for every $t\in \R^+\setminus Q$,
\begin{equation*}
f'\circ \bar u(t) \in \SBV_{\loc}(\R).
\end{equation*}
By slicing theory (see \cite{AFP_book}), the Cantor part of $D_x(f'\circ u)$ vanishes, therefore it remains to prove that also $D_t(f'\circ u)$ has no
Cantor part.

Moreover, denoting by $\mu_k^+$ the dissipation measure of the entropy $\eta_k^+(w)=(w-k)^+$, we have that the velocity $f'\circ u$ 
satisfies the following equation:
\begin{equation}\label{E_f'_diss}
f'(u)_t + \bar q(u)_x = \bar\mu,
\end{equation}
where 
\begin{equation*}
\bar q(w)=\frac{f'(w)^2}{2} \qquad \mbox{and} \qquad \bar \mu=\int_\R \left(f^{(3)}(w)\mu_w^+\right)dw.
\end{equation*}
By Volpert chain rule for functions of bounded variation $\bar q\circ \bar u(t)\in \SBV_{\loc}(\R)$ for every $t\in \R^+\setminus Q$;
in particular the Cantor part of the measure $D_x (\bar q \circ u)$ vanishes.
By Proposition \ref{P_structure}, $\bar \mu$ is absolutely continuous with respect to $\mathcal H^1\llcorner A$, with $A$ countably 
1-rectifiable. In particular it has no Cantor part. Therefore, by \eqref{E_f'_diss}, it follows that the measure $D_t(f'\circ u)$ has no Cantor part,
and this concludes the proof.
\end{proof}

\begin{figure}
     \centering
       \subfloat[][{\normalsize $d_n=b_n-a_n -tDv((a_n,b_n))$, therefore $d_n \approx t|\mu|((a_n,b_n))$ and $\mathcal L^1(\X(0,Y))\approx|\mu|(E)$.}]{
     \centering
     \def\svgwidth{0.5\columnwidth}
     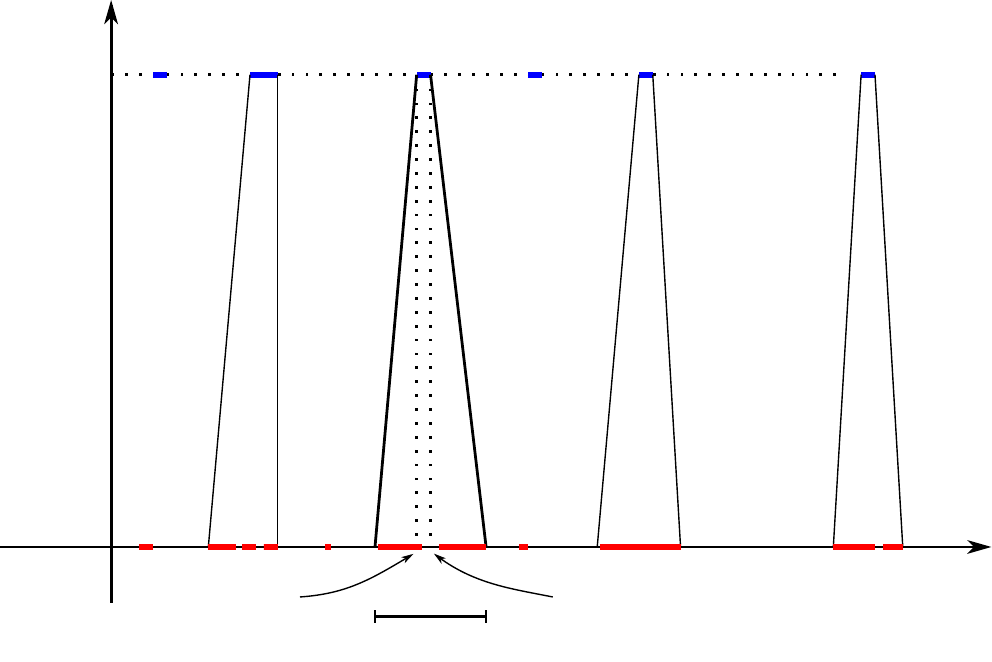
     \label{F_backward}}
     \subfloat[][{\normalsize Since $l_{\tau}\ge 0$, it holds
     $l_0\le \frac{\tau l_t}{\tau -t}$.}]{
     \centering
     \def\svgwidth{0.35\columnwidth}
     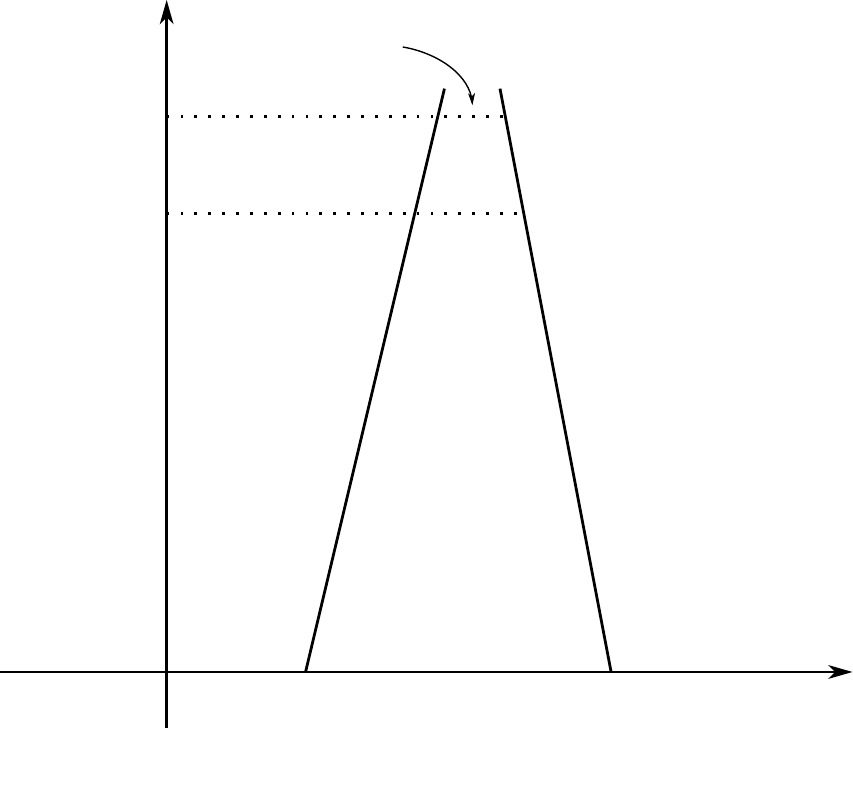
     \label{F_opening}}
     \caption{SBV regularity of $f'\circ u$.}
   \end{figure}

\section{Examples}
\label{S_examples}

In this section we present four examples: the first one shows that Theorem
\ref{T_Cheng} does not hold in general removing the assumption of polynomial degeneracy at the inflection points of the flux.
The second example is about the regularity in the kinetic formulation, the third one concerns fractional $\BV$ spaces and the last example
shows the sharpness of Theorem \ref{T_frac_regularity}.

\subsection{Polynomial degeneracy condition is needed in Theorem \ref{T_Cheng}}
In \cite{BM_structure} is provided an example of entropy solution to \eqref{E_cl} such that $f'\circ u\notin \BV_\loc(\R^+\times \R)$.
The flux $f$ in that example is weakly genuinely nonlinear and it has countably many inflection points.
Adapting the same idea, we provide here an example with the same property and such that $f$ has only one inflection point: in view of
Theorem \ref{T_Cheng}, that inflection point has not polynomial degeneracy.

\subsubsection{Building block}
For every $n\in\N$ let $g_n:[-1,1]\rightarrow \R$ be odd and such that
\begin{equation*}
g_n(x)= 
\begin{cases}
0 & \mbox{if }x \in (-1,-a_{n-1}), \\
\e_n & \mbox{if }x\in (-a_{n-1},-2a_n), \\
b_n & \mbox{if }x\in (-2a_n,-a_n), \\
0 & \mbox{if }x\in (-a_n,0)
\end{cases}
\end{equation*}
with 
\begin{equation*}
a_1<\frac{1}{2}, \quad a_n<\frac{a_{n-1}}{2}, \quad \sum_n \e_n<1, \quad \sum_nb_n < 1
\end{equation*}
and let $f:[-1,1]\rightarrow \R$ the unique continuous function for which for $\mathcal L^1$-almost every $x\in [-1,1]$
\begin{equation}\label{E_f''}
f''(x)=\sum_{n=1}^\infty g_n(x)
\end{equation}
with $f(0)=0$ and $f'(-1)=0$.
We consider the solution $u^n$ with initial datum
\begin{equation*}
u_0^n(x)=
\begin{cases}
-a_n & \mbox{if }x<0, \\
a_n & \mbox{if }x\in (0,d_n), \\
-a_n & \mbox{if }x>d_n,
\end{cases}
\end{equation*}
where $d_n>0$ will be chosen.

\begin{figure}
\centering
\def\svgwidth{0.7\columnwidth}
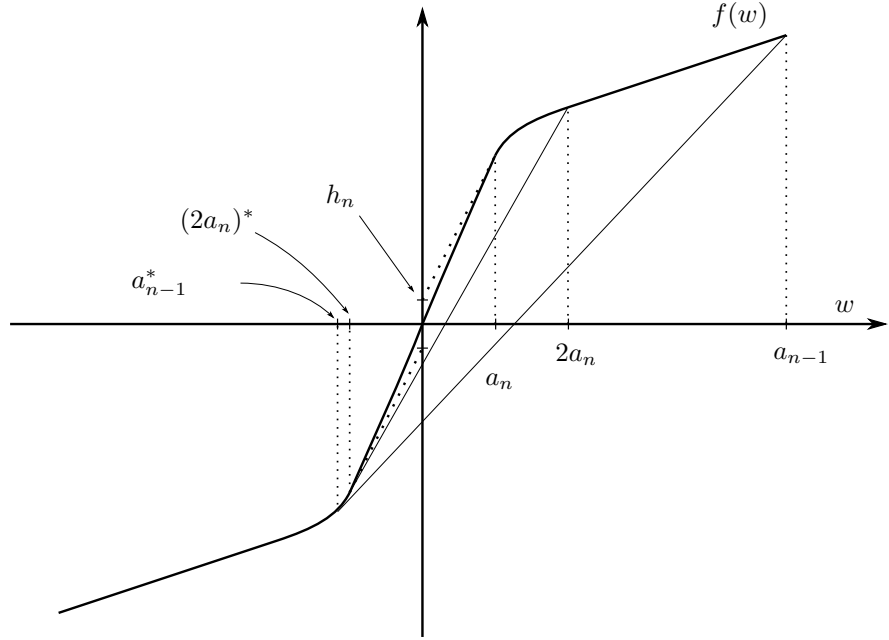
\caption{The flux $f$ in the interval $[-a_{n-1},a_{n-1}]$.}
\label{F_flux_an}
\end{figure}

\begin{figure}
\centering
\def\svgwidth{0.8\columnwidth}
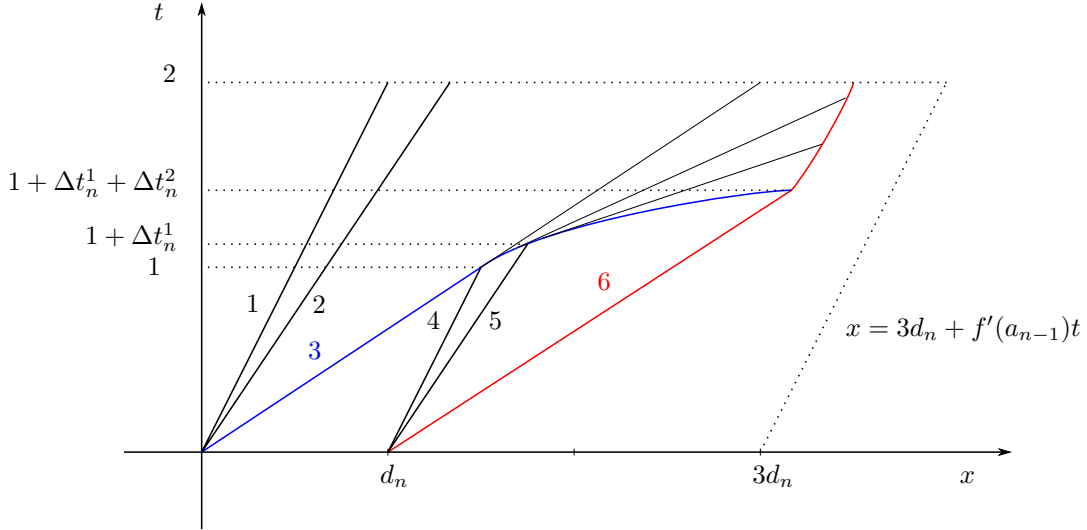
\caption{The solution $u$ for $t\in (0,2)$.}
\label{F_solution_basic}
\end{figure}

The parameters $\e_n,a_n,b_n$ will be chosen in particular in such a way that 
\begin{equation}\label{E_cond_example_star}
a_n<(-2a_n)^*<-a_{n-1}^*<2a_n,
\end{equation}
where $(-2a_n)^*$ denotes the conjugate point of $-2a_n$ defined in Lemma \ref{L_conjugate}.
We assume it at the moment and we describe the entropy solution (see Figure \ref{F_flux_an} and \ref{F_solution_basic}):
 for small $t>0$ the solution is obtained solving the two Riemann problems
at $x=0$ and $x=d_n$. Being $f$ odd, it suffices to discuss the Riemann problem at $x=0$.
The solution has a strict rarefaction between the curves 1 and 2 with values $-a_{n-1}$ and $-2a_n$ respectively, then another rarefaction
between the values $-2a_n$ and $a_{n-1}^*$ and finally
a left-contact discontinuity 3 that travels with speed $f'(a_{n-1}^*)$.
We set
\begin{equation}\label{E_dn}
d_n= f'(a_{n-1}^*)-f'(a_{n-1})
\end{equation}
so that the left-contact discontinuity 3 interacts with the characteristic 4 of therarefaction starting from $x=d_n$ at time $t=1$.
Then the left-contact discontinuity cancels the rarefaction and increases its speed. In particular it interacts with the characteristic 5 with value 
$a_{2n}$ at time $t + \Delta t^1_n$ with
\begin{equation}\label{E_delta1}
\Delta t^1_n\le \frac{f'(2a_n)-f'(a_{n-1})}{f'(a_{n-1}^*)-f'(2a_n)},
\end{equation}
indeed $f'(2a_n)-f'(a_{n-1})$ is the distance of the two curves at time $t=1$ and $f'(a_{n-1}^*)-f'(2a_n)$ is smaller than the difference 
of their speeds.
After time $1+\Delta t^1_n$ the left contact discontinuity moves with speed bigger than $f'(2a_n^*)$. Moreover by convexity of the curve 3, 
the distance between the curves 3 and 6 at time $1+ \Delta t^1_n$ is less than $d_n$. Therefore, recalling \eqref{E_dn},
 curve 3 interacts with curve 6 at time $1+\Delta t^1_n +\Delta t^2_n$ with 
\begin{equation}\label{E_delta2}
\Delta t^2_n \le \frac{f'(a_{n-1}^*)-f'(a_{n-1})}{f'((2a_n)^*)-f'(a_{n-1}^*)}.
\end{equation}
Finally observe that the speed of curve 6 decreases after the collision with curve 3, in particular
\begin{equation}\label{E_supp_example}
u(t,x)=-a_{n-1} \quad \text{for every }(t,x) \in\{ (t,x)\in (0,2)\times \R: x<f'(a_{n-1})t\mbox{ or }f'(a_{n-1})t+3d_n<x)\}.
\end{equation}

Now we estimate $\TV (f'\circ u^n(t))$ for $t\in (1+\Delta t^1_n+\Delta t^2_n,2)$: given $t$ as before, consider the characteristic $\X(\cdot, y_t)$
entering in curve 6 from the left at time $t$. By monotonicity of the flow, the distance at time 1 between the characteristic and curve 6 is at least
$d_n$. Moreover, since the speed of curve 6 for every $t\in (0,2)$ is bigger than $f'(a_{n-1})$ and since the characteristic is convex,
the speed $v_{\max}(t)$ of the characteristic at time $t$ is such that 
\begin{equation*}
v_{\max}(t)-f'(a_{n-1})\ge \frac{d_n}{t-1}.
\end{equation*}
Therefore, if we denote by 
\begin{equation*}
A_n:=\{(t,x)\in (1+\Delta^1_n +\Delta^2_n,2) \times \R: f'(a_{n-1})t<x< f'(a_{n-1})t+3d_n\},
\end{equation*}
it holds
\begin{equation*}
|D_x (f'\circ u^n)|(A_n) \ge \int_{(1+\Delta^1_n +\Delta^2_n)}^2 \left(v_{\max}(t)-f'(a_{n-1})\right)dt \ge d_n 
\log\left( \frac{1}{\Delta t^1_n+\Delta t^2_n}\right).
\end{equation*}
This additional logarithm allows to conclude the example after choosing in an appropriate way the parameters $a_n,\e_n,b_n$.

\subsubsection{General example}
In order to build the general counterexample we consider an initial datum of the following form:
\begin{equation*}
u_0=-\chi_{(-3\|f'\|_\infty,0)}+ \sum_{n=1}^\infty\sum_{i=1}^{N_n}\left(a_{n-1} \chi_{[x^n_i,x^n_i +d_n]} - a_{n-1} \chi_{(x^n_i +d_n,x^n_i+3d_n)}\right),
\end{equation*}
where $x^n_i$ is defined inductively by
\begin{equation*}
\left\{
\begin{aligned}
x^1_1= 0, &  \\
x^n_{i+1} = x^n_i+3d_n & \quad \mbox{for }n\ge1,\, i=1,\ldots, N_n-1, \\
x^{n+1}_1=x^n_{N_n}+3d_n & \quad \mbox{for }n\ge 1.
\end{aligned}
\right.
\end{equation*}
For every $n\ge 1$ and $i=1,\ldots, N_n-1$, denote by
\begin{equation*}
S^n_i:=\{(t,x)\in (0,2)\times \R: x_i^n+f'(a_{n-1})t<x<x_i^n+f'(a_{n-1})t + 3d_n\}.
\end{equation*}
By \eqref{E_supp_example}, for every $(t,x)\in S^n_i$,
\begin{equation*}
u(t,x)=u^n(t, x-x^n_i),
\end{equation*}
where $u^n$ is the solution described in the previous step.
Therefore 
\begin{equation*}
|D_x(f'\circ u)|((1,2)\times \R) \ge \sum_{n=1}^\infty N_nd_n\log\left( \frac{1}{\Delta t^1_n+\Delta t^2_n}\right).
\end{equation*}
In order to have $u_0$ with bounded support, we need
\begin{equation}\label{E_space}
\sum_{n=1}^\infty N_nd_n<+\infty
\end{equation}
and finally, choosing $\e_n,b_n\le a_n^n$ we have that $f^{(p)}(0)=0$ for every $p\ge 2$.
Therefore we conclude by proving that there exists $\e_n,a_n,b_n>0$ such that 
\begin{equation}\label{E_goal_example}
\begin{split}
\e_n,b_n\le a_n^n, & \qquad a_n<(-2a_n)^*<a_{n-1}^*<2a_n, \\
\sum_{n=1}^\infty N_nd_n<+\infty, &\qquad  \sum_{n=1}^\infty N_nd_n\log\left( \frac{1}{\Delta t^1_n+\Delta t^2_n}\right)=+\infty,
\end{split}
\end{equation}
where we recall that $d_n$ is defined by \eqref{E_dn}.
In particular we need to estimate from above $\Delta t^1_n$ and $\Delta t^2_n$.
By \eqref{E_delta1}, \eqref{E_delta2} and \eqref{E_f''},
\begin{equation}\label{E_delta12}
\Delta t^1_n\le \frac{\e_n(a_{n-1}-2a_n)}{b_n(2a_n-a_{n-1}^*)} \qquad \mbox{and}\qquad 
\Delta t^2_n \le \frac{b_n(2a_n-a_{n-1}^*)+\e_n(a_{n-1}-2a_n)}{b_n(a_{n-1}^*-(2a_n)^*)}.
\end{equation}
We estimate now $(2a_n)^*$ and $a_{n-1}^*$.
Imposing $(2a_n)^*=(1+\alpha_n)a_n$ for some $\alpha_n\in (0,1)$ we get by definition of $(2a_n)^*$,
\begin{equation}\label{E_2an*}
f'((1+\alpha)a_n))((3+\alpha)a_n) = f((1+\alpha)a_n)+f(2a_n).
\end{equation}
Let $h_n:= f(a_n)-f'(a_n)a_n$; by elementary computations
\begin{equation*}
h_n= \sum_{i=n+1}^\infty \Delta_i, \qquad \mbox{where }\Delta_i:= \frac{\e_i}{2}(a_{i-1}^2-4a_i^2) + \frac{3}{2}a_i^2b_i.
\end{equation*}
Using this notation \eqref{E_2an*} is equivalent to
\begin{equation*}
(\alpha_n^2+6\alpha_n -1) a_n^2b_n + 2h_n=0.
\end{equation*}
If we denote by $\alpha$ the positive root of $\alpha^2+6\alpha -1=0$, i.e. $\alpha = \sqrt{10} -3$, we have that 
\begin{equation}\label{E_alpha_n}
\alpha_n=\alpha + r^1\left(\frac{h_n}{a_n^2b_n}\right),
\end{equation}
and $r^1(s)\to 0$ as $s\to 0$.
Similarly we impose $a_{n-1}^*=(2-\beta_n)a_n$ and we get
\begin{equation*}
\left[\frac{\beta_n^2}{2}-(2+R_n)\beta_n +1 \right] a_n^2b_n = 2h_n - \e_n\frac{(a_{n-1}-2a_n)^2}{2},
\end{equation*}
where $R_n= \frac{a_{n-1}}{a_n}$. Therefore
\begin{equation}\label{E_beta_n}
\beta_n= \frac{a_n}{a_{n-1}}+ r^3_n \left(\frac{a_n}{a_{n-1}}\right) + r^2\left(\frac{\e_n+h_n}{a_n^2b_n}\right),
\end{equation}
with $r^2(s)\to 0$ as $s \to 0$ and $r^3_n(s)=O(s^2)$ as $s\to 0$.

Let us take now $b_n=a_n^n$ for every $n\ge 1$. Therefore if 
\begin{equation*}
a_n<\frac{a_{n-1}}{3} \quad \mbox{and} \quad \e_{n+1}<\e_n,
\end{equation*}
then $\Delta_n<\frac{\Delta_{n-1}}{2}$, so that $h_n < 2\Delta_{n+1}$.
Therefore \eqref{E_alpha_n} reduces to
\begin{equation}\label{E_alpha}
\alpha_n= \alpha + \tilde r^1_n\left(\frac{\e_{n+1}}{a_n^{n+2}}\right) + \tilde r^2_n\left(\frac{a_{n+1}^{n+3}}{a_n^{n+2}}\right)
\end{equation}
and \eqref{E_beta_n} reduces to
\begin{equation}\label{E_beta}
\beta_n=  \frac{a_n}{a_{n-1}}+ r^3\left(\frac{a_n}{a_{n-1}}\right) + \tilde r^4_n\left(\frac{\e_n}{a_n^{n+2}}\right) + \tilde r^5_n\left(\frac{a_{n+1}^{n+3}}{a_n^{n+2}}\right).
\end{equation}

Fix $\e'\in (0,\alpha/2)$. We choose the parameters $a_n$ and $\e_n$.
By definition $a_0=1$; let $a_1\in (0,1/3)$ such that $|\tilde r^3(a_1)|< \e'a_1/3$. The existence is granted by the fact that $r^3(s)=O(s^2)$ as 
$s\to 0$. 
Moreover let $\e_1\in (0,1/2)$ be such that
\begin{equation*}
\tilde r^1_1 \left(\frac{\e_1}{a_1^3}\right)<\frac{\e'}{2} \qquad \mbox{and} \qquad \tilde r^4_1 \left(\frac{\e_1}{a_1^3}\right)<\frac{\e'a_1}{3}.
\end{equation*}
Inductively, since for every $n\ge 1$ the remainders $\tilde r^1_n,\tilde r^2_n, \tilde r^4_n, \tilde r^5_n$ are infinitesimal at 0 and 
$\tilde r^3_n(s)= O(s^2)$ as $s\to 0$, it is possible to choose $a_n$ and $\e_n$ (for every $n$ first choose $a_n$ then $\e_n$) such that for every $n\ge 1$,
\begin{enumerate}
\item
\begin{equation*}
\sum_{n=1}^\infty\e_n<1, \qquad \sum_{n=1}^\infty a_n^n<1,\qquad  \e_n<a_n^n,
\end{equation*}
\item
\begin{equation*}
\left|\tilde r^1_n\left(\frac{\e_{n+1}}{a_n^{n+2}}\right)\right|< \frac{\e'}{2}, \qquad \left|\tilde r^2_n\left(\frac{a_{n+1}^{n+3}}{a_n^{n+2}}\right)\right| < \frac{\e'}{2};
\end{equation*}
\item
\begin{equation*}
\left|\tilde r^3_n\left(\frac{a_n}{a_{n-1}}\right)\right|<\e' \frac{a_n}{3a_{n-1}}, \quad \left|\tilde r^4_n \left(\frac{\e_n}{a_n^{n+2}}\right)\right|<\e'\frac{a_n}{3a_{n-1}},
\quad \left|\tilde r^5_n\left(\frac{a_{n+1}^{n+3}}{a_n^{n+2}}\right)\right|< \e'\frac{a_n}{3a_{n-1}};
\end{equation*}
\item
\begin{equation}\label{E_log1}
\log\left(\frac{a_n}{a_{n+1}}\right)> n;
\end{equation}
\item
\begin{equation}\label{E_log2}
\log\left(\frac{a_n^{n+2}}{\e_n}\right)>n;
\end{equation}
\end{enumerate}
Conditions (1), (2) and (3) implies in particular that all the assumptions we made in the previous parts (in particular \eqref{E_cond_example_star})
are satisfied. Condition (4) and (5) will be useful in a moment. 

With this choice, by \eqref{E_delta12} we have that 
\begin{equation*}
\Delta t^1_n \le \frac{\e_n}{a_n^{n+1}\beta_n} \le \frac{\e_n}{a_n^{n+2}(1-\e')},
\end{equation*}
and
\begin{equation*}
\Delta t^2_n \le \frac{\beta_na_n^{n+1}+\e_n}{a_n^{n+1}(1-\alpha_n-\beta_n)}\le \frac{a_{n-1}}{ca_n}+\frac{\e_n}{ca_n^{n+1}},
\end{equation*}
where $c>0$ is a constant such that $1-\alpha_n-\beta_n>c$. Such a constant exists by \eqref{E_alpha}, \eqref{E_beta} and the choice of the 
parameters.
Therefore by \eqref{E_log1} and \eqref{E_log2}, there exists $\tilde c>0$ such that
\begin{equation}\label{E_log_final}
\log\left(\frac{1}{\Delta t^1_n+\Delta t^2_n}\right) \ge \tilde c n.
\end{equation}
Choosing finally
\begin{equation*}
N_n= \left\lfloor \frac{1}{n^2d_n} \right\rfloor
\end{equation*}
we have that 
\begin{equation*}
\sum_{n=1}^\infty N_nd_n \approx \sum_{n=1}^\infty \frac{1}{n^2}<+\infty
\end{equation*}
and by \eqref{E_log_final},
\begin{equation*}
\sum_{n=1}^\infty N_nd_n\log\left( \frac{1}{\Delta t^1_n+\Delta t^2_n}\right) \gtrsim \sum_{n=1}^{\infty}\frac{1}{n}=+\infty.
\end{equation*}
This concludes the analysis of this example.

\begin{figure}
\centering
\def\svgwidth{\columnwidth}
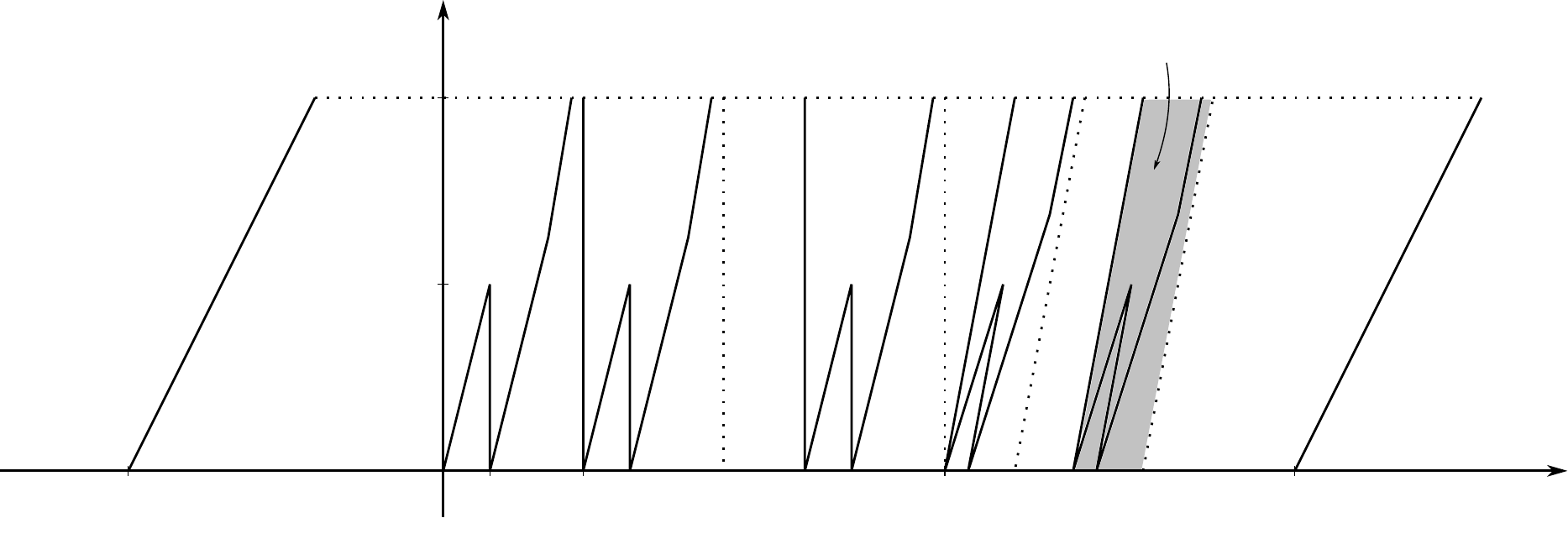
\caption{The solution $u$ for $t\in (0,2)$.}
\end{figure}

\subsection{$\partial_w\mu$ is not a measure}
We start by briefly recalling the kinetic formulation of \eqref{E_cl} (see \cite{LPT_kinetic}): $u\in L^\infty(\R^+\times \R)$ is an entropy solution of \eqref{E_cl} 
if and only if 
\begin{equation*}
\partial_t\chi_{\{u>w\}}+f'(w)\partial_x\chi_{\{u>w\}}=\partial_w\mu,
\end{equation*}
where $\mu\in \mathcal M(\R\times\R^+\times\R)$ can be obtained as
\begin{equation*}
\mu=\mathcal L^1\otimes \mu_k^+
\end{equation*}
with $\mu_k^+$ the dissipation of the entropy $\eta_k^+(w)=(w-k)^+$.

In \cite{DLR_dissipation} it is proved that Theorem \ref{T_Cheng} implies that there exists a constant $C>0$ such that
\begin{equation*}
\|\partial^2_w\mu\|_{\mathcal M} \le C.
\end{equation*}
Then this has been used to get a refined averaging lemma and finally to deduce the rectifiability of the entropy dissipation measure.

The following example shows that there exists a degenerate flux $f$ such that even the first derivative $\partial_w\mu$ can not be 
represented as a Radon measure.

\subsubsection{Building block}
Consider a flux as in Figure \ref{F_measure_flux}.

Now we consider the initial datum
\begin{equation*}
u_0=3L\chi_{[0,A]}.
\end{equation*}
The solution has a shock starting from 0 moving with velocity 0 between the values $0$ and $3L$ that does not interact with anything for
$t \in [0, AL/h]$. In particular we choose
\begin{equation*}
A=\frac{h}{L} \qquad \mbox{so that} \qquad \supp u(t)\subset [0,2A] \quad \mbox{and} \quad u(t,0-)=0, \quad  u(t,0+)=3L
\end{equation*}
for every $t\in [0,1]$.
We compute $|\partial_w\mu|$ along the shock at $x=0$:
by standard computations,
\begin{equation*}
\mu_k^+\llcorner (\{0\}\times [0,1]) =( f(3L)-f(k))\mathcal H^1\llcorner (\{0\}\times [0,1]) =-f(k)\mathcal H^1\llcorner (\{0\}\times [0,1]),
\end{equation*}
therefore
\begin{equation*}
|\partial_w\mu| (\{0\}\times [0,1]) = \int_L^{2L}|f'(w)|dw= 2\frac{hL}{a}.
\end{equation*}

\begin{figure}
\centering
\def\svgwidth{0.7\columnwidth}
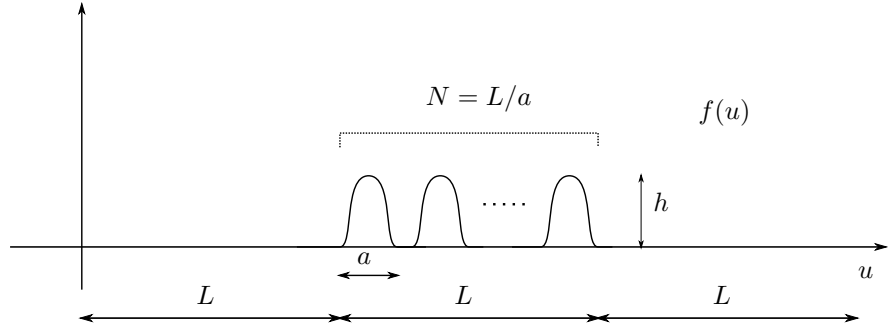
\caption{The flux $f$ for the basic step of the construction.}
\label{F_measure_flux}
\end{figure}

\subsubsection{General example}
The flux is obtained repeating the construction above with smaller and smaller parameters and the initial datum is obtained placing 
side by side $N_n$ multiples of characteristic functions at step $n$. See Figure \ref{F_flux_full} and Figure \ref{F_initial}.

\begin{figure}
\centering
\def\svgwidth{0.7\columnwidth}
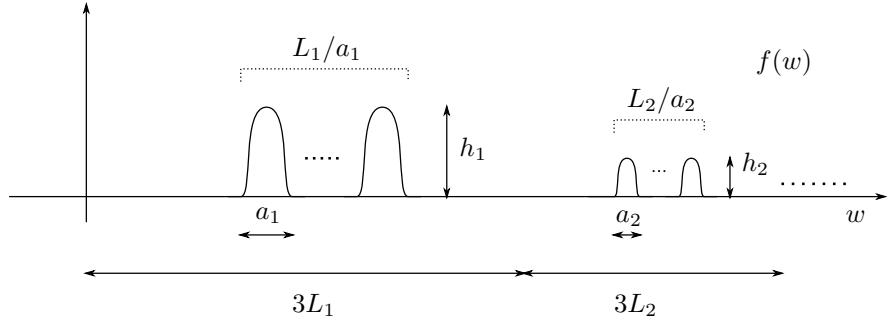
\caption{The flux $f$ for the general example.}
\label{F_flux_full}
\end{figure}

\begin{figure}
\centering
\def\svgwidth{0.7\columnwidth}
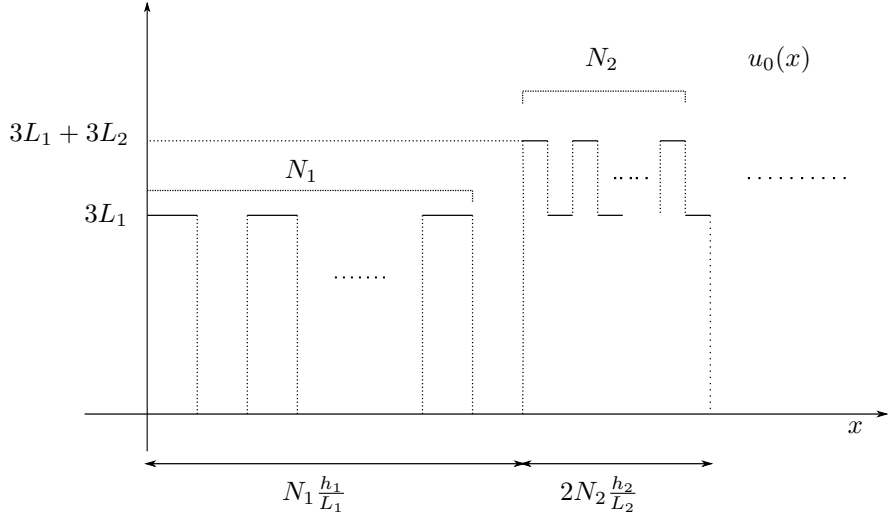
\caption{The initial datum $u_0$ for the general example.}
\label{F_initial}
\end{figure}

At step $n$ choose $h_n=a_n^n$ in order to have a $C^\infty$ flux.
In order to have an $L^\infty$ initial datum we need
\begin{equation*}
\sum_nL_n<\infty.
\end{equation*}
In order to have the initial datum with bounded support it suffices to have
\begin{equation*}
\sum N_n\frac{a_n^n}{L_n} <\infty
\end{equation*}
and finally the distribution $\partial_w\mu$ is not a Radon measure if 
\begin{equation*}
\sum_nN_na_n^{n-1}L_n=\infty.
\end{equation*}
A possible choice is
\begin{equation*}
L_n=2^{-n},\qquad a_n=8^{-n}, \qquad N_n=\frac{8^{(n^2)}}{4^n}.
\end{equation*}

\subsection{Positive and negative fractional total variation}
In this section we provide an example that proves the following Proposition.

\begin{proposition}\label{P_bv+-}
For every $p>1$ there exists a function $u:[0,1]\rightarrow [0,1]$ such that
\begin{equation*}
\left(\TV^{1/ p}_+ u\right)^p:= \sup_{\mathcal P([0,1])}\sum_{i=1}^{k-1}\left[(u(x_{i+1})-u(x_i))^+\right]^p=1
\end{equation*}
and 
\begin{equation*}
\left(\TV^{1/p}_- u\right)^p:= \sup_{\mathcal P([0,1])}\sum_{i=1}^{k-1}\left[(u(x_{i+1})-u(x_i))^-\right]^p=+\infty.
\end{equation*}
\end{proposition}

\begin{remark}
As we already mentioned, the conclusion of the proposition above cannot hold for $p=1$.
In fact the trivial relation holds:
\begin{equation}\label{E_tv+-}
\TV_+(u) = \TV_-(u) + u(1)-u(0).
\end{equation}
This proves that a bounded function with finite positive total variation is a function of bounded variation. 
In particular if a bounded function $u:\R\rightarrow \R$ is one-sided Lipschitz, then it has locally finite total variation.
By the Oleinik estimate, this argument applies to entropy solutions to \eqref{E_cl} with $f$ uniformly convex.
In the more general case of strictly convex fluxes with polynomial degeneracy, it holds a one-sided H\"{o}lder estimate.
As observed in \cite{CJJ_fractional}, an analogous of \eqref{E_tv+-} for $\TV^{1/p}$ would allow to 
apply the same argument to get fractional BV regularity of the solution $u$. 
The example shows that this cannot be done, however as in \cite{CJJ_fractional} for the convex case and here in Section \ref{S_Cheng}, it is
enough to rely on the $\BV$ regularity of $f'\circ u$.
\end{remark}

Let $n\ge 0$ and $C_n$ be the $n$-th step in the construction of the Cantor set:
\begin{equation*}
C_n=\left\{x: x=\sum_{i=1}^\infty \frac{x_i}{3^i} \mbox{ with }x_i\in \{0,2\} \,\forall i=1,\ldots,n \mbox{ and } x_i\in \{0,1,2\} \,\forall k>n\right\}. 
\end{equation*}
Let $u_0=\Id\llcorner [0,1]$ and fix $\alpha=2^{\frac{p-1}{p}} -1<1$.
Define by induction $u_n$ for $n\ge 1$:
\begin{equation*}
u_n=
\begin{cases}
u_{n-1} & \mbox{in }[0,1]\setminus C_{n-1} \\
v_n &  \mbox{in }C_{n-1}
\end{cases}
\end{equation*}
where $v_n$ is defined on each connected component $[a,b]$ of $C_{n-1}$ as the piecewise affine interpolation between the points
\begin{equation*}
\begin{split}
&(a,u_{n-1}(a)), \quad \left(a+\frac{b-a}{3}, u_{n-1}(a) + \frac{1+\alpha}{2}(u_{n-1}(b)-u_{n-1}(a))\right) \\
&\left(a+\frac{2}{3}(b-a), u_{n-1}(a) + \frac{1-\alpha}{2}(u_{n-1}(b)-u_{n-1}(a))\right), \quad (b, u_{n-1}(b)).
\end{split}
\end{equation*}
See Figure \ref{F_BV+-} for the first two steps of the construction.

Observe that $\alpha$ has been chosen in such a way that
\begin{equation*}
\TV^{1/p}_+ u_1= \TV^{1/p}_+  u_0 =1.
\end{equation*}
Let $n\ge 1$ and let $[a,b]$ be a connected component of $C_n$. A straightforward computation leads to
\begin{equation*}
\begin{split}
\|u_n-u_{n-1}\|_\infty = &~ u_n\left(\frac{2a}{3}+\frac{b}{3}\right)- u_{n-1}\left(\frac{2a}{3}+\frac{b}{3}\right) \\
= &~ \left( \frac{1+\alpha}{2}-\frac{1}{3}\right)(u_{n-1}(b)-u_{n-1}(a)) \\
= &~ \left( \frac{1+\alpha}{2}-\frac{1}{3}\right)\left(\frac{1+\alpha}{2}\right)^{n-1}.
\end{split}
\end{equation*}

Since $\alpha<1$ this implies that the sequence $u_n$ converges uniformly to a continuous function $u$.

We estimate from below the negative $1/p$-variation of $u$ by considering $(\bar x_1,\ldots,\bar x_{2^{n+1}})\in \mathcal P$ where
$\{\bar x_1,\ldots,\bar x_{2^{n+1}}\} = \partial C_n$:
\begin{equation*}
\sum_{i=1}^{2^{n+1}-1}[(u(x_{i+1})-u(x_i))^-]^p=\sum_{i=1}^{2^{n+1}-1}[(u_n(x_{i+1})-u_n(x_i))^-]^p=
\sum_{j=1}^{n}2^j\alpha^p\left(\frac{1+\alpha}{2}\right)^{jp} = n\alpha^p.
\end{equation*}
In particular the negative $1/p$-variation of $u$ is not finite.

Now we prove that the positive $1/p$-variation is equal to 1.
It is sufficient to prove
\begin{equation*}
\TV^{1/p}_+ u_n \le 1
\end{equation*}
for every $n$ by lower semicontinuity of $\TV^{1/p}_+ $ with respect to pointwise convergence.

Since $u_n$ is piecewise monotone and $\phi(u)=|u|^p$ is convex it is easy to show that there exits $(x_1,\ldots,x_{2k})\in \mathcal P$
such that
\begin{equation}\label{E_optimal}
\TV^{1/p}_+ u_n=\sum_{i=1}^{2k-1}[(u_n(x_{i+1})-u_n(x_i))^+]^p= \sum_{i=1}^k[u_n(x_{2i})-u_n(x_{2i-1})]^p
\end{equation}
and for every $i=1,\ldots,2k$
\begin{equation*}
x_i=\bar x_{\sigma(i)},
\end{equation*}
where $\sigma:[1,2k]\cap \N\rightarrow [1,2^{n+1}]\cap \N$ is strictly increasing.
By \eqref{E_optimal}, it immediately follows that $\sigma(1)=1$, $\sigma(2k)=2^{n+1}$ and $\sigma$ maps even indexes into even
indexes and odd indexes into odd indexes.

We are going to prove that given $(x_1,\ldots,x_{2k})\in \mathcal P$ such that \eqref{E_optimal} holds with $k>1$ there exists
$(y_1,\ldots, y_{2k-2})\in \mathcal P$ which realizes the $\TV^{1/p}_+ u_n$ too.
$(y_1,\ldots, y_{2k-2})$ is obtained eliminating two consecutive points in $(x_1,\ldots,x_{2k})$. 

We need also the following property that follows by the optimality of the partition: given $(x_1,\ldots,x_{2k})\in \mathcal P$ such that \eqref{E_optimal} holds, for every 
$j=1,\ldots,k-1$
\begin{equation}\label{E_successore}
\sigma(2j)=2l\quad \Longrightarrow \quad \sigma(2j+1)=2l+1.
\end{equation}
Let $\bar j\in[1,k-1]\cap \N$ such that
\begin{equation}\label{E_minimality}
u_n(x_{2\bar j})-u_n(x_{2\bar j +1})=\min_{j\in [1,\ldots k-1]} u_n(x_{2 j})-u_n(x_{2j +1}).
\end{equation}
\textbf{Claim.} If $(x_1,\ldots,x_{2k})\in \mathcal P$ is optimal, then $(x_1,\ldots, x_{2k})\setminus (x_{2\bar j},x_{2\bar j+1})$ is
still optimal. 

First we observe that iterating this argument $k-1$ times we get that
\eqref{E_optimal} holds for $(0,1)\in \mathcal P$ so that
\begin{equation*}
\TV_s^+u_n \le 1
\end{equation*}
and this reduces the proof of Proposition \ref{P_bv+-} to the proof of the claim.

The claim is a consequence of the convexity of $\phi(u)=|u|^p$, which is exploited in the following lemma.
\begin{lemma}
Let $w<z$ and
\begin{equation*}
u_1\le w, \quad u_2\ge z, \quad v_1=w+\frac{1+\alpha}{2}(z-w), \quad v_2=w+\frac{1-\alpha}{2}(z-w).
\end{equation*}
Then 
\begin{equation*}
(v_1-u_1)^p+(u_2-v_2)^p\le (u_2-u_1)^p.
\end{equation*}
\end{lemma}
\begin{proof}
By elementary computations,
\begin{equation*}
\begin{split}
(v_1-u_1)^p+(u_2-v_2)^p =&~\int_0^{v_1-u_1}pt^{p-1}dt+\int_0^{u_2-v_2}pt^{p-1}dt \\
 =&~ \int_0^{v_1-w}pt^{p-1}dt+\int_{v_1-w}^{v_1-u_1}pt^{p-1}dt + \int_0^{z-v_2}pt^{p-1}dt + \int_{z-v_2}^{u_2-v_2}pt^{p-1}dt.
\end{split}
\end{equation*}
Since $(v_1-w)^p+(z-v_2)^p= (z-w)^p$,
\begin{equation*}
\begin{split}
(v_1-u_1)^p+(u_2-v_2)^p =&~(z-w)^p + \int_{v_1-w}^{v_1-u_1}pt^{p-1}dt +  \int_{z-v_2}^{u_2-v_2}pt^{p-1}dt \\
\le &~ (z-w)^p + \int_{z-w}^{u_2-u_1} pt^{p-1}dt \\
= &~ (u_2-u_1)^p,
\end{split}
\end{equation*}
where the inequality holds since $pt^{p-1}$ is increasing with respect to $t$.
\end{proof}

We want to apply the previous lemma with
\begin{equation*}
u_1=u(x_{2\bar j -1}), \quad u_2=u(x_{2\bar j +2}), \quad v_1=u(x_{2\bar j}), \quad v_2=u(x_{2\bar j+1})
\end{equation*}
and
\begin{equation*}
w= u(x_{2\bar j}-(x_{2\bar j +1}- x_{2\bar j})), \quad  z= u(x_{2\bar j+1}+(x_{2\bar j +1}- x_{2\bar j})).
\end{equation*}
The two equalities
\begin{equation*}
v_1=w+\frac{1+\alpha}{2}(z-w) \quad \mbox{and} \quad v_2=w+\frac{1-\alpha}{2}(z-w)
\end{equation*}
hold by construction. Therefore it remains to check that $u_1\le w$ and $u_2\ge z$. Since they are similar we prove only the first inequality: 
by the minimality in \eqref{E_minimality} it follows that $x_{2\bar j-1}\le x_{2\bar j}-(x_{2\bar j +1}- x_{2\bar j})$ and
therefore by optimality in \eqref{E_optimal} it follows that
\begin{equation*}
u_1=u(x_{2\bar j -1})\le u(x_{2\bar j}-(x_{2\bar j +1}- x_{2\bar j}))= w.
\end{equation*}

Hence we can apply the lemma and this implies the claim, therefore the proof of Proposition \ref{P_bv+-} is complete.

\begin{figure}
\centering
\def\svgwidth{0.7\columnwidth}
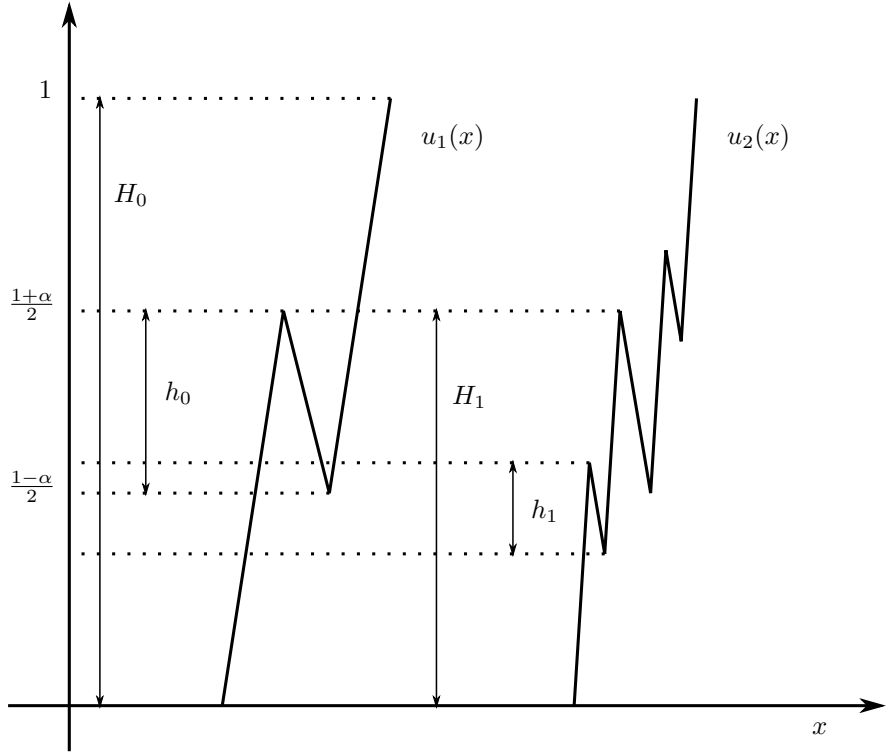
\caption{The first two steps of the construction of the function $u$ in Proposition \ref{P_bv+-}.}\label{F_BV+-}
\end{figure}

\subsection{Theorem \ref{T_frac_regularity} is sharp}
We show here for completeness, the already known sharpness of Theorem \ref{T_frac_regularity}: see
\cite{CJ_oscillating} for a similar construction and \cite{DLW_averaging}, where in particular the optimality is shown in the setting of fractional
Sobolev spaces.

Let $p\in \N$ and consider the flux $f(u)=u^{p+1}$ of degeneracy $p$.
We provide a bounded initial datum $u_0$ with compact support such that for every $q\in [1,p)$ the entropy solution $u$ at time 1 does not 
belong to $\BV^{1/q}(\R)$.

Consider first the entropy solution of \eqref{E_cl} with $f(u)=u^{p+1}$ and $u_0= a \chi_{[0,L]}$ for some $a,L>0$.
The solution for small $t>0$ is given by a rarefaction starting from $x=0$ and a shock starting from $x=L$. The maximal speed of the rarefaction
is $f'(a)=(p+1)a^p$ and the the velocity $\lambda$ of the shock is given by the Rankine-Hugoniot condition:
\begin{equation*}
\lambda = \frac{f(a)-f(0)}{a-0}= a^p.
\end{equation*}
Therefore
\begin{equation}\label{E_est_basic}
t<\frac{L}{pa^p} \qquad \Longrightarrow \qquad \max u(t)=a \quad \mbox{and} \quad \supp u(t)\subset [0,L+ta^p].
\end{equation}

Consider now an initial datum of the form
\begin{equation*}
u_0=\sum_{n=1}^\infty a_n \chi_{[x_n,x_n+L_n]}.
\end{equation*}
Choose
\begin{equation*}
L_n= (p+1)a_n^p > p a_n^p, \qquad x_1=0, \quad x_n = x_{n-1} + L_n + a_n^p.
\end{equation*}
Let $q\ge 1$, by the choice above and \eqref{E_est_basic}, it holds
\begin{equation*}
\supp u_0 \subset \left[0, (p+2)\sum_{n=1}^\infty a_n^p\right], \qquad 
\left(\TV^{1\over q}u(1)\right)^q = 2 \sum_{n=1}^\infty a_n^q.
\end{equation*}
Therefore in order to conclude the example it suffices to consider a nonnegative sequence $(a_n)_{n\in \N}\in \ell^p \setminus \ell^q$ for every $q<p$. For example let
\begin{equation*}
a_n=\left[\frac{1}{n[\log (1+ n)]^2}\right]^{\frac{1}{p}}.
\end{equation*}

\noindent
\textbf{Acknowledgement.} The author gratefully acknowledges Stefano Bianchini for several discussions and sugge\-stions, and for a careful reading of a preliminary version of this paper.

\end{document}